\documentclass[11pt,reqno]{amsart}

\usepackage[letterpaper,margin=1.15in]{geometry}
\usepackage{amsmath,amssymb,amsthm,mathtools,mathrsfs}
\usepackage{microtype}
\usepackage[colorlinks=true,linkcolor=blue,citecolor=blue,urlcolor=blue]{hyperref}
\usepackage[nameinlink,capitalise,noabbrev]{cleveref}
\usepackage{enumitem}
\usepackage{mlmodern}
\usepackage{fix-cm}
\usepackage{ulem}

\usepackage{xcolor}

\numberwithin{equation}{section}

\newtheorem{theorem}{Theorem}[section]
\newtheorem{proposition}[theorem]{Proposition}
\newtheorem{lemma}[theorem]{Lemma}
\newtheorem{corollary}[theorem]{Corollary}
\theoremstyle{definition}
\newtheorem{definition}[theorem]{Definition}

\theoremstyle{remark}
\newtheorem{remark}[theorem]{Remark}

\DeclareMathOperator{\divv}{div_{\textit{v}}}
\DeclareMathOperator{\supp}{supp}
\DeclareMathOperator{\sgn}{sgn}
\newcommand{\R}{\mathbb R}
\newcommand{\K}{\mathscr K}
\newcommand{\A}{\mathscr A}
\newcommand{\M}{\mathcal M}
\newcommand{\W}{\mathbb W}
\newcommand{\cC}{\mathcal C}

\newcommand{\cS}{\mathcal S}
\newcommand{\norm}[1]{\lVert#1\rVert}
\newcommand{\abs}[1]{\lvert#1\rvert}
\newcommand{\pair}[2]{\langle #1,#2\rangle}
\newcommand{\eps}{\varepsilon}

\title[Kinetic equations with singular Morrey-type drift]
{Kinetic SDEs and Kinetic Fokker-Planck equations with critical Morrey-type drifts, and local blow-ups}

\author{Jean-Fran\c{c}ois Jabir}

\address{ HSE University, Department of Statistics and Data Analysis and LSA, 
 Moscow, Russia.}
\email{jjabir@hse.ru}

\author{Damir Kinzebulatov}

\address{Universit\'{e} Laval, D\'{e}partement de math\'{e}matiques et de statistique, Qu\'{e}bec, QC, Canada}

\email{damir.kinzebulatov@mat.ulaval.ca}

\author{Chengjun Yue}

\address{Universit\'{e} Laval, D\'{e}partement de math\'{e}matiques et de statistique, Qu\'{e}bec, QC, Canada}

\email{chengjun.yue.1@ulaval.ca}

\thanks{The research of J.-F. J. is supported by the Russian Science Foundation project (project No. 24-11-00123). The research of D.K. is supported by NSERC Discovery grant (RGPIN-2024-04236). The research of C.\,Y.\,is supported by the CRM-Laval postdoctoral fellowship.}


\begin{document}

\begin{abstract}
We develop a Sobolev theory for the kinetic
Fokker-Planck equation subject to a singular Morrey-form $L^\infty+\M^q_{\K}$ force field which can produce local blow-ups due to strong attraction to a submanifold.
Our approach is based on the fractional representation for the Duhamel series of the equation and on establishing 
kinetic Adams-type estimates. Accordingly, we construct weak solutions of  Brownian-driven kinetic SDEs in a critical Morrey set-up,  obtain kinetic Krylov-type estimates and a strongly continuous Feller propagator, and establish 
uniqueness criteria for a large class of weak solutions.
\end{abstract}

\maketitle

\fontsize{10.4pt}{4.4mm}\selectfont

\setcounter{tocdepth}{1}

\tableofcontents

\section{Introduction}\label{sec:intro}

\subsection{Brief overview of the paper.}   Kinetic equations describe the statistical evolution of inertial particles subject to environmental noise. These equations 
arise naturally in Vlasov systems in statistical mechanics and the study of Langevin-type dynamics \cite{Langevin-08}. Notably, related SDEs generically correspond to stochastic Newtonian dynamics $$\ddot{X}_t=F_t(X_t,\dot{X}_t)+\sqrt{2}\dot{W}_t$$ which arise in models of molecular motion, electromagnetic phenomena, turbulence, social or economic multi-agent systems as well as in deep learning methods (we may refer to \cite{Bossy-05}, \cite{Serfaty-26}, \cite{ParTos-13}, \cite{ChChBaJo-18} and references therein).

The goal of this paper is to develop a Sobolev theory for the forward equation
\begin{equation}\label{eq:forward-intro}
 \lambda\eta(t,x,v)+\K\eta(t,x,v)-\divv(b\eta)(t,x,v)=f{(t,x,v)},
 \qquad(t,x,v)\text{ in }\R^{1+2d},\quad  d \geq 2,
\end{equation}
with the  kinetic
forward Kolmogorov operator
\begin{equation*}
 \K:=\partial_t+v\cdot\nabla_x-\Delta_v.
\end{equation*}
Here
the diffusion acts only in velocity; however,
transport couples velocity and position, and the commutator
identity
\[
 [\partial_{v_j},\partial_t+v\cdot\nabla_x]=\partial_{x_j}
\]
shows how velocity regularization is transmitted to the position
variables. 
The drift $b:\R^{1+2d}\rightarrow\R^d$ is Borel measurable and may have critical
singularities in time, position, and velocity, i.e.\,can introduce singular attractive forces and irregular alignment interactions. We then apply our Sobolev theory to develop weak  solution and Feller theory for the  
kinetic stochastic
system
\begin{equation}\label{eq:sde-intro}
\left\{
\begin{array}{l}
  dX_t=V_t\,dt,\\[2mm]
  dV_t=-b(t,X_t,V_t)\,dt+\sqrt2\,dW_t, \quad (X_0,V_0)=(x_0,v_0) \in \mathbb R^{2d}.
 \end{array}
 \right.
\end{equation}

Recall that the kinetic dilation related to $\K$ is $(t,x,v)\longmapsto (r^2t,r^3x,rv)$,
and so the
dimension of space-time is
\begin{equation}\label{eq:Q}
 Q:=2+3d+d=4d+2.
\end{equation}
The corresponding definition of the kinetic cylinder\footnote{$\cC_r(t,x,v)$ has measure comparable with $r^Q$.} is
\begin{equation}\label{eq:backward-cylinder}
	\begin{split}
		\cC_r(t,x,v):=\bigl\{(s,y,w):{}&-r^2<s-t<0,\ \abs{w-v}<r,\ \abs{y-x-(s-t)v}<r^3\bigr\},
	\end{split}
\end{equation}	
which determines (see also Section \ref{sec:geometry}) the Morrey class $g\in \M^q_{\K},\,q>1$, with
\begin{equation}\label{kinetic-Morrey}
g\in L^q_{loc}(\R^{1+2d}),\ \norm{g}_{\M^q_{\K}}
	:=\sup_{(t,x,v)\in\R^{1+2d},\ r>0}
	r\left(\frac1{\abs{\cC_r(t,x,v)}}
	\int_{\cC_r(t,x,v)}\abs{g(\xi)}^q\,d\xi\right)^{1/q}<\infty.
\end{equation}
This class is critical in the sense that it is invariant under the kinetic scaling $g_r(t,x,v)=r g(r^2t,r^3x,rv).$

Our main hypothesis on the drift component $b$ in \eqref{eq:forward-intro} and \eqref{eq:sde-intro} consists of the class 
\begin{equation}\label{Morrey-drift}
 b=b_{\mathrm s}+b_{\mathrm b},
 \qquad b_{\mathrm s}\in\M^q_{\K},\qquad
 b_{\mathrm b}\in L^\infty(\R^{1+2d},\R^d),\tag{$\mathrm{H}_{\K}$}
\end{equation}
with Morrey norm $\norm{b_{\mathrm s}}_{\M^q_{\K}}$ required to be sufficiently small.  The Morrey class $\M^q_{\K}$ becomes larger as $q$ gets smaller, so it is reasonable to consider 
$q$ close to $1$ (thus, up to the strict threshold 
 $q>1$, we are dealing with the largest possible critical Morrey class). Fixing $q$ determines also how small $\norm{b_{\mathrm s}}_{\M^q_{\K}}$ must be; this 
 smallness condition on the Morrey norm turns out to be 
  necessary  for \eqref{eq:sde-intro}: Proposition~\ref{prop:hardy-nonexistence} gives a
counterexample to weak existence if the Morrey norm of the drift is too large.

Our class $\M^q_{\K}+L^\infty$ of drifts is general in the sense that it does not presuppose any specific structure of the drift, such as being divergence-free or having gradient form. In particular, it is invariant under $L^\infty$ multipliers. This provides us with flexible means for modifying and constructing  interaction kernels in particle systems with attracting interactions, as we may briefly illustrate at once.

\subsection*{Examples}
	The kinetic Morrey class $\M^q_{\K}$ contains 
	the following vector fields:

1.~Let \(e\in\mathbb R^d\) be fixed. If \(1<q<2\), then, for 
 \(t_0\in\mathbb R\),
		\[
		b(t,x,v)=\kappa |t-t_0|^{-1/2}e
		\in\mathcal M_{\mathcal K}^q.
		\]

If \(1<q<d\), then
		\[
		b(t,x,v)=\kappa\frac{v}{|v|^2}
		\in\mathcal M_{\mathcal K}^q,
		\]

If \(1<q<3d\), then, for 
 some reference point \(x_0\in\mathbb R^d\),
		\[
		b(t,x,v)=\kappa |x-x_0|^{-1/3}e
		\in\mathcal M_{\mathcal K}^q,
		\]
the proof of the three above inclusions is in Proposition \ref{prop:many-particle}. These previous examples can be combined  i.e.\,we can add them up (so that e.g.\,the set of singularities constitutes a dense subset of $\mathbb R^{1+2d}$) or multiply by functions in $L^\infty$. 
Multiplying $b$ by a constant $\kappa>0$ scales its Morrey norm by $\kappa$.

\smallskip
		
2.~If $|b|$ is in the Lebesgue class $L^{p_t}\bigl(\mathbb R_t,
       L^{p_v}(\mathbb R_v^d, L^{p_x}(\mathbb R_x^d))\bigr)$ (i.e. $\int \big(\int\|b(s,\cdot,v)\|^{p_v}_{L^{p_x}}\,dv\big)^{p_t/p_v}\,ds<\infty$) with $p_t,p_x,p_v\in(1,\infty]$ satisfying
		\[
 \frac{2}{p_t}+\frac{d}{p_v}+\frac{3d}{p_x}=1,
\]
then it is in a 
 Morrey class (Proposition \ref{prop:lebesgue-class}).

\smallskip

3.~A motivation for allowing local blow-ups of the drift comes from interacting particle systems with singular attractive forces. Indeed, a system of \(N\) Langevin-type particles in the physical dimension \(d_0\) can be viewed as a single kinetic system in dimension \(d=d_0N\) with $b$ modeling internal forces. For example, the (time-homogeneous) pairwise 
interaction drifts of the form 
$$
b=(b^1,\dots,b^N):\mathbb R^{d_0N} \rightarrow \mathbb R^{d_0N}, \qquad b^i:\mathbb R^{d_0N} \rightarrow \mathbb R^{d_0},
$$
where
$$
 b^i(x)=\frac{1}{N}\sum_{j=1, j\ne i}^N\kappa_{ij}
 \frac{x^i-x^j}{|x^i-x^j|^{4/3}}, \quad x=(x^1,\dots,x^N) \in \mathbb R^{d_0N}
$$
or
$$
  b^i(v)=\frac{1}{N}\sum_{j=1, j\ne i}^N \gamma_{ij}
 \frac{v^i-v^j}{|v^i-v^j|^2}, \quad v=(v^1,\dots,v^N) \in \mathbb R^{d_0N},
$$
and where $1 \leq i \leq N$, $\kappa_{ij}$, $\gamma_{ij} \geq 0$ and both drifts are set to zero at the collisions,
have kinetic degree \(-1\) and therefore exhibit the critical behaviour considered in this paper.  The pair-collision sets \(\{x^i=x^j\}\) and the velocity-synchronization sets \(\{v^i=v^j\}\) are then submanifolds of the phase space $\mathbb R^{2d}$. 
Other examples 
are the many-particle drifts $b^i(x,v)=\frac{1}{N}\sum_{j=1,j \neq i}^N K(x^i-x^j,v^i-v^j)$ with 
$$
K(\xi,\nu)=\kappa|\xi|^{-\frac{1}{3}}\frac{\nu}{1+|\nu|^2},\qquad{\xi,\nu\in\mathbb R^{d_0}},
$$
or more broadly the family of interaction kernels
	$$
	b^i(x,v)=\frac{1}{N}\sum_{j=1, j\ne i}^N a_{ij}(x^i-x^j,v^i-v^j) \biggl[ \kappa_{ij}\frac{x^i-x^j}{|x^i-x^j|^{4/3}} +  \gamma_{ij}\frac{v^i-v^j}{|v^i-v^j|^2} \biggr],
	$$
	where $a_{ij}$ are bounded measurable.
These can be viewed as extended singular versions of the characteristic interaction forces leading to swarming at large time (see \cite{CattiauxDelebecquePedeches} and references therein). The $a_{ij}$'s model non-negative communication rates (essentially acting for close-by particles) which, with a velocity drag force, introduce 
velocity alignment and agglomeration in compact clouds of particles at large time.  This phenomenon is particularly known in the case  $a_{i,j}(x^i-x^j,v^i-v^j)\frac{v^i-v^j}{|v^i-v^j|^2}=\frac{v^i-v^j}{(1+|x^i-x^j|^2)^\kappa}$ in the deterministic and stochastic cases. Introducing a further concentrating force $\frac{v^i-v^j}{|v^i-v^j|^2}$ (assuming e.g. the communication rates are bounded) introduces a stronger form of velocity alignment, and the additional force $\frac{x^i-x^j}{|x^i-x^j|^{4/3}}$ a faster agglomeration of the dynamics. Although large-time behaviors are not discussed in our paper, such questions may be re-addressed in future works.

 The inclusion of the above many-particle drifts in our Morrey class is proved in Proposition \ref{prop:many-particle}. Let us also add that the possibility to select $q>1$ close to $1$ in our Morrey class hypothesis on the drift  is important since it allows us to handle the above many-particle drifts in dimension $d_0=2$.

\smallskip

 Our results thus provide a wellposedness 
  and Feller theory for such stochastic particle systems when the Morrey norm of the attractive interaction is sufficiently small. Conversely, the nonexistence counterexample exhibited in Proposition \ref{prop:hardy-nonexistence} shows that a sufficiently strong inward Hardy-type position-velocity concentrating force 
   can drive the process into the singular set and prevent its continuation as a weak solution. Let us finally add that our condition on the Morrey norm is $d$-dependent, and does not yield weak solution theory of particle systems with an $N$-independent condition on the strength of interaction; in the non-degenerate case, however, something can be said about such general critical interactions when $d_0 \geq 3$ \cite{Ki_AIHP}.

\subsection{Main contributions of the paper} Our first main result, Theorem \ref{thm:main}, develops the Sobolev theory of \eqref{eq:forward-intro} by constructing the Duhamel series for  \eqref{eq:forward-intro} in terms of the fractional powers of
$\lambda+\K$ and $b$. To ensure convergence of the Duhamel series, we obtain Adams-type inequalities \cite{Adams2} on the kinetic homogeneous group: for $1<p<\infty$,
\begin{align*}
 \norm{\abs b^{1/p}(\lambda+\K)^{-1/(2p)}}_{p\rightarrow p}
 &\lesssim \norm b_{\M^q_{\K}}^{1/p},\\
 \norm{\abs b^{1/p}\nabla_v
        (\lambda+\K)^{-1/2-1/(2p)}}_{p\rightarrow p}
 &\lesssim \norm b_{\M^q_{\K}}^{1/p}.
\end{align*}
This yields 
the regularity range in \eqref{eq:forward-intro}:
\[
 \eta\in\W_{\K}^{1/p,p},
 \qquad
 f\in\W_{\K}^{-1-1/p',p},\qquad \frac1{p}+\frac1{p'}=1,
\]
(see \eqref{kinetic-Sobolev} for the definition of $\W^{\gamma,p}_{\K}$) and uniqueness in the corresponding weak class.  Since $Q=4d+2$, a point
mass distribution in $\mathbb R^{1+2d}$ 
belongs to the right-hand-side space whenever
\begin{equation*}
 1<p<\frac{Q-1}{Q-2}=\frac{4d+1}{4d},
\end{equation*}
and so,  choosing appropriately the source term $f$ we can consider \eqref{eq:forward-intro} with a Dirac delta-function in the initial condition.

The second main result, Theorem \ref{thm:sde}, concerns the kinetic SDE 
\eqref{eq:sde-intro}.
Our PDE theory provides kinetic Krylov estimates, existence of a weak solution, conditional weak uniqueness in the Krylov-type class and produces a
strongly continuous backward Feller propagator.

\subsection{Our approach}
1.~The approach to the Sobolev theory of 
Kolmogorov operators via the fractional representation of the Duhamel/Neumann series appeared in the non-degenerate case in \cite{Ki_ASNS}. 
Therein the second-named author 
developed regularity theory of elliptic operator $-\Delta + b \cdot \nabla$ with drift $b$ in the class of weakly form-bounded drifts that is, in fact, larger than the corresponding non-degenerate Morrey class with $q>1$ close to $1$. 
This approach can be viewed as a rather far-reaching extension of the idea
behind the Birman-Schwinger principle for Schr\"odinger operators and is  related to  much earlier work of Sem\"{e}nov \cite{BS} on obtaining characterization of the regularity of the eigenfunctions by working in $L^p$.

The Sobolev theory of \cite{Ki_ASNS} and the ensuing Feller theory were used in \cite{KiS_SPA} to prove for the associated non-degenerate  SDEs  weak existence and uniqueness among all weak solutions 
constructed via regularization of the drift. 
 A parabolic version of \cite{Ki_ASNS, KiS_SPA} was developed in the non-degenerate setting for
time-inhomogeneous Morrey drifts in \cite{Ki_NoDEA} and extended to
stable-driven 
 SDEs in \cite{KiY}.

2.~We prove conditional weak uniqueness for the kinetic stochastic system \eqref{eq:sde-intro} in the class of weak solutions satisfying Krylov estimates ($=$ reasonable occupation estimates). This result is not contained in \cite{KiS_SPA} -- indeed, it requires a parabolic argument -- and has been proved earlier 
in the nondegenerate setting for distributional drifts
\cite{ZhangZhaoDistributional}, Lebesgue class drifts
\cite{XiaXieZhangZhao}, and weak Lebesgue class drifts \cite{RoecknerZhaoWeak}. We also refer to Krylov \cite{Krylov3} for recent results in the non-degenerate setting on weak existence and conditional uniqueness for SDEs with drift in the Morrey class with $q>c_d$, $c_d>\frac{d}{2}$, and a VMO diffusion matrix.

3.~The maximal-function arguments that we use to prove kinetic Adams estimate build on works
\cite{AM, CoifmanRochberg, DH, Heinonen} that deal with a more general setting of 
spaces of homogeneous type.

\subsection{Literature on kinetic Kolmogorov equations and SDEs}

The first result on the operator $\K$ can be traced back 
to Kolmogorov 
\cite{Kolmogoroff} who gave the explicit Gaussian transition density
of integrated Brownian motion and exhibited smoothing despite the absence
of diffusion in the position variables.  H\"ormander's theorem
\cite{Hormander} later placed this phenomenon in a general framework: the
velocity diffusion fields and their first commutators with the transport
field generate all spatial directions.  The intrinsic geometry is therefore
governed by noncommutative space--time translations and anisotropic
dilations rather than by the Euclidean parabolic structure.  Garofalo and
Lanconelli \cite{GarofaloLanconelli} related the level sets of the
fundamental solution to invariant Harnack inequalities, and Lanconelli and
Polidoro \cite{LanconelliPolidoro} developed the associated Lie-group and
dilation structure for a broad class of ultraparabolic Kolmogorov operators.  These ideas
led, among other results, to global lower and Gaussian upper bounds for
fundamental solutions \cite{Polidoro1997,PascucciPolidoro2003}.  We refer
to \cite{AnceschiPolidoro} for a survey of this classical theory.

The regularity theory has developed along several lines.
The
intrinsic Sobolev--Morrey theory of Polidoro and Ragusa
\cite{PolidoroRagusa} was followed by variational kinetic Sobolev
frameworks \cite{AAMN}, kinetic maximal $L^p$ estimates
\cite{NiebelZacher,DongYastrzhembskiy}, and sharp intrinsic Sobolev
embeddings \cite{PascucciPesce,AuscherNiebel}. 
Golse, Imbert, Mouhot and Vasseur extended De Giorgi--Nash--Moser methods
to kinetic Fokker--Planck equations with bounded measurable principal
coefficients, obtaining H\"older regularity and a Harnack inequality
\cite{GIMV}.  Fundamental solutions and Gaussian estimates for variable
or rough principal coefficients have also been constructed, see, for
example, \cite{AuscherImbertNiebel}.  These works principally concern
roughness of the diffusion matrix, local regularity of already constructed
weak solutions, or lower-order terms with substantially more
integrability than those considered below.

For kinetic SDEs and their backward Kolmogorov equations, H\"older
assumptions on the degenerate drift yield weak and strong well-posedness,
with sharp or nearly sharp thresholds in several models
\cite{ChaudruStrong,ChaudruWeak}.  Mixed Sobolev assumptions were treated
in \cite{FedrizziEtAl,ZhangHamiltonian}. In particular, these works use
the smoothing properties of the Kolmogorov equation to construct regular
stochastic flows or to prove well-posedness of the associated martingale
problem.  Heat kernel and gradient estimates under low-regularity
assumptions were obtained in \cite{ChaudruMenozziPesceZhang}.  Singular
coefficients of distributional order have also been treated by
paracontrolled methods in \cite{HaoZhangZhuZhu}.  More recently,
\cite{ChenHaoZhang} established weak well-posedness for time-dependent
distribution-valued drifts in weighted anisotropic H\"older spaces under a
strict subcritical condition and the assumption that $\divv b$ is
bounded.  

Uniqueness of probability solutions to
Fokker--Planck--Kolmogorov equations with partially degenerate
diffusion matrices and unbounded coefficients was studied,
under suitable additional assumptions, by Bogachev,
R\"ockner and Shaposhnikov~\cite{BRS2015}.

Singular interaction kernels depending jointly on position and
velocity also arise in the work of Hao, Jabir, Menozzi, R\"ockner
and Zhang~\cite{HJMRZ2024}, who establish quantitative propagation
of chaos for moderately interacting particle approximations of
kinetic McKean--Vlasov equations with interaction kernels in
anisotropic Besov spaces.

For mixed-Lebesgue drifts in the strictly subcritical range, maximum
principles and kinetic Krylov estimates leading to weak existence for
even density-dependent kinetic SDEs were established in
\cite{ZhangMcKean}.  The scaling-critical Hardy fields considered below
are outside those mixed-Lebesgue hypotheses.

The closest result to the present work is that of Ren and Zhang
\cite{RenZhang}.  For kinetic SDEs with H\"older diffusion and an
unbounded velocity drift in an intrinsic kinetic Kato class, they
construct a weak solution and prove uniqueness within a class requiring
the one-time marginal laws to be absolutely continuous for almost every time.  They also
obtain sharp two-sided heat kernel bounds under a strengthened Kato
assumption.  Their hypotheses include mixed Lebesgue drifts in the
strictly subcritical range
\[
 \frac{2}{p_t}+\frac{d}{p_v}+\frac{3d}{p_x}<1.
\]
The present paper addresses a different scaling-critical endpoint.  A
drift of kinetic degree $-1$, such as e.g.\,Lebesgue drift
\[
 \frac{2}{p_t}+\frac{d}{p_v}+\frac{3d}{p_x}=1
\]
or having even stronger singularity, such as
$|v|^{-1}$, $|t-t_0|^{-1/2}$, or $|x-x_0|^{-1/3}$, does not belong to the Kato class, but it
belongs to the critical kinetic Morrey class of this work. Our drifts are not required to be divergence-free or have gradient structure.

Let us add that, very recently, Chen and Ren \cite{ChenRen} obtained PDE well-posedness for
anisotropic mixed-Lebesgue drifts in a scaling-supercritical range by
exploiting the cancellation $\divv b=0$.  This result is complementary to
ours since its cancellation mechanism is unavailable for attracting drifts, or more generally the drifts
considered in the present paper.

\subsection*{Notations}

\begin{itemize}
\item We write
\[
 z=(x,v)\in\R^{2d},\qquad
 \zeta=(t,z)=(t,x,v)\in\R^{1+2d},\qquad
 \xi=(s,y,w)\in\R^{1+2d},
\]
and use
\[
 dz=dx\,dv,\qquad d\zeta=dt\,dx\,dv.
\]

\item The symbol $\abs{\cdot}$ denotes the Euclidean norm when applied to a
vector and the Lebesgue measure when applied to a measurable set.  We write
$\mathbf 1_E$ for the indicator function of a measurable set $E$.  When
acting on a space--time function, $\mathbf 1_{[s,T]}$ means
$\mathbf 1_{[s,T]\times\R^{2d}}$.

\item $L^p$ denotes the Lebesgue space
\[
 L^p(\R^{1+2d},d\zeta),
 \qquad
 \norm{f}_p:=\norm{f}_{L^p(\R^{1+2d})}.
\]
The same notation is used for finite-dimensional vector-valued functions,
with $\abs{f}$ denoting their Euclidean norm.  We write
$L^p_{\rm loc}$ for the space of measurable functions $f$ such that
\[
 \norm{\mathbf 1_Kf}_p<\infty
 \qquad\text{for every compact }K\subset\R^{1+2d}.
\]
The kinetic metric $d_\ast(\zeta,\xi)$ introduced below (see \eqref{eq:dstar}) induces the usual Euclidean topology on $\mathbb R^{1+2d}$, so this definition is equivalent to requiring local
$L^p$-integrability on every kinetic ball.

\item For $m\ge1$, 
\[
 C_\infty(\R^m)
 :=
 \left\{
 f\in C_b(\R^m):
 \lim_{\abs y\to\infty}f(y)=0
 \right\},
\]
endowed with the supremum norm. 

\item $C_c^\infty(\R^m)$ denotes
the space of smooth compactly supported functions.  

\item $\cS(\R^m)$ denotes the Schwartz space (i.e. the set of all $f\in C^\infty$ such that $\sup_{y\in\R^m}|y|^l|D^{\alpha}f(y)|<\infty$ for all $l\ge 1$, $\alpha\in\mathbb N^{m}$).

\item For $1\le p\le\infty$, the conjugate exponent is denoted by $p'$:
\[
 \frac1p+\frac1{p'}=1.
\]

\item If $T:L^p\rightarrow L^r$ is bounded, we write
$
 \norm{T}_{p\rightarrow r}
 :=
 \sup_{\norm f_p\le1}\norm{Tf}_r.
$
In particular, $\norm{T}_{p\rightarrow p}$ denotes the operator norm of $T$ on
$L^p$.

\item For brevity, we often make use of the notation
$$
\langle h\rangle:=\int_{\mathbb R^{1+2d}}h\,d\zeta, \quad \langle f,g\rangle:=\langle fg \rangle. $$
Throughout this paper, all functions and vector
fields are real-valued.

\item We write
$
 A\lesssim B
$
when $A\le CB$ for some positive 
 constant $C$. Constants denoted by $c$ or $C$ may change from line to line. 

\item For $f\in C([s,T], \R^m)$, we define the 
uniform modulus of
continuity by
\[
 \mathfrak w_{s,T}(f,\delta)
 :=
 \sup_{\substack{s\le u,t\le T\\\abs{t-u}\le\delta}}
 \abs{f(t)-f(u)}.
\]
When the interval $[s,T]$ is fixed, we write simply
$\mathfrak w(f,\delta)$.

\item $[ M,N]_t$ denotes the cross-variation bracket  between two continuous local martingales and $[M]_t=[M,M]_t$ the quadratic variation of $M$.

\item 
For locally finite Radon measures on a locally compact space,
vague convergence $\rightharpoonup$ means convergence of integrals against
every continuous compactly supported test function.

\item We also use $\rightharpoonup$ for weak convergence in function spaces.

\item Finally, we denote by $\Rightarrow$ weak convergence of probability measures i.e. for $\{\mu_n\}_n$ and $\mu$ probabilities on a metric space $E$,  
\[
\mu_n\Rightarrow\mu \quad\Leftrightarrow\quad \int_E f\,d\mu_n\rightarrow \int_E f\,d\mu
 \qquad\text{for every }f\in C_b(E).
\]
\end{itemize}

\bigskip

\section{Kinetic geometry and kinetic Bessel potentials}\label{sec:geometry}

\subsection{The kinetic group}
For $\zeta=(t,x,v)$ and $\xi=(s,y,w)$ in $\R^{1+2d}$, set the composition law
\begin{equation}\label{CompositionLaw}
 \zeta\circ\xi=(t+s,x+y+s v,v+w).
\end{equation}
Then the inverse element is given by
\[
 \zeta^{-1}=(-t,-x+tv,-v).
\]
Lebesgue measure is a left and right Haar measure, and the vector fields
\[
 Y:=\partial_t+v\cdot\nabla_x,
 \qquad X_j:=\partial_{v_j},\quad 1\le j\le d,
\]
are left invariant.  The maps
\begin{equation*}
 \delta_r(t,x,v):=(r^2t,r^3x,rv),\qquad r>0,
\end{equation*}
are group automorphisms.  Lebesgue measure scales by $r^Q$, with $Q$ as in
\eqref{eq:Q}, and $\K=Y-\sum_jX_j^2$ is homogeneous of degree two.

Define the left-invariant, generally non-symmetric, distance
\begin{equation}\label{eq:d}
 d(\zeta,\xi)
 :=\max\left\{
  \abs{t-s}^{1/2},\ \abs{w-v},\
  \abs{y-x-(s-t)v}^{1/3}\right\}.
\end{equation}
It is immediate from the group law that
$d(\zeta,\xi)=d(0,\zeta^{-1}\circ\xi)$.  It also satisfies the triangle
inequality. 
Indeed, 
let $\zeta=(t,x,v)$, $\theta=(r,z,u)$, $\xi=(s,y,w)$,
$a=d(\zeta,\theta)$ and $c=d(\theta,\xi)$. Then
\begin{align*}
	|y-x-(s-t)v|
	&\le |z-x-(r-t)v|+|y-z-(s-r)u|+|(s-r)(u-v)|\\
	&\le a^3+c^3+ac^2\le(a+c)^3.
\end{align*}
The time and velocity components are straightforward.
Moreover,
\begin{equation}\label{eq:reverse-comparable}
 2^{-1/3}d(\xi,\zeta)\le d(\zeta,\xi)\le2^{1/3}d(\xi,\zeta).
\end{equation}
Consequently,
\begin{equation}\label{eq:dstar}
 d_*(\zeta,\xi):=d(\zeta,\xi)+d(\xi,\zeta)
\end{equation}
is a left-invariant metric equivalent to $d$.  The corresponding balls
$B_r(\zeta):=\{\xi:d_*(\zeta,\xi)<r\}$ satisfy
\begin{equation}\label{eq:ball-volume}
 c_d r^Q\le\abs{B_r(\zeta)}\le C_d r^Q.
\end{equation}
Thus $(\R^{1+2d},d_*,dtdxdv)$ is a doubling space of homogeneous type.

With this set-up, we can define the Morrey class $\M^q_\K$ equally as in Section \ref{sec:intro}.  
\begin{definition}\label{def:morrey}
	Let $q>1$.  A Borel function $g$ belongs to the critical kinetic Morrey class
	$\M^q_{\K}$ if $g\in L^q_{\mathrm{loc}}$ and
	\begin{equation*}
		\norm g_{\M^q_{\K}}
		:=\sup_{\zeta\in\R^{1+2d},\ r>0}
		r\left(\frac1{\abs{B_r(\zeta)}}
		\int_{B_r(\zeta)}\abs{g(\xi)}^q\,d\xi\right)^{1/q}<\infty.
	\end{equation*}
\end{definition}
This class where backward cylinders $C_r$ defined in \eqref{eq:backward-cylinder} are replaced by $d_\ast$-balls $B_r$, is equivalent to the one introduced in \eqref{kinetic-Morrey}, with an equivalent norm. Indeed, recalling that \begin{equation*}
	\begin{split}
		\cC_r(t,x,v):=\bigl\{(s,y,w):{}&-r^2<s-t<0,\ \abs{w-v}<r,\ \abs{y-x-(s-t)v}<r^3\bigr\},
	\end{split}
\end{equation*}
we have
\[
\cC_r(t_0,x_0, v_0)\subset B_{3r}(t_0,x_0, v_0)
\]
and, conversely,
\[
B_r(t_0,x_0,v_0)
\subset
\cC_{\sqrt 2r}
\bigl(t_0+r^2,x_0+r^2v_0,v_0\bigr).
\]
\subsection{The kinetic Gaussian kernel}

Let $\{\mathsf P_\tau\}_{\tau\ge0}$ be the contraction semigroup on
$L^p(\R^{1+2d})$, $1\le p\le\infty$, generated by 
 $-\K$.
Explicitly,
\begin{equation}\label{eq:spacetime-semigroup}
 (\mathsf P_\tau f)(t,x,v)
 =\int_{\R^{2d}}p_\tau(x,v;y,w)f(t-\tau,y,w)\,dy\,dw,
\end{equation}
where
\begin{equation*}
 C_\tau=\begin{pmatrix}
  \tau I_d&-\frac{\tau^2}{2}I_d\\[2mm]
  -\frac{\tau^2}{2}I_d&\frac{\tau^3}{3}I_d
 \end{pmatrix},
 \qquad
 \Xi_\tau=\binom{w-v}{y-x+\tau v},
\end{equation*}
and
\begin{equation}\label{eq:heat-kernel}
 p_\tau(x,v;y,w)
 =\left(\frac{\sqrt3}{2\pi\tau^2}\right)^d
   \exp\left(-\frac14\pair{C_\tau^{-1}\Xi_\tau}{\Xi_\tau}\right).
\end{equation}
The above corresponds to the transition density of
\[
dX_r=-V_r\,dr, \quad dV_r=\sqrt2\,dW_r,
\]
started from $(x,v)$. It follows directly from
\eqref{eq:heat-kernel} that, for suitable positive
constants $c_i=c_i(d)$,
\begin{equation}\label{eq:gaussian-upper}
\begin{aligned}
 p_\tau(x,v;y,w)
 &\le c_1\tau^{-2d}
  \exp\left[-c_2\left(
   \frac{\abs{w-v}^2}{\tau}
  +\frac{\abs{y-x+\tau v}^2}{\tau^3}\right)\right],\\
 \abs{\nabla_v p_\tau(x,v;y,w)}
 &\le c_3\tau^{-1/2}\,\widetilde p_\tau(x,v;y,w),
\end{aligned}
\end{equation}
where $\widetilde p_\tau$ is a Gaussian of the same form as in the upper-bound with a smaller
constant in the exponential.  Analogous estimates hold for the adjoint
semigroup.

For $\lambda>0$ and $\K$ the forward Kolmogorov operator, define the operator $\A$ and its adjoint $\A^*$ as
\begin{equation*}
 \A:=\lambda+\K,
 \qquad \A^*=\lambda+\K^*=\lambda-\partial_t-v\cdot\nabla_x-\Delta_v.
\end{equation*}
(Although $\A$ depends on the choice of $\lambda$, for notation convenience, we make this dependency implicit in our writing.)

For $\gamma>0$, define the negative fractional powers by
\begin{equation}\label{eq:fractional-power}
 \A^{-\gamma/2}f
 =\frac1{\Gamma(\gamma/2)}
   \int_0^\infty \tau^{\gamma/2-1}e^{-\lambda\tau}
       \mathsf P_\tau f\,d\tau,
\end{equation}
and define $(\A^*)^{-\gamma/2}$  and $\A^{\gamma/2}=(\A^{-\gamma/2})^{-1}$ analogously. In particular,
\begin{equation*}
 \A^{-1}f(t,x,v)
 =\int_{-\infty}^t e^{-\lambda(t-s)}
  \int_{\R^{2d}}p_{t-s}(x,v;y,w)f(s,y,w)\,dy\,dw\,ds.
\end{equation*}
Fundamentals on fractional powers of operators, the specific derivations of $\A^{\gamma/2}$ and other key properties used throughout are collected in Appendix \ref{app:fractional-powers} and below.

In particular, we write $\A^{-\gamma/2}(\cdot,\xi):=\A^{-\gamma/2}\delta_{\{\xi\}}$
and $(\A^\ast)^{-\gamma/2}(\cdot,\xi):=(\A^\ast)^{-\gamma/2}\delta_{\{\xi\}}$,
where $\delta_{\{\xi\}}$ is the Dirac mass at $\xi\in\R^{1+2d}$.
\begin{proposition}[Fractional kernel bounds]\label{prop:kernels}
Let $0<\gamma<Q$, $1<r<\infty$, and $\lambda>0$.
The following assertions hold, with constants independent of $\lambda$ in
the pointwise estimates.
\begin{enumerate}[label=\textup{(\roman*)}]
\item The integral kernels of $\A^{-\gamma/2}$ and
$(\A^*)^{-\gamma/2}$ satisfy
\begin{equation}\label{eq:fractional-kernel-bound}
 \abs{\A^{-\gamma/2}(\zeta,\xi)}
 +\abs{(\A^*)^{-\gamma/2}(\zeta,\xi)}
 \le C_{d,\gamma}d_*(\zeta,\xi)^{\gamma-Q}.
\end{equation}
\item The differentiated kernels satisfy
\begin{equation}\label{eq:gradient-fractional-kernel-bound}
 \begin{split}
 &\abs{\nabla_v\A^{-(1+\gamma)/2}(\zeta,\xi)}\\
 &\qquad
 +\abs{\nabla_v(\A^*)^{-(1+\gamma)/2}(\zeta,\xi)}
 \le C_{d,\gamma}d_*(\zeta,\xi)^{\gamma-Q}.
 \end{split}
\end{equation}
\item One has
\begin{align}
 \norm{\A^{-\gamma/2}}_{r\rightarrow r}
 +\norm{(\A^*)^{-\gamma/2}}_{r\rightarrow r}
 &\le2\lambda^{-\gamma/2},\label{eq:power-Lp}\\
 \norm{\nabla_v\A^{-(1+\gamma)/2}}_{r\rightarrow r}
 +\norm{\nabla_v(\A^*)^{-(1+\gamma)/2}}_{r\rightarrow r}
 &\le C_{d,r,\gamma}\lambda^{-\gamma/2}.
 \label{eq:gradient-power-Lp}
\end{align}
\end{enumerate}
\end{proposition}
For the sake of completeness, 
the proof of Proposition \ref{prop:kernels}, which essentially relies on suitable estimates on $\{\mathsf P_\tau\}_{\tau\ge0}$, is included in Appendix \ref{prop:kernels_proof}.

For later use, define the intrinsic Riesz potential
\begin{equation}\label{eq:Tgamma}
 I_\gamma f(\zeta)
 :=\int_{\R^{1+2d}}
       \frac{\abs{f(\xi)}}{d_*(\zeta,\xi)^{Q-\gamma}}\,d\xi,
 \qquad 0<\gamma<Q.
\end{equation}
Clearly, $I_\gamma$ is positively homogeneous
and subadditive.  On nonnegative functions it coincides with the action of the linear 
operator having the same integral kernel.
Thus Proposition \ref{prop:kernels} gives the pointwise bounds
\begin{align}
 \abs{\A^{-\gamma/2}f}
 +\abs{(\A^*)^{-\gamma/2}f}
 &\le CI_\gamma f,\label{eq:power-by-I}\\
 \abs{\nabla_v\A^{-(1+\gamma)/2}f}
 +\abs{\nabla_v(\A^*)^{-(1+\gamma)/2}f}
 &\le CI_\gamma f.
 \label{eq:gradient-by-I}
\end{align}

\begin{proposition}[Kinetic Sobolev embeddings]\label{prop:sobolev}
Let $0<\gamma<Q$.
\begin{enumerate}[label=\textup{(\roman*)}]
\item If $1<r<Q/\gamma$ and $1/s=1/r-\gamma/Q$, then
\begin{equation}\label{eq:HLS}
 \norm{\A^{-\gamma/2}f}_{s}
 +\norm{(\A^*)^{-\gamma/2}f}_{s}
 \le C\norm f_r.
\end{equation}
\item If $1<r<\infty$ and $\gamma>Q/r$, then $\A^{-\gamma/2}$ and
$(\A^*)^{-\gamma/2}$ map $L^r$ continuously into
$C_\infty(\R^{1+2d})$.
\end{enumerate}
\end{proposition}

We included the proof of Proposition \ref{prop:sobolev} in Appendix \ref{prop:sobolev_proof}.

\subsection{Kinetic Bessel-potential spaces}\label{sec:spaces}

Fix $\lambda>0$ and $1<p<\infty$.  For $\gamma>0$, define
\begin{equation*}
 \W_{\K}^{\gamma,p}
 :=\A^{-\gamma/2}L^p(\R^{1+2d}),
 \qquad
 \norm u_{\W_{\K}^{\gamma,p}}
 :=\norm{\A^{\gamma/2}u}_p.
\end{equation*}
For $\gamma<0$, let $\W_{\K}^{\gamma,p}$ be the completion of $L^p$ in the
norm
\begin{equation}\label{kinetic-Sobolev}
 \norm f_{\W_{\K}^{\gamma,p}}:=\norm{\A^{\gamma/2}f}_p.
\end{equation}
Define $\W_{\K^*}^{\gamma,p}$ using $\A^*$ in the same way. We have the duality relation
\begin{equation}\label{eq:Bessel-duality}
 (\W_{\K}^{\gamma,p})^*=\W_{\K^*}^{-\gamma,p'},
\end{equation}
with the distributional pairing extending the usual $L^2$ pairing.  These spaces
refine the scale of kinetic Sobolev spaces; for further details see \cite{NiebelZacher,AuscherNiebel}.

The operator $\A$ is an isometric isomorphism
\begin{equation*}
 \A:\W_{\K}^{\gamma,p}\rightarrow\W_{\K}^{\gamma-2,p}.
\end{equation*}
In particular, if
$
 f\in\W_{\K}^{-1-1/p',p}
       =\W_{\K}^{1/p-2,p}$,
then
$
 \A^{-(1-1/(2p))}f\in L^p.
$

The following consequence of Proposition \ref{prop:sobolev} will be needed in order to consider a delta-function in the initial condition for the forward equation.

\begin{corollary}[Point evaluation]\label{cor:point-evaluation}
If $s>Q/p'$, then
\begin{equation*}
 \W_{\K^*}^{s,p'}\hookrightarrow C_\infty(\R^{1+2d}),
\end{equation*}
and therefore
\begin{equation*}
 \delta_{\{\zeta_0\}}\in\W_{\K}^{-s,p}
 \qquad\text{for every }\zeta_0\in\R^{1+2d}.
\end{equation*}
\end{corollary}

\begin{proof}
Choose $\gamma$ with $Q/p'<\gamma<\min\{s,Q\}$.  The factorization
$
 (\A^*)^{-s/2}
 =(\A^*)^{-\gamma/2}(\A^*)^{-(s-\gamma)/2},
$
Proposition~\ref{prop:kernels}(iii), and
Proposition~\ref{prop:sobolev}(ii) give the asserted embedding for every
$s>Q/p'$, including $s\ge Q$. Owing to the duality $C^*_\infty(\mathbb R^{1+2d})\hookrightarrow (\W_{\K^*}^{s,p'})^*$, we have $\delta_{\{\zeta_0\}}\in\W_{\K}^{-s,p}$.
\end{proof}

\bigskip

\section{Main results}\label{sec:main-results}

\subsection{Hypothesis, smallness condition  and Duhamel blocks} \

\medskip
\noindent
\textbf{Hypothesis $(\mathrm H_{\K})$.}
There exist $q>1$ and a decomposition of the drift into a singular and a bounded part, respectively:
\begin{equation}
\label{eq:HK}
 b=b_{\mathrm s}+b_{\mathrm b},
 \qquad b_{\mathrm s}\in\M^q_{\K},
 \quad b_{\mathrm b}\in L^\infty(\R^{1+2d},\R^d).\tag{$\mathrm H_{\K}$}
\end{equation}
The norm $\norm{b_{\mathrm s}}_{\M^q_{\K}}$ will be chosen sufficiently small later on; see \eqref{eq:smallness}.
\medskip

Define the power drift $$b^{1/p'}:=b|b|^{-1+1/p'} \quad \text{on $\{0<|b|<\infty\}$}$$ and set it equal to zero elsewhere.
For $1<p<\infty$, define, on smooth test functions, the operators
\begin{align}
 H_p(b)&:=\abs b^{1/p}\A^{-1/(2p)},
 \label{eq:Hp}\\
 G_p(b)&:=-\A^{-(1-1/(2p))}
       \divv\bigl(b^{\frac{1}{p'}}\,\cdot\bigr),
 \label{eq:Gp}
\end{align}
which will serve as the building blocks for the Duhamel series for $\eta$, solution to \eqref{eq:forward-intro}.

As shown in Proposition \ref{prop:rellich},  $H_p(b),G_p(b)$ are bounded operators on $L ^p$ 
and
\begin{align}
 \norm{H_p(b)}_{p\rightarrow p}
 &\le C_{d,p,q}\norm{b_{\mathrm s}}_{\M^q_{\K}}^{1/p}
 +C_{d,p}\lambda^{-1/(2p)}\norm{b_{\mathrm b}}_\infty^{1/p}
 =:\mathfrak h_{p,\lambda}(b)  \label{eq:hbound} \\
 \norm{G_p(b)}_{p\rightarrow p}
 =\norm{G_p^*(b)}_{p'\rightarrow p'}
 & \le C_{d,p',q}\norm{b_{\mathrm s}}_{\M^q_{\K}}^{1/p'}
 +C_{d,p'}\lambda^{-1/(2p')}\norm{b_{\mathrm b}}_\infty^{1/p'} 
 =:\mathfrak g_{p,\lambda}(b),
 \label{eq:gbound}
\end{align}
where the adjoint of $G_p{(b)}$ is
\begin{equation*}
 G_p^*{(b)}=b^{\frac{1}{p'}}\cdot
       \nabla_v(\A^*)^{-(1-1/(2p))}.
\end{equation*}
We will use the same notation for the extension of these operators by continuity to $L^p$.

(Let us emphasize that in the kinetic setting one cannot commute the gradient $\nabla_v$ and the free resolvent since 
$
 [\partial_{v_j},\partial_t+v\cdot\nabla_x]=\partial_{x_j}.
$)

For fixed $1<p<\infty$, we can state our smallness condition on the drift $b$: there exists
$\lambda>0$ such that
\begin{equation}\label{eq:smallness}
	\mathfrak g_{p,\lambda}(b)\mathfrak h_{p,\lambda}(b)<1.
	\tag{$\mathrm H_{\mathrm s,p}$}
	\end{equation}
In particular, by \eqref{eq:hbound}-\eqref{eq:gbound}, for every fixed $d,p,q$, there is $c_{d,p,q}>0$ such that
\eqref{eq:smallness} holds for all sufficiently large $\lambda$ whenever
\begin{equation*}
 \norm{b_{\mathrm s}}_{\M^q_{\K}}<c_{d,p,q}.
\end{equation*}

\subsection{Regularization of the drift}\label{sec:regularization}

\begin{definition}\label{def:admissible}
A sequence of bounded smooth vector fields $\{b_n\}_n$ is an admissible (or $p$-admissible)  
regularization of $b$ 
if
\begin{enumerate}[label=\textup{(\roman*)}]
\item $b_n\rightarrow b$ almost everywhere and in $L^q_{\mathrm{loc}}$ where $q$ is the exponent in the decomposition  \eqref{eq:HK};

\item For all $n$, $b_n\in L^\infty+\M^q_\K$ where the singular and bounded parts satisfy, uniformly in $n$, the bounds in
\eqref{eq:smallness} with $b_n$ in place of $b$, namely,
\[
 \sup_n\mathfrak g_{p,\lambda}(b_n)
 \mathfrak h_{p,\lambda}(b_n)<1;
\]
\item As $n\rightarrow \infty$,
\begin{equation}\label{eq:strong-factors}
 H_p(b_n)\rightarrow H_p(b),
 \qquad G_p(b_n)\rightarrow G_p(b)
 \quad\text{strongly on }L^p.
\end{equation}
\end{enumerate}
\end{definition}

\begin{proposition}[Existence of admissible regularizations]
\label{prop:regularization}
Let $b$ satisfy \eqref{eq:HK} and \eqref{eq:smallness}.  Then there is an
admissible regularization $\{b_n\}_n$ for $b$ for which
\begin{equation}\label{eq:regularization-bounds}
 \norm{b_{n,\mathrm s}}_{\M^q_{\K}}
 \le\norm{b_{\mathrm s}}_{\M^q_{\K}},
 \qquad
 \norm{b_{n,\mathrm b}}_\infty
 \le\norm{b_{\mathrm b}}_\infty.
\end{equation}
\end{proposition}

\begin{proof}
We separate truncation from mollification.  This is useful because the
operators in \eqref{eq:Hp}--\eqref{eq:Gp} depend nonlinearly on the total
vector field $b$, rather than separately on the two terms in
\eqref{eq:HK}.

Choose a sequence of cut-offs $\chi_n\in C_c^\infty(\R^{1+2d})$ such that
$0\le\chi_n\le1$, $\chi_n=1$ on $B_n(0)$, and $\chi_n\rightarrow1$
pointwise.  With the $min\{\cdot,\cdot\}$ 
below interpreted as equal to $1$ on $\{b=0\}$, let us put for a.e. $\zeta\in \mathbb R^{1+2d}$,
\begin{equation}\label{eq:regularization-scalar-cutoff}
 a_n(\zeta):=\chi_n(\zeta)
       \min\left\{1,\frac{n}{\abs{b(\zeta)}}\right\},
 \qquad
 c_n:=a_nb,
 \qquad
 c_{n,\mathrm s}:=a_nb_{\mathrm s},
 \qquad
 c_{n,\mathrm b}:=a_nb_{\mathrm b}.
\end{equation}
Then $c_n=c_{n,\mathrm s}+c_{n,\mathrm b}$,
$0\le a_n\le1$, and $a_n\rightarrow1$ almost everywhere.  Moreover,
\[
 \abs{c_n}\le n,
 \qquad
 \abs{c_{n,\mathrm s}}\le n+\norm{b_{\mathrm b}}_\infty,
 \qquad
 \abs{c_{n,\mathrm b}}\le\norm{b_{\mathrm b}}_\infty.
\]
Thus both components are bounded and compactly supported. We also have, clearly,
\begin{equation}\label{eq:truncated-component-bounds}
 \norm{c_{n,\mathrm s}}_{\M^q_{\K}}
 \le\norm{b_{\mathrm s}}_{\M^q_{\K}},
 \qquad
 \norm{c_{n,\mathrm b}}_\infty
 \le\norm{b_{\mathrm b}}_\infty.
\end{equation}
Since $a_n\rightarrow1$ and $0\le a_n\le1$, applying the dominated convergence theorem on compact
sets also shows that
\begin{equation}\label{eq:truncated-local-convergence}
 c_{n,\mathrm s}\rightarrow b_{\mathrm s}\quad\hbox{in }L^q_{\mathrm{loc}},
 \qquad
 c_{n,\mathrm b}\rightarrow b_{\mathrm b}\quad\hbox{in }L^q_{\mathrm{loc}},
\end{equation}
as well as a.e.

The reason for using the same scalar $a_n$ in both components $c_{n,\mathrm s}$ and $ c_{n,\mathrm b}$ is explained by observing the pair of
exact identities: for all $f\in L^p$,
\begin{equation}\label{eq:truncated-factor-identities}
 \begin{split}
 H_p(c_n)f&=a_n^{1/p}H_p(b)f,\\
 G_p(c_n)f&=G_p(b)\bigl(a_n^{1/p'}f\bigr).
 \end{split}
\end{equation}
Indeed, the first identity follows immediately from $\abs{c_n}^{1/p}=a_n^{1/p}\abs b^{1/p}$. For the second one, we have 
$c_n^{\frac{1}{p'}}=a_n^{1/p'}b^{\frac{1}{p'}}$
 from which we note that, for $f\in L^p$ and
$\varphi\in L^{p'}$, we have
\[
 \begin{split}
 \langle G_p(c_n)f,\varphi\rangle
 &=\int f\,c_n^{\frac{1}{p'}}\cdot
       \nabla_v(\A^*)^{-(1-1/(2p))}\varphi\,d\zeta\\
 &=\int a_n^{1/p'}f\,b^{\frac{1}{p'}}\cdot
       \nabla_v(\A^*)^{-(1-1/(2p))}\varphi\,d\zeta\\
 &=\langle G_p(b)(a_n^{1/p'}f),\varphi\rangle.
 \end{split}
\] 
 The
first identity in
\eqref{eq:truncated-factor-identities}, dominated convergence in $L^p$, and
the boundedness of $H_p(b)$ imply
\[
 H_p(c_n)f\rightarrow H_p(b)f
 \qquad\hbox{in }L^p
\]
for every $f\in L^p$.  The second identity and
$a_n^{1/p'}f\rightarrow f$ in $L^p$ similarly imply
\[
 G_p(c_n)f\rightarrow G_p(b)f
 \qquad\hbox{in }L^p.
\]
After this truncation, we next apply a mollifier.  Let $\rho\ge0$ be smooth, supported in $B_1(0)$ and
having integral equal to $1$, and put
$\rho_\eps(\zeta):=\eps^{-Q}\rho(\delta_{1/\eps}\zeta)=\eps^{-Q}\rho(t/\eps^2,x/\eps^3,v/\eps)$. For the mollification, we will use the convolution 
along the composition law \eqref{CompositionLaw}:
\begin{equation*}
 (\rho_\eps*g)(\zeta)
 :=\int\rho_\eps(\eta)g(\eta^{-1}\circ\zeta)\,d\eta.
\end{equation*}
This convolution is a contraction in the critical Morrey norm $\|\cdot\|_{\M^q_\K}$.  In fact,
left invariance of the related distance $d_*$ and of Lebesgue measure gives
$
\xi^{-1}\circ B_r(\zeta)=B_r(\xi^{-1}\circ\zeta),
$
and hence 
\begin{align*}
 \norm{\rho_\eps*g}_{L^q(B_r(\zeta))}
 &\le\int\rho_\eps(\xi)
 \norm{g(\xi^{-1}\circ\,\cdot)}_{L^q(B_r(\zeta))}\,d\xi=\int\rho_\eps(\xi)
 \norm g_{L^q(B_r(\xi^{-1}\circ\zeta))}\,d\xi.
\end{align*}
After multiplication by $r\abs{B_r}^{-1/q}$ and taking the supremum over
$r$ and $\zeta$, this yields
\begin{equation}\label{eq:convolution-contractions}
 \norm{\rho_\eps*g}_{\M^q_{\K}}
 \le\norm g_{\M^q_{\K}} \quad \text{ and } \quad
 \norm{\rho_\eps*g}_\infty\le\norm g_\infty.
\end{equation}

For fixed $n$, write now $c_{n,\eps}:=\rho_\eps*c_n$ for which we have the further decomposition $c_{n,\eps}=\rho_\eps*c_{n,\mathrm b}+\rho_\eps*c_{n,\mathrm s}\in L^\infty+\M^q_\K$. The approximation of the identity gives $c_{n,\eps}\rightarrow c_n$ in every finite local Lebesgue
space and almost everywhere as $\eps\downarrow0$.  For $0<\eps\le1$, all
these functions are uniformly bounded and supported in a fixed compact set
depending on $n$.  It follows, first for any $f\in C_c^\infty$ (which, we may recall, is dense in $L^p$), that
\begin{equation}\label{eq:mollified-H-convergence}
 H_p(c_{n,\eps})f\rightarrow H_p(c_n)f
 \qquad\hbox{in }L^p.
\end{equation}
Indeed, the difference being 
$
 \bigl(\abs{c_{n,\eps}}^{1/p}-\abs{c_n}^{1/p}\bigr)
       \A^{-1/(2p)}f,
$
$$
\|H_p(c_{n,\eps})f- H_p(c_n)f\|_p\le \||c_{n,\eps}-c_n|^{\frac1{p}}\A^{-1/(2p)}f\|_p\\
$$
and the convergence to $0$ follows next from \eqref{eq:power-Lp} in Proposition \ref{prop:kernels}(iii) and the dominated convergence theorem (recalling that $|c_{n}|\le n$ and so $|c_{n,\eps}|\le n$).

To handle the second limit $G_p(c_{n,\eps})\rightarrow G_p(c_n)$ in $L^p$
, let us introduce the operator
\[
 D_p:=-\A^{-(1-1/(2p))}\divv,
\]
for which we have $G_p(c_n)(f)=D_p(c^{\frac1{p'}}_n f)$. Its adjoint is
$D_p^*=\nabla_v(\A^*)^{-(1+1/p')/2}$.  Then \eqref{eq:gradient-power-Lp} in 
Proposition \ref{prop:kernels}(iii) with $\gamma=1/p'$ and the duality $\|D_pF\|_{p}=\sup_{\|G\|_{p'}=1}\langle F,D_p^*G\rangle$  give
\begin{equation*}
 \norm{D_pF}_p
 \le C_{d,p}\lambda^{-1/(2p')}\norm F_p,
 \qquad F\in L^p(\R^{1+2d},\R^d).
\end{equation*}
Therefore, for $f\in C_c^\infty$, as $\eps\rightarrow 0$,
\begin{align*}
 \norm{G_p(c_{n,\eps})f-G_p(c_n)f}_p
 &\le C_{d,p}\lambda^{-1/(2p')}
   \norm{\bigl(c^{\frac{1}{p'}}_{n,\eps}-c^{\frac{1}{p'}}_n\bigr)f}_p
 \rightarrow0,
\end{align*}
follows again by dominated convergence.  Uniform bounds and density extend
\eqref{eq:mollified-H-convergence} and the corresponding convergence for
$G_p$ to every fixed $f\in L^p$.

It remains to choose one mollification scale in terms of 
$n$ to eventually set $c_{n,\varepsilon_n}$ as an admissible approximation. First note that we can choose $\varepsilon_n\downarrow0$ rapidly such that
\begin{equation}\label{eq:diagonal-local-errors}
	\norm{\rho_{\eps_n}*c_{n,\mathrm s}-c_{n,\mathrm s}}_{L^q(B_n)}^q
	+\norm{\rho_{\eps_n}*c_{n,\mathrm b}-c_{n,\mathrm b}}_{L^q(B_n)}^q
	<2^{-n}.
\end{equation}
Then let \(\{f_j:j\ge1\}\subset C_c^\infty(\mathbb R^{1+2d})\) be a countable dense subset of \(L^p\). By the preceding operator convergences, we may choose \(\varepsilon_n\downarrow0\), with \(0<\varepsilon_n<1/n\), such that for $1\le j\le n$,
\begin{equation}\label{eq:diagonal-factor-errors}
	\begin{split}
		&\norm{\bigl(H_p(\rho_{\eps_n}*c_n)-H_p(c_n)\bigr)f_j
		}_p\\
		&\qquad+
		\norm{\bigl(G_p(\rho_{\eps_n}*c_n)-G_p(c_n)\bigr)
			f_j
		}_p
		<2^{-n},
	\end{split}
\end{equation}
and \eqref{eq:diagonal-local-errors} holds simultaneously.
Define
\begin{equation*}
 b_{n,\mathrm s}:=\rho_{\eps_n}*c_{n,\mathrm s},
 \qquad
 b_{n,\mathrm b}:=\rho_{\eps_n}*c_{n,\mathrm b},
 \qquad
 b_n:=b_{n,\mathrm s}+b_{n,\mathrm b}
       =\rho_{\eps_n}*c_n.
\end{equation*}
These vector fields are smooth, bounded, and compactly supported.
Equations \eqref{eq:truncated-local-convergence} and
\eqref{eq:diagonal-local-errors} give convergence of the two components in
$L^q_{\mathrm{loc}}$. 
\eqref{eq:diagonal-local-errors}, applied on each fixed ball, also gives
$b_{n,\mathrm s}-c_{n,\mathrm s}\rightarrow0$ and
$b_{n,\mathrm b}-c_{n,\mathrm b}\rightarrow0$ almost everywhere. Hence
$b_n\rightarrow b$ almost everywhere.

Combining \eqref{eq:truncated-factor-identities} with
\eqref{eq:diagonal-factor-errors} proves \eqref{eq:strong-factors} first on
the dense set $\{f_j\}$ and then on all of $L^p$, using the uniform operator
bounds.  Finally, \eqref{eq:truncated-component-bounds} and
\eqref{eq:convolution-contractions} give the stronger estimates
\[
 \norm{b_{n,\mathrm s}}_{\M^q_{\K}}
 \le\norm{b_{\mathrm s}}_{\M^q_{\K}},
 \qquad
 \norm{b_{n,\mathrm b}}_\infty
 \le\norm{b_{\mathrm b}}_\infty.
\]
Thus \eqref{eq:regularization-bounds} holds, and
\[
 \sup_n\mathfrak g_{p,\lambda}(b_n)\mathfrak h_{p,\lambda}(b_n)
 \le\mathfrak g_{p,\lambda}(b)\mathfrak h_{p,\lambda}(b)<1.
\]
This proves admissibility.
\end{proof}

\subsection{Forward and backward Kolmogorov equations} 

We can now state our main results on the solution to Equation \ref{eq:forward-intro}.
\begin{theorem}[Forward PDE theory]\label{thm:main}
Let $d\ge2$ and $b$ satisfy \eqref{eq:HK}. Fix $1<p<\infty$ and
assume \eqref{eq:smallness}.  Then:
\begin{enumerate}[label=\textup{(\roman*)}]
\item For every $f\in\W_{\K}^{-1-1/p',p}$, the equation
\[
 \A\eta-\divv(b\eta)=f
\]
has a unique weak solution $\eta\in\W_{\K}^{1/p,p}$ (see Definition \ref{def:weak}), and, with $H_p(b)$ and $G_p(b)$ as in \eqref{eq:Hp}-\eqref{eq:Gp}, the following Duhamel series representation holds:
\begin{equation}\label{eq:main-representation}
 \eta=\A^{-1/(2p)}(I+G_p(b)H_p(b))^{-1}
       \A^{-(1-1/(2p))}f.
\end{equation}
Moreover,
\begin{equation}\label{eq:apriori}
 \norm\eta_{\W_{\K}^{1/p,p}}
 \le \frac{1}{1-\norm{G_p(b)}_{p\rightarrow p}\norm{H_p(b)}_{p\rightarrow p}}
       \norm f_{\W_{\K}^{-1-1/p',p}};
\end{equation}
\item Let $\{b_n\}_n$ be any admissible regularization of $b$ as is introduced in Section \ref{sec:regularization}, and let $f \in C_c^\infty(\mathbb R^{1+2d})$. Then $\eta_n$ defined by \eqref{eq:main-representation} (with $b_n$ in place of $b$) is a classical solution to the equation
$$
 \A\eta_n-\divv(b_n\eta_n)=f,
$$
and
$$
\eta_n \rightarrow \eta \quad \text{ in } \W_{\K}^{1/p,p} \text{ as } n \rightarrow \infty.
$$
Thus, $\eta$ does not depend on the choice of an admissible regularization $\{b_n\}$ of $b$.
\item If
	\begin{equation}\label{eq:point-p-range}
		1<p<\frac{Q-1}{Q-2}=\frac{4d+1}{4d},
	\end{equation}
	then, for every $\zeta_0=(t_0,x_0,v_0)\in\R^{1+2d}$,
	the following Cauchy problem for the kinetic
	Fokker--Planck equation
	\begin{equation*}
		\begin{cases}
			\lambda\eta+\partial_t\eta+v\cdot\nabla_x\eta
			-\Delta_v\eta-\operatorname{div}_v(b\eta)=0,
			& t>t_0,\quad (x,v)\in\R^{2d},\\[1mm]
			\eta|_{t=t_0}=\delta_{\{(x_0,v_0)\}},
		\end{cases}
	\end{equation*}
	has a unique weak solution $\eta$ whose extension by zero
	to $t<t_0$ belongs to $\W_{\K}^{1/p,p}$.
	This extension corresponds to choosing
	$f=\delta_{\{\zeta_0\}}$ in~(i), and the convergence assertion
	in~(ii) remains valid.
	The resulting
	solution gives the forward fundamental solution issued from
	$\zeta_0$.
\end{enumerate}
\end{theorem}

The second assertion can be viewed as an approximation uniqueness result.

\begin{definition}\label{def:weak}
Let $f\in\W_{\K}^{1/p-2,p}$.  A function
$\eta\in\W_{\K}^{1/p,p}$ is a $\W_{\K}^{1/p,p}$-weak solution of
\begin{equation}\label{eq:forward}
 \A\eta-\divv(b\eta)=f
\end{equation}
if
\begin{equation}\label{eq:weak-distribution}
 \pair{\eta}{\A^*\varphi}
 +\int_{\R^{1+2d}} b\eta\cdot\nabla_v\varphi
 =\pair f\varphi
 \qquad\text{for every }\varphi\in\cS(\R^{1+2d}).
\end{equation}
\end{definition}

We may highlight that the second term on the left hand side of \eqref{eq:weak-distribution} is finite. Indeed, write
$h=\A^{1/(2p)}\eta\in L^p$.  Then, for $H_p$ as in \eqref{eq:Hp},
$\abs b^{1/p}\eta=\abs b^{1/p}\A^{-1/(2p)}(\A^{1/(2p)}\eta)=H_p(b) h\in L^p$.
Since $b\in L^1_{\mathrm{loc}}$, H\"older's inequality 
 gives $b\eta\in L^1_{\mathrm{loc}}$ and further
$$
\int_{\R^{1+2d}}|b\eta\cdot \nabla_v\varphi|\le \int_{\R^{1+2d}}|b||b|^{-1/p} |\nabla_v\varphi||H_p(b)h|\le \Big(\int_{\R^{1+2d}}\underbrace{|b|^{p'(1-1/p)}}_{=|b|}|\nabla_v\varphi|^{p'}\Big)^{1/p'} \|H_p(b)h\|_p.
$$

Using the definition of the Morrey norm and taking into account that the drift also has the bounded component, we conclude at most polynomial
growth of $\int_{B_R(0)}|b|$ as $R\to\infty$.
Therefore, $\int |b||\nabla_v\varphi|^{p'}<\infty$ for
$\varphi\in\cS$, and the same weighted H\"older inequality proves
absolute convergence of the drift pairing in
\eqref{eq:weak-distribution}.

\begin{lemma}[Factorization of the weak equation]\label{lem:factorization}
Let \eqref{eq:smallness} hold.  A function
$\eta\in\W_{\K}^{1/p,p}$ is a weak solution of \eqref{eq:forward} if and only
if $h=\A^{1/(2p)}\eta$ satisfies
\begin{equation}\label{eq:factorized-equation}
 (I+G_p(b) H_p(b))h=\A^{-(1-1/(2p))}f
 \qquad\text{in }L^p.
\end{equation}
\end{lemma}

\begin{proof}
We have
\[
 \eta=\A^{-1/(2p)}h,
 \qquad b\eta=b^{\frac{1}{p'}}H_ph.
\]
Put $a=1-1/(2p)$.  The definition of $G_p(b)$ by duality, tested against
Schwartz functions, gives
\[
 \A^{-a}\divv(b\eta)=-G_p(b)H_p(b)h\in L^p.
\]
The fractional powers preserve Schwartz space, so this is an identity
of tempered distributions and it proves
$\divv(b\eta)\in \A^{a}L^p=\W_{\K}^{1/p-2,p}$, with
\[
 \|\divv(b\eta)\|_{\W_{\K}^{1/p-2,p}}
 \le\|G_p(b)\|_{p\rightarrow p}\,\|H_p(b)\|_{p\rightarrow p}\,\|\eta\|_{\W_{\K}^{1/p,p}}.
\]
We can therefore apply
\(\A^{-(1-1/(2p))}\) to \eqref{eq:forward}.  Since
\((1-1/(2p))+1/(2p)=1\), the first term becomes \(h\), while the second
becomes
$G_p(b)H_p(b)h$ by \eqref{eq:Gp}.  This proves \eqref{eq:factorized-equation}.
All operations are continuous between the Bessel potential spaces introduced
in Section \ref{sec:spaces}, so this argument is reversible.
\end{proof}

Define
\begin{equation}\label{eq:main-result-backward-factors}
 R_p(b)=\abs b^{1/p}\nabla_v\A^{-1/2-1/(2p)},\qquad
 Q_p(b)=\A^{-1/(2p')}\bigl(b^{\frac{1}{p'}}\cdot\bigr).
\end{equation}
Definition \ref{def:admissible}, i.e.\,the notion of an admissible regularization of $b$, and Proposition \ref{prop:regularization}, extend in a straightforward manner to the operators $R_p$, $Q_p$.

\begin{theorem}[Backward PDE theory]\label{thm:backward}
Assume $\norm{R_p(b)}_{p\rightarrow p}\norm{Q_p(b)}_{p\rightarrow p}<1$.  For every
$f\in\W_{\K}^{-1/p',p}$, the equation
\[
 \A u+b\cdot\nabla_vu=f
\]
has a unique weak solution $u\in\W_{\K}^{1+1/p,p}$ characterized by
\begin{equation}\label{eq:backward-representation}
 u=\A^{-1/2-1/(2p)}
      (I+Q_pR_p)^{-1}\A^{-1/(2p')}f.
\end{equation}
If $p>Q-1$, then $1+1/p>Q/p$, and the kinetic Sobolev embedding makes this
solution continuous, possibly after a modification on a measure zero set. 
\end{theorem}

The notion of weak solution for  $(\A +b\cdot\nabla_v)u=f$ is analogous to the one in Definition \ref{def:weak}. The last statement is the pointwise regularity that we will use to establish the
Krylov estimates for the kinetic SDE.

\subsection{Weak solutions of the kinetic SDE}

We now apply the backward PDE theory to
\begin{equation}\label{eq:kinetic-sde}
 \begin{split}
  X_t&=x+\int_s^t V_r\,dr,\\
  V_t&=v-\int_s^t b(r,X_r,V_r)\,dr+\sqrt2\,(W_t-W_s),
  \qquad s\le t\le T.
 \end{split}
\end{equation}

Let $\Omega_{s,T}:=C([s,T], \R^{2d})$, 
$Z_t=(X_t,V_t)$ be the related canonical 
 process, and let $\mathcal F_t$ be the completed (with the usual conditions) version of $\sigma(Z_r:s\le r\le t)$.

\begin{definition}
\label{def:weak-sol}
Given $s\in[0,T]$, $z=(x,v)$, a weak
solution of \eqref{eq:kinetic-sde} is a probability
measure on $\Omega_{s,T}$, such that the canonical process $Z$ satisfies \eqref{eq:kinetic-sde}, for some standard $\mathcal F_t$-Brownian motion $\{W_t-W_s\}_{s\le t\le T}$ and such that
$\int_s^T\abs{b(r,Z_r)}\,dr \in L^1(\Omega_{s,T})$.
\end{definition}

The wellposedness of this weak solution is equivalently (e.g. \cite{StrookVaradhan-06}, \cite{EthierKurtz-05}) given by the martingale problem related to \eqref{eq:kinetic-sde} which states as follows.
\begin{definition}\label{def:martingale-problem} For $0\le s\le T$ and $z=(x,v)\in\mathbb R^{2d}$, a probability measure $P_{s,z}$ on $(\Omega_{s,T},\{\mathcal F_t\}_{t\in[s,T]})$ is said to be a solution of the martingale problem for $b$ at the initial time-state $(s,z)$ if under $\mathbb P_{s,z}$, $Z_s=z$ a.s., $\mathbb E_{s,z}[\int_s^T|b(r,Z_r)|\,dr]:=\mathbb E^{\mathbb P_{s,z}}[\int_s^T|b(r,Z_r)|\,dr]<\infty$ and, for all $\varphi\in C^1_b([s,T];\mathcal C^2_c(\mathbb R^{2d}))$ and for
$$
		L_r(\varphi)(r,x,v):=\partial_r\varphi(r,x,v)+v\cdot \nabla_x\varphi(r,x,v)-b(r,x,v)\cdot\nabla_v\varphi(r,x,v)+\triangle_v\varphi(r,x,v),
		$$	
the process
		$$
		M^\varphi_t=\varphi(t,Z_t)-\varphi(s,z)-\int_s^tL_r\varphi(r,Z_r)\,dr,\qquad s\le t\le T,
		$$
		is an $\mathcal F_t$-martingale. 
\end{definition}

We now call
an exponent $p$ \textit{backward admissible} when the adjoint counterparts of
the two operators in \eqref{eq:main-result-backward-factors} have product norm
strictly smaller than one; see \eqref{eq:adjoint-backward-factors} and
\eqref{eq:backward-smallness} in Section \ref{sec:sde}. This is again guaranteed by sufficiently
small $\norm{b_{\mathrm s}}_{\M^q_{\K}}$, after increasing $\lambda$ when a
bounded component is present.

We first state a special version of our SDE weak well-posedness result whose proof is shorter and demonstrates the
main idea. Here the drift has no separate bounded
component and satisfies
\begin{equation}\label{eq:HK1}
 b=b_{\mathrm s}\in\M^q_{\K}\cap
 L^{1+\varepsilon_0}(\R^{1+2d},\R^d)
 \quad\text{for some }\varepsilon_0>0.\tag{$H_{\K,\eps_0}$}
\end{equation}
Recall the notation $Q=4d+2$.

\begin{theorem}[Weak SDE theory under global higher integrability of the drift]
\label{thm:sde}
Let $d\ge2$, $T>0$, and assume \eqref{eq:HK1}.  Let $p$ and $p_1$ be
backward admissible and suppose
\begin{equation}\label{eq:sde-exponents}
 p>Q-1,\qquad p_1>p\frac q{q-1}.
\end{equation}
Set $\varepsilon_*=\varepsilon_0/p$.  Then, for every starting point
$(s,z)$, the following hold.
\begin{enumerate}[label=\textup{(\roman*)}]
\item There exists a weak solution $\mathbb P_{s,z}$ of
\eqref{eq:kinetic-sde}.  It satisfies
\begin{equation}\label{eq:drift-integrability}
 \mathbb E_{s,z}\int_s^T\abs{b(r,Z_r)}^{1+\varepsilon_*}\,dr\le C_T,
\end{equation}
and hence, for all $s\le u<t\le T$,
\[
 \mathbb E_{s,z}\int_u^t\abs{b(r,Z_r)}\,dr
 \le C_T^{1/(1+\varepsilon_*)}
      \abs{t-u}^{\varepsilon_*/(1+\varepsilon_*)}.
\]
It also satisfies the unweighted and drift-weighted kinetic Krylov bounds
\begin{align*}
	\mathbb E_{s,z}\int_s^T\abs{h(r,Z_r)}\,dr
	&\le C_T\norm{\mathbf1_{[s,T]}h}_p,\\
	\mathbb E_{s,z}\int_s^T\abs{b(r,Z_r)h(r,Z_r)}\,dr
	&\le C_T\norm{\mathbf1_{[s,T]}\abs b^{1/p}h}_p.
\end{align*}

\item The transition operators
$$P^{s,t}g(z)=\mathbb E_{s,z}g(Z_t)$$ form a conservative, strongly
continuous backward Feller propagator on $C_\infty(\R^{2d})$, and so $\{\mathbb P_{s,z}\}$
is a strong Markov family.

\item Weak uniqueness holds among solutions satisfying the two Krylov bounds
\begin{align*}
	\mathbb E^{\mathbb Q}_{s,z}\int_s^T\abs{h(r,Z_r)}\,dr
	&\le K_{\mathbb Q}\norm{\mathbf1_{[s,T]}h}_p,\\
	\mathbb E^{\mathbb Q}_{s,z}\int_s^T
	\abs{b(r,Z_r)h(r,Z_r)}\,dr
	&\le K_{\mathbb Q}
	\norm{\mathbf1_{[s,T]}\abs b^{1/p}h}_p,
\end{align*}
for $K_{\mathbb Q}$ a constant independent of $h$.

\item If the smallness condition also holds at some
$$1<r<(Q-1)/(Q-2),$$ let $\eta_{s,z}^{\lambda}$ be the
point-source solution given by Theorem \ref{thm:main} (taking $f=\delta_{\{(s,z)\}}$). Then the discounted
space--time occupation measure $(u,A)\in[s,T]\times\mathcal B(\mathbb R^{2d})
\longmapsto
\int_s^u e^{-\lambda(t-s)}
\mathbb P_{s,z}(Z_t\in A)\,dt$  satisfies
\begin{equation*}
e^{-\lambda(t-s)}\mathbb P_{s,z}(Z_t\in d\zeta)\,dt
=\eta_{s,z}^{\lambda}(t,\zeta)\,dt\,d\zeta.
\end{equation*}
In particular, for a.e.\,$t\in(s,T)$, the law of $Z_t$
has density $e^{\lambda(t-s)}\eta_{s,z}^{\lambda}(t,\cdot)$.
\end{enumerate}
\end{theorem}

In the next theorem we remove the global integrability condition \eqref{eq:HK1}  with only a few additional efforts (see Section \ref{sec:local-removal}). In particular, we will have to use some polynomial weights. The bounded part of the drift will be absorbed by the choice of sufficiently large $\lambda$ in the PDE theory of the corresponding Kolmogorov backward equation. Introducing 
\begin{align*}
		\mathfrak r_{p,\lambda}^{\rho}(b)
		&:=C^{R,\mathrm s}_{d,p,q,\kappa}
		\norm{b_{\mathrm s}}_{\M^q_{\K}}^{1/p}
		+C^{R,\mathrm b}_{d,p,\kappa}\lambda^{-1/(2p)}
		\norm{b_{\mathrm b}}_\infty^{1/p},\\
		\mathfrak q_{p,\lambda}^{\rho}(b)
		&:=C^{Q,\mathrm s}_{d,p,q,\kappa}
		\norm{b_{\mathrm s}}_{\M^q_{\K}}^{1/p'}
		+C^{Q,\mathrm b}_{d,p,\kappa}\lambda^{-1/(2p')}
		\norm{b_{\mathrm b}}_\infty^{1/p'},
\end{align*}
for some constants $C^{R,\mathrm s}_{d,p,q,\kappa}$, $C^{R,\mathrm b}_{d,p,\kappa}$, $C^{Q,\mathrm s}_{d,p,q,\kappa}$ and $C^{Q,\mathrm b}_{d,p,\kappa}$ (see Lemma \ref{lem:weighted-backward}), we have the following wellposedness result.

\begin{theorem}[Weak SDE theory without global integrability]
\label{thm:sde-local}
Let $b$ satisfy \eqref{eq:HK}.  Suppose $p>Q-1$ is backward admissible
and satisfies the centered weighted product condition
$$
\mathfrak q_{p,\lambda}^{\rho}(b)
\mathfrak r_{p,\lambda}^{\rho}(b)<1,
$$
and let $p_1$ be backward admissible with
\[
 p_1>p\frac q{q-1}.
\]
Choose
\begin{equation*}
 0<\varepsilon_*<\min\left\{1,\frac{q-1}{p}\right\}.
\end{equation*}
Then all four conclusions $(i)$ to $(iv)$ of Theorem \ref{thm:sde} hold, with this value of
$\varepsilon_*$.  In particular, weak existence and the tightness
argument require no global $L^a$ assumption on $b$.  No smallness is imposed
on $b_{\mathrm b}$.
\end{theorem}

\bigskip

\section{Counterexample to weak existence for the kinetic SDE (blow-up)}
\label{sec:counterexample}

Weak existence for the SDE \eqref{eq:kinetic-sde} may fail when the kinetic Morrey norm $\|b_{\mathrm s}\|_{\M^q_\K}$ of the drift is too large. The following position-velocity dependent modification of the
classical Hardy drift counterexample, cf.\
\cite[Example~7.4]{BFGM} (in a non-degenerate case), strongly concentrates the dynamics on the singularity set 
$\{v=x\}$. 

\begin{proposition}[An inward singularity along $\{v=x\}$]
\label{prop:hardy-nonexistence}
Let $d\ge2$ and $\kappa>0$, and define
\begin{equation*}
 b_\kappa(t,x,v)=
 \begin{cases}
  \displaystyle\kappa\frac{v-x}{|v-x|^2},&v\ne x,\\[1ex]
  0,&v=x.
 \end{cases}
\end{equation*}
Then $b_\kappa\in\M^q_{\K}$ for every $1<q<d$, with
\begin{equation}\label{eq:hardy-velocity-morrey}
 \norm{b_\kappa}_{\M^q_{\K}}\le C_{d,q}\kappa.
\end{equation}
If $\kappa\ge d-1$, then
the SDE \eqref{eq:kinetic-sde} with $b=b_\kappa$ and initial data
$(s,x,v)=(0,0,0)$ has no weak solution on any interval $[0,T]$,
$T>0$.

The same nonexistence result 
holds for the compactly
supported drift
\begin{equation*}
 b_{\kappa,\chi}=\chi b_\kappa,
 \qquad
 \chi\in C_c^\infty(\R^{1+2d}),\quad 0\le\chi\le1,
\end{equation*}
provided that $\chi=1$ on a neighborhood of $(0,0,0)$.
This drift satisfies
\begin{equation}\label{eq:localized-hardy-membership}
 b_{\kappa,\chi}\in\M^q_{\K}\cap
 L^{1+\varepsilon_0}(\R^{1+2d}, \R^d),
 \qquad
 1<q<d,\quad 0<\varepsilon_0<d-1.
\end{equation}
\end{proposition}

\begin{remark}
The physical interpretation requires us to consider the drift defined in terms of $v-\omega x$ where $\omega$ has units of inverse time, rather than $v-x$, but we will skip this. The crux of the argument is that, in any putative weak solution,
the singular drift $b_\kappa$ forces $X$ and $V$ to coincide on
a nontrivial time interval, which is incompatible with the finite
total variation of $X$ and the infinite total variation of $V$
due to its Brownian component.
\end{remark}

\begin{proof}
\emph{Morrey norm and global integrability.}
Write $\zeta=(t,x,v)$. By the definition of $d_*$, every
$(s,y,w)\in B_r(\zeta)$ satisfies
\[
 |s-t|<r^2,\qquad
 |w-v|<r,\qquad
 |y-x-(s-t)v|<r^3.
\]
For $q<d$, the elementary estimate
\[
 \int_{|w-v|<r}|w-y|^{-q}\,dw
 \le C_{d,q}r^{d-q}
\]
holds uniformly in $y,v$ and $r>0$.
Indeed, if $|v-y|<2r$, the integration domain is contained in
$\{|w-y|<3r\}$. If $|v-y|\ge2r$, then $|w-y|\ge r$
throughout the integration domain.

We integrate first in $w$, using this estimate, and then in
$y$ and $s$. It follows that
\[
 \begin{aligned}
 \int_{B_r(\zeta)}|b_\kappa|^q
 &\le
 \kappa^q
 \int_{t-r^2}^{t+r^2}
 \int_{|y-x-(s-t)v|<r^3}
 \int_{|w-v|<r}|w-y|^{-q}\,dw\,dy\,ds\\
 &\le C_{d,q}\kappa^q r^{2+3d}r^{d-q}
 =C_{d,q}\kappa^q r^{Q-q}.
 \end{aligned}
\]
Together with \eqref{eq:ball-volume}, this proves
\eqref{eq:hardy-velocity-morrey}.
Since $|b_{\kappa,\chi}|\le|b_\kappa|$, the same Morrey bound
holds for $b_{\kappa,\chi}$.
The change of variables $(x,v)\mapsto(x,v-x)$ has Jacobian
one, and $|v-x|^{-a}$ is locally integrable for every $a<d$.
Compact support therefore proves
\eqref{eq:localized-hardy-membership}.

\emph{Localization of a hypothetical weak solution.}
Assume $\kappa\ge d-1$.
We prove nonexistence for $b_{\kappa,\chi}$; the proof for
$b_\kappa$ is identical with $\chi\equiv1$.
Suppose that a weak solution exists on $[0,T]$.
Choose $h\in(0,T]$ and $a\in(0,1]$ such that $\chi=1$ whenever
\[
 0\le t\le h,\qquad |x|\le a,\qquad |v|\le a.
\]
Define the stopping time
\[
 \tau=h\wedge\inf\{t\in[0,T]:
                    |X_t|\ge a\ \hbox{or}\ |V_t|\ge a\},
\]
with the convention $\inf\emptyset =\infty$. 
Continuity and the initial condition
 $(X_0,V_0)=(0,0)$ imply $\tau>0$
almost surely.

Set $Y_t=V_t-X_t$ and $R_t=|Y_t|$.
Since $dX_t=V_t\,dt$, up to time $\tau$ we have
\begin{equation}\label{eq:hardy-stopped-velocity}
 dY_t
 =-\kappa\frac{Y_t}{|Y_t|^2}
       \mathbf1_{\{Y_t\ne0\}}\,dt
   -V_t\,dt+\sqrt2\,dW_t,
 \qquad Y_0=0.
\end{equation}
Thus the change of variables introduces the additional
drift $-V_t$, which is bounded up to $\tau$.
The required integrability $\mathbb E\int_0^t|b_{\kappa,\chi}(r,X_r,V_r)|\,dr <\infty$ 
implies
\begin{equation*}
 A_t:=
 \int_0^{t\wedge\tau}
       \frac{\mathbf1_{\{R_r>0\}}}{R_r}\,dr<\infty
 \quad\hbox{almost surely},\qquad 0\le t\le T.
\end{equation*}
All quotients with denominator $R_r$ are assigned the value
zero when $R_r=0$.

\emph{Zero occupation time at the singularity.}
Each coordinate of $Y=V-X$ is a continuous semimartingale.
Since $X$ has finite variation, for each $j$-coordinate, 
$
 [Y^{(j)}]_t=[V^{(j)}]_t=2t.
$
The occupation-times formula (see e.g. \cite[Corollary 1, pp. 219]{Protter-05}) for the first coordinate gives
\[
 2\int_0^T\mathbf1_{\{Y_r^{(1)}=0\}}\,dr
 =
 \int_{\R}\mathbf1_{\{u=0\}}
          L_T^u(Y^{(1)})\,du=0.
\]
Because $\{Y_r=0\}\subset\{Y_r^{(1)}=0\}$, we obtain, for any finite time-horizon $T$,
\begin{equation}\label{eq:hardy-zero-occupation}
 \int_0^T\mathbf1_{\{R_r=0\}}\,dr=0
 \quad\hbox{almost surely}.
\end{equation}

\emph{Arriving at a contradiction.} We will now derive a contradiction with \eqref{eq:hardy-zero-occupation}. To  this end, for $m\ge1$, put
\[
 \sigma_m
 :=\tau\wedge\inf\{t\in[0,T]:A_t\ge m\},
\]
and let us establish (through It\^o's expansion) that $\int_0^{\sigma_m}\mathbf1_{\{R_r=0\}}\,dr>0$ almost surely.

The process $A$ is continuous, nondecreasing, and starts
from zero. Hence $\sigma_m>0$ is non-decreasing almost surely and
$A_{\sigma_m}\le m$.

The extra bounded drift $-V_{t}$ in
\eqref{eq:hardy-stopped-velocity} is handled by the increasing
concave function $r\mapsto1-e^{-r}$.
For $\varepsilon>0$, define its smooth radial approximation as
\[
 \phi_\varepsilon(y)
 :=(|y|^2+\varepsilon^2)^{1/2},
 \qquad
 \psi_\varepsilon(y)
 :=1-e^{-\phi_\varepsilon(y)}.
\]
For $y\ne0$ and $w\in\R^d$, direct differentiation yields
\begin{equation*}
 \begin{aligned}
 H_\varepsilon(y,w)
 &:=
 \left(
 \Delta_y-
 \left(\kappa\frac{y}{|y|^2}+w\right)\cdot\nabla_y
 \right)\psi_\varepsilon(y)\\
 &=
 e^{-\phi_\varepsilon(y)}
 \left(
 \frac{d-\kappa}{\phi_\varepsilon(y)}
 -\frac{|y|^2}{\phi_\varepsilon(y)^3}
 -\frac{y\cdot w}{\phi_\varepsilon(y)}
 -\frac{|y|^2}{\phi_\varepsilon(y)^2}
 \right).
 \end{aligned}
\end{equation*}
In particular, for $y\ne0$ and $|w|\le1$,
\[
 |H_\varepsilon(y,w)|
 \le\frac{|d-\kappa|+1}{|y|}+2,
\]
and
\[
 H_\varepsilon(y,w)\rightarrow 
 e^{-|y|}
 \left(
 \frac{d-1-\kappa}{|y|}
 -\frac{y}{|y|}\cdot w-1
 \right).
\]
At $y=0$, the actual drift coefficient in It\^o's formula
for $\psi_\varepsilon(Y_t)$ equals
$d e^{-\varepsilon}/\varepsilon$:
the singular drift is defined to be zero there and
$\nabla\psi_\varepsilon(0)=0$.
Its contribution to the time integral vanishes by
\eqref{eq:hardy-zero-occupation}.

Apply It\^o's formula to
$\psi_\varepsilon(Y_{t\wedge\sigma_m})$.
Since $|\nabla\psi_\varepsilon|\le1$, the stochastic integral
has expectation zero and we have
\[
 \mathbb E\psi_\varepsilon(Y_{t\wedge\sigma_m})
 =
 1-e^{-\varepsilon}
 +
 \mathbb E\int_0^{t\wedge\sigma_m}
       \mathbf1_{\{R_r>0\}}
       H_\varepsilon(Y_r,V_r)\,dr.
\]
Before time $\sigma_m$, we have $|V_r|\le a\le1$, and so
\[
 \mathbb E\int_0^{t\wedge\sigma_m}\sup_{\eps}\mathbf1_{\{R_r>0\}}
 	|H_\varepsilon(Y_r,V_r)|dr\le \mathbb E\int_0^{t\wedge\sigma_m}
 \left(
 (|d-\kappa|+1)
       \frac{\mathbf1_{\{R_r>0\}}}{R_r}+2
 \right)dr
 \le (|d-\kappa|+1)m+2T.
\]
Thus dominated convergence applies to the drift integral and since 
\[
 \sup_y
 \left|\psi_\varepsilon(y)-(1-e^{-|y|})\right|
 \le\varepsilon,
\]
letting $\varepsilon\downarrow0$ therefore gives
\begin{equation}\label{eq:hardy-radial-contradiction}
 \begin{aligned}
 \mathbb E\bigl[1-e^{-R_{t\wedge\sigma_m}}\bigr]
 &=
 \mathbb E\int_0^{t\wedge\sigma_m}
 \mathbf1_{\{R_r>0\}}e^{-R_r}
 \left(
 \frac{d-1-\kappa}{R_r}
 -\frac{Y_r}{R_r}\cdot V_r-1
 \right)dr\\
 &\le0.
 \end{aligned}
\end{equation}
Indeed, $\kappa\ge d-1$ and $|V_r|\le1$ imply that the
expression in parentheses is nonpositive.
The term $-1$, coming from the concavity of the radial test
function, absorbs the additional drift $-V_r$.

The left-hand side of
\eqref{eq:hardy-radial-contradiction} is nonnegative.
It follows that $R_{t\wedge\sigma_m}=0$ almost surely for
every $t\in[0,T]$.
Taking a countable dense set of times and using continuity,
we conclude that the stopped path is identically zero.
Therefore,
\[
 \int_0^{\sigma_m}\mathbf1_{\{R_r=0\}}\,dr
 =\sigma_m>0
 \quad\hbox{almost surely},
\]
contrary to \eqref{eq:hardy-zero-occupation}.
This proves nonexistence for
$\kappa \geq d-1$.
\end{proof}

\bigskip

\section{Main Rellich-type estimate from kinetic Morrey class via Adams estimates}
\label{sec:morrey}

Recall that \(Q=4d+2\) is the homogeneous dimension fixed in \eqref{eq:Q}.

We will need the following maximal function and estimates.  For $0\le\beta<Q$, let
\begin{equation}\label{eq:fractional-maximal}
 \mathsf M_\beta h(\zeta)
 :=\sup_{r>0}r^\beta
  \frac1{\abs{B_r(\zeta)}}\int_{B_r(\zeta)}\abs{h(\xi)}\,d\xi,
 \qquad \mathsf M:=\mathsf M_0.
\end{equation}

\begin{proposition}
	\label{lem:maximal-function}
On $(\mathbb R^{1+2d}, d_\ast, dtdxdv
)$, the following assertions hold:
	
	\begin{enumerate}[label=\textup{(\roman*)}]
		\item The operator $\mathsf M$ is of weak type $(1,1)$ and is
		bounded on $L^s$ for every $1<s\le\infty$.
		
		\item If $1<q<\infty$ and $1\le a<q$, and $f\in\mathcal M_{\mathcal K}^q$, then
		$$
		\|
		(\mathsf M(|f|^a) )^{\frac{1}{a}}
	\|_{\mathcal M_{\mathcal K}^q}
		\le
		C_{d,q,a}\|f\|_{\mathcal M_{\mathcal K}^q}.
		$$

		\item If \(f\in L_{loc}^1\) and
		\(\mathsf Mf<\infty\) almost everywhere, then for \(0<\delta<1\), 
		\((\mathsf Mf)^\delta\in A_1\), that is
		$$
		\mathsf M((\mathsf Mf)^\delta) \le C_{d,\delta} (\mathsf Mf)^\delta \quad \text{a.e.}
		$$
	
	\end{enumerate}
\end{proposition}

\begin{proof}
 \textup{(i)} and \textup{(iii)} are standard results from maximal-function theory on spaces of homogeneous type; see \cite[pp.71]{CW} and \cite[Lemma~3.4]{LLW}, respectively. 

\textup{(ii)} See \cite[Corollary~4.2]{AM} for a related Morrey-space estimate. We give a direct proof for $\mathcal M_{\mathcal K}^q$.
	Fix
	$B=B_r(\zeta_0)$  and decompose
	$$
	|f|^a
	=
	|f|^a\mathbf 1_{2B}
	+
	|f|^a\mathbf 1_{(2B)^c}
	=:F_1+F_2.
	$$
	Since $a<q$, the $L^{\frac{q}{a}}$-boundedness of $\mathsf M$   gives
	$$
	\frac{1}{|B|}\int_B(\mathsf MF_1)^{\frac{q}{a}}
	\le
	C_1 \frac{1}{|B|} \int_{2B}|f|^q
	\le
	C_2 \| f\|_{\M_\K^q}^q r^{-q}.
	$$
	If $\zeta\in B$, every ball centered at $\zeta$ that meets
	$(2 B)^c$ has radius at least $cr$. Then, by H\"older's
	inequality and the definition of $\mathcal M_{\mathcal K}^q$,
	$$
	\mathsf MF_2(\zeta)
	\le
	\sup_{\rho\ge cr}
	\Big(
	\frac{1}{|B_\rho(\zeta)|} \int_{B_\rho(\zeta)}|f|^q
	\Big)^{\frac{a}{q}}
	\le
	C_3 \| f\|_{\M_\K^q}^a r^{-a},
	$$
	whence
	$$
		\frac{1}{|B|}
	\int_B(\mathsf MF_2)^{\frac{q}{a}}
	\le
	C_4 \| f\|_{\M_\K^q}^q r^{-q}.
	$$
	Combining the two estimates yields
	$$
	r 
	\Big(
		\frac{1}{|B|}
	\int_B
	(\mathsf M(|f|^a))^{\frac{q}{a}}
	\Big)^{\frac{1}{q}}
	\le C_5 \| f\|_{\M_\K^q}.
	$$
	Taking the supremum over $B$ proves \textup{(ii)}.
\end{proof}

\begin{proposition}[Hedberg-type inequality]\label{lem:hedberg}
	If $0<\alpha<\beta<Q$, then
	\begin{equation}\label{eq:hedberg}
	I_\alpha f
	\le
	C_{\alpha,\beta,d}
	\bigl(\mathsf M_\beta f\bigr)^{\frac{\alpha}{\beta}}
	\bigl(\mathsf M f\bigr)^{1-\frac{\alpha}{\beta}}.
\end{equation}
\end{proposition}

\begin{proof}
Fix \(\zeta\in \mathbb R^{1+2d}\) and \(R>0\). 	If either maximal function is infinite, the estimate is immediate.
If either one vanishes, then \(f=0\) almost everywhere. We may therefore assume that
\[
0<\mathsf M f(\zeta),\,\mathsf M_\beta f(\zeta)<\infty.
\]
Let us split the integral defining \(I_\alpha f(\zeta)\) into
the regions \(d_*(\zeta,\xi)<R\) and \(d_*(\zeta,\xi)\ge R\).
For the near region, using a dyadic decomposition,
\begin{align*}
	\int_{d_*(\zeta,\xi)<R}
	\frac{|f(\xi)|}{d_*(\zeta,\xi)^{Q-\alpha}}\,d\xi 
	&=	\sum_{j=0}^{\infty}
	\int_{2^{-j-1}R \le d_*(\zeta,\xi)<2^{-j}R}
	\frac{|f(\xi)|}{d_*(\zeta,\xi)^{Q-\alpha}}\,d\xi \\
	&\le
	\sum_{j=0}^{\infty}
	(2^{-j-1}R)^{\alpha-Q}
	\int_{B_{2^{-j}R}(\zeta)}|f(\xi)|\,d\xi\\
	&\le C_1 R^\alpha \mathsf M f(\zeta).
\end{align*}
For the far region,	similarly,
\begin{align*}
		\int_{d_*(\zeta,\xi)\ge R}
	\frac{|f(\xi)|}{d_*(\zeta,\xi)^{Q-\alpha}}\,d\xi
	&=	\sum_{j=0}^{\infty}
	\int_{2^{j}R \le d_*(\zeta,\xi)<2^{j+1}R}
	\frac{|f(\xi)|}{d_*(\zeta,\xi)^{Q-\alpha}}\,d\xi \\
	&\le
	\sum_{j=0}^{\infty}
	(2^jR)^{\alpha-Q}
	\int_{B_{2^{j+1}R}(\zeta)}|f(\xi)|\,d\xi\\
	&\le
	C_2 	\sum_{j=0}^{\infty}
	(2^j R)^{\alpha-\beta} (2^{j+1}R)^{\beta} \frac{1}{|B_{2^{j+1}R} (\zeta)|} 
	\int_{B_{2^{j+1}R} (\zeta)} |f(\xi)|\,d\xi\\
	&\le C_2 R^{\alpha-\beta}\mathsf M_\beta f(\zeta),
\end{align*}
where the last series converges because \(\alpha<\beta\). Hence
\[
I_\alpha f(\zeta)
\le
C_1R^\alpha \mathsf M f(\zeta)
+
C_2R^{\alpha-\beta}\mathsf M_\beta f(\zeta).
\]
Taking
$
R=
\left(
\frac{\mathsf M_\beta f(\zeta)}{\mathsf Mf(\zeta)}
\right)^{\frac{1}{\beta}}
$
gives \eqref{eq:hedberg}.
\end{proof}

\begin{proposition}[Kinetic Adams inequality]\label{prop:adams}
	Let \(1<p<\infty\), $q>1$, and $g\in\M^q_{\K}$. Then
\begin{equation*}
	\int_{\mathbb R^{1+2d}}
	|g|\, |I_{\frac{1}{p}}f|^p
	\le
	C_{p,q,d}\norm g_{\M^q_{\K}}\|f\|_{p}^p
\end{equation*}
or, equivalently, 
	\begin{equation}\label{eq:adams-operator}
		\norm{\abs g^{\frac{1}{p}}I_{\frac{1}{p}}}_{p\rightarrow p}
		\le C_{d,p,q}\norm g_{\M^q_{\K}}^{\frac{1}{p}}.
	\end{equation}

\end{proposition}

\begin{remark}
As mentioned in the introduction, the maximal-function arguments that we use to extend the proof of the Adams estimate \cite[Theorem 7.3]{Adams2} to the kinetic setting  build on 
\cite{AM, CoifmanRochberg, DH, Heinonen} who consider a more general setup of 
spaces of homogeneous type.

It should also be noted that, for \(p=2\), the preceding estimate already appeared in \cite[Appendix~C]{L}, where the author used a fixed-point argument to establish the existence and uniqueness of solutions to certain nonlinear equations. The proof of the estimate there, however, follows a different approach from that of Adams and ours.
 However, we need this estimate for the full range $1<p<\infty$. Indeed, we need to be able to (a) take $p$ large when handling the backward equation, in order to apply the Sobolev embedding theorem Proposition~\ref{prop:sobolev} to establish Krylov estimates and to construct the Feller propagator; (b) take $p$ close to $1$ when handling the forward equation, in order to have delta-function $\delta_x$ in the initial condition and thus prove uniqueness of the probability density evolving from a fixed point.
 
 We refer to \cite{AX1, Krylov} regarding the theory of Morrey spaces in the non-degenerate setting, and, in particular, in what concerns parabolic Morrey embeddings, to \cite{AX2, Krylov2}.
\end{remark}

\begin{proof}
Since the integral kernel of \(I_{1/p}\) is positive,
it suffices to consider \(f\ge0\).

We first recall the \(A_1\) condition that will be used below.  For a
nonnegative locally integrable function \(w\), write
\[
 [w]_{A_1}
 :=
 \inf\bigl\{C>0:\mathsf Mw\le Cw\ \text{a.e.}\bigr\},
\]
where \(\mathsf M\) is the intrinsic Hardy--Littlewood maximal operator
defined in \eqref{eq:fractional-maximal}. Thus \(w\in A_1\) precisely when
\([w]_{A_1}<\infty\).

We separate the argument into two steps.  We first prove the estimate under
the auxiliary assumption that a suitable power of \(\abs g\) belongs to
\(A_1\). The constant in this intermediate estimate is allowed to depend
on its \(A_1\)-characteristic.  We then remove this assumption by replacing
\(\abs g\) with a Coifman--Rochberg majorant whose \(A_1\)-characteristic is
bounded solely in terms of the fixed exponents and the doubling constant.

\smallskip
\noindent\emph{Step 1: Estimate under the $A_1$ assumption.}
Choose $q_0$ and $\sigma>0$ such that
$$
1<q_0<\min\{q,Q\},
\qquad
1+\sigma p<q_0,
\qquad
\frac{1}{p}+\sigma p'<q_0,
\qquad \sigma p'<1.
$$
Such a choice is possible for every $p>1$ and $q>1$.  Since $q_0<q$,
H\"older's inequality gives
$\|g\|_{\M^{q_0}_{\K}}\le\|g\|_{\M^q_{\K}}$.
The case $g=0$ is immediate, so assume $g\ne0$.  For the moment, suppose
additionally that $|g|^{q_0}\in A_1$, and put
$A_0=[|g|^{q_0}]_{A_1}$.  All constants in this step may depend on
$d,p,q_0,\sigma,A_0$.  We claim that 
\begin{align}\label{weighted_eqn}
	\int_{\mathbb R^{1+2d}}
	|g|\,|I_{\frac{1}{p}}f|^p
	\le C_{d,p,q_0,\sigma,A_0}
	\|g\|_{\M^{q_0}_\K}\|f\|_{p}^p.
\end{align}
Indeed,
for $p\in(1,\infty)$, the symmetry and positivity of the kernel  give
\begin{align}\label{duality}
	\int_{\mathbb R^{1+2d}} |g|\,|I_{\frac{1}{p}}f|^p
	=
	\langle
I_{\frac{1}{p}}
	(|g|\,(I_{\frac{1}{p}}f)^{p-1}),f
	\rangle 
	&\le
	\|
	I_{\frac{1}{p}}
	(|g|\,(I_{\frac{1}{p}}f)^{p-1})
	\|_{p'}
	\|f\|_{p}.
\end{align}
Note that H\"older's inequality yields 
\begin{align}\label{T_est1}
	I_{\frac{1}{p}}(|g|\,(I_{\frac{1}{p}}f)^{p-1})(\zeta)
	&=
	I_{\frac{1}{p}}
	(
	|g|^{\frac{1}{p}+\sigma}
	|g|^{\frac{1}{p'}-\sigma}
	|I_{\frac{1}{p}}f|^{p-1}
	)\notag\\
	&\le
	\big(
	I_{\frac{1}{p}}(|g|^{1+\sigma p})
	\big)^{\frac{1}{p}}
	\big(
	I_{\frac{1}{p}}
	(|g|^{1-\sigma p'}(I_{\frac{1}{p}}f)^p)
	\big)^{\frac{1}{p'}}.
\end{align}
Applying the Hedberg-type inequality with
$\alpha=\frac{1}{p}$ and $\beta=1+\sigma p$, we obtain
\begin{align*}
	I_{\frac{1}{p}}(|g|^{1+\sigma p})
	&\le C_1
	\big(
	\mathsf M_{1+\sigma p}(|g|^{1+\sigma p})
	\big)^{\frac{1}{p(1+\sigma p)}}
	\big(
	\mathsf M(|g|^{1+\sigma p})
	\big)^{1-\frac{1}{p(1+\sigma p)}}.
\end{align*}
Since $1+\sigma p<q_0$, Hölder's inequality on each
intrinsic ball gives 
\begin{align*}
	\mathsf M_{1+\sigma p}(|g|^{1+\sigma p})(\zeta)
	&=
	\sup_{r>0}
	r^{1+\sigma p}
	\frac{1}{|B_r(\zeta)|}
	\int_{B_r(\zeta)}|g|^{1+\sigma p}\\
	&\le
	\sup_{r>0}
	\Big(
	r
	\Big(
	\frac{1}{|B_r(\zeta)|}
	\int_{B_r(\zeta)}|g|^{q_0}
	\Big)^{\frac{1}{q_0}}
	\Big)^{1+\sigma p}\\
	&\le
	\|g\|_{\mathrm \M^{q_0}_\K}^{1+\sigma p}.
\end{align*}
For every $0<\kappa<q_0$, $|g|^{q_0}\in A_1$ gives
\begin{align*}
	\mathsf M(|g|^\kappa)
	\le
	(\mathsf M(|g|^{q_0}))^{\frac{\kappa}{q_0}}
	\le
	[|g|^{q_0}]_{A_1}^{\frac{\kappa}{q_0}}
	|g|^\kappa
	 \qquad\text{a.e.}
\end{align*}
Hence,
\begin{align}	\label{T_est2}
	I_{\frac{1}{p}}(|g|^{1+\sigma p})
	\le C_2
	\|g\|_{ \M^{q_0}_\K}^{\frac{1}{p}}
	|g|^{1+\sigma p-\frac{1}{p}}.
\end{align}
	Similarly, applying the Hedberg-type inequality again we have
\begin{align}\label{T_est3}
	I_{\frac{1}{p}}
	(|g|^{\frac1p+\sigma p'})
	\le
	C_3
	\big(
	\mathsf M_{\frac1p+\sigma p'}
	(|g|^{\frac1p+\sigma p'})
	\big)^{\frac{1}{1+\sigma pp'}}
	\big(
	\mathsf M
	(|g|^{\frac1p+\sigma p'})
	\big)^{1-\frac{1}{1+\sigma pp'}}
	\le
	C_4\|g\|_{\M^{q_0}_\K}^{\frac{1}{p}}
	|g|^{\sigma p'}.
\end{align}

Combining \eqref{T_est1}-\eqref{T_est3}, we obtain
\begin{align*}
	\|
	I_{\frac{1}{p}}(|g| |I_{\frac{1}{p}} f|^{p-1})
	\|_{p'}
	&\le
	\Big(
	\int_{\mathbb R^{1+2d}}
	\big(
	I_{\frac{1}{p}}(|g|^{1+\sigma p})
	\big)^{\frac{p'}{p}}
	I_{\frac{1}{p}}
	\bigl(|g|^{1-\sigma p'}|I_{\frac{1}{p}} f|^p\bigr)
	\Big)^{\frac{1}{p'}}\\
	&\le C_5
	\|g\|_{\M_\K^{q_0}}^{\frac{1}{p^2}}
	\Big(
	\int_{\mathbb R^{1+2d}}
	|g|^{\frac1p+\sigma p'}
	I_{\frac{1}{p}}
	\bigl(|g|^{1-\sigma p'}|I_{\frac{1}{p}} f|^p\bigr)
	\Big)^{\frac{1}{p'}}\\
	&= C_5
	\|g\|_{\M_\K^{q_0}}^{\frac{1}{p^2}}
	\Big(
	\int_{\mathbb R^{1+2d}}
	I_{\frac{1}{p}}
	(|g|^{\frac1p+\sigma p'})
	|g|^{1-\sigma p'}|I_{\frac{1}{p}} f|^p
	\Big)^{\frac{1}{p'}}\\
	&\le C_6
	\|g\|_{\M_\K^{q_0}}^{\frac{1}{p}}
	\Big(
	\int_{\mathbb R^{1+2d}}
	|g||I_{\frac{1}{p}} f|^p
	\Big)^{\frac{1}{p'}}
	.
\end{align*}
Recalling \eqref{duality}, we conclude that
$$
\int_{\mathbb R^{1+2d}}
|g|\,|I_{\frac{1}{p}}f|^p
\le C_6
\|g\|_{\M_\K^{q_0}}^{\frac{1}{p}}
\left(
\int_{\mathbb R^{1+2d}}
|g|\,|I_{\frac{1}{p}}f|^p
\right)^{\frac{1}{p'}}
\|f\|_{p}.
$$
If the weighted integral on the left-hand side is finite, then, with
the zero case being trivial, we  obtain \eqref{weighted_eqn}.

It remains to remove this finiteness condition. To this end, we apply the
preceding argument first to truncated potentials.
Let $f\in C_c^\infty(\mathbb R^{1+2d})$, and define
$$
I_{\frac{1}{p}}^{\varepsilon,R}f(\zeta)
:=
\int_{\{\varepsilon<d_*(\zeta,\xi)<R\}}
\frac{f(\xi)}
{d_*(\zeta,\xi)^{Q-\frac{1}{p}}}\,d\xi,
\qquad
0<\varepsilon<R<\infty.
$$
The function $I_{\frac{1}{p}}^{\varepsilon,R}f$ is bounded and has
bounded support. Since $g\in L_{\mathrm{loc}}^1$, it follows that
$$
\int_{\mathbb R^{1+2d}}
|g|\,
|
I_{\frac{1}{p}}^{\varepsilon,R}f
|^p
<\infty.
$$
Replace every occurrence of the integral operator in
\eqref{duality} and \eqref{T_est1} by the same truncated operator
$I_{1/p}^{\varepsilon,R}$, and put
$U=I_{1/p}^{\varepsilon,R}f$.  Its kernel is still positive and symmetric.
The two estimates \eqref{T_est2}--\eqref{T_est3} remain valid for the
truncated operator, since its kernel is bounded by the full kernel.
Repeating the duality calculation therefore gives, with
$S_{\varepsilon,R}=\int |g|U^p<\infty$,
\[
 S_{\varepsilon,R}
 \le C\|g\|_{\M^{q_0}_{\K}}^{1/p}
       S_{\varepsilon,R}^{1/p'}\|f\|_p.
\]
Division by $S_{\varepsilon,R}^{1/p'}$ when this quantity is nonzero
gives, after renaming the constant, uniformly in $\varepsilon$ and $R$,
$$
\int_{\mathbb R^{1+2d}}
|g|\,
|
I_{\frac{1}{p}}^{\varepsilon,R}f
|^p
\le C_6
\|g\|_{\M_\K^{q_0}}\|f\|_{p}^p.
$$
Since $f\ge0$, the monotone convergence theorem gives
$$I_{\frac{1}{p}}^{\frac{1}{n},n}f\uparrow I_{\frac{1}{p}}f.$$
Another application of the monotone convergence theorem therefore implies that
\eqref{weighted_eqn} holds for every
$f\in C_c^\infty(\mathbb R^{1+2d})$.

The estimate extends to arbitrary nonnegative $f\in L^p$ by
approximation with nonnegative smooth compactly supported functions,
Fatou's lemma, and the positivity of the kernel.  Applying this to $|f|$
gives \eqref{weighted_eqn} for every $f\in L^p$.  With the convention
\eqref{eq:Tgamma}, the resulting map is a bounded sublinear operator.
Indeed,
\[
 |I_{1/p}f-I_{1/p}h|\le I_{1/p}(f-h),
\]
so the weighted estimate also gives its continuous extension by density.

\smallskip
\noindent\emph{Step 2: Construction of an \(A_1\) majorant.}
Choose $q_1$ such that $
q_0<q_1<q $
and define
$$
\widetilde g
:=
\left(\mathsf M(|g|^{q_1})\right)^{\frac{1}{q_1}}.
$$
Since $g\in\mathcal M_{\mathcal K}^{q}$ and $q_1<q$, we have
$|g|^{q_1}\in L_{\mathrm{loc}}^1$. 
By the Lebesgue differentiation theorem \cite[Theorem 1.8]{Heinonen} applied to $|g|^{q_1}$, we have $|g|\le \widetilde g$ almost everywhere.

Moreover, Proposition~\ref{lem:maximal-function}(ii)
implies
\begin{align*}
	\|\widetilde{g} \|_{\M_\K^q}=
	\|
	(\mathsf M(|g|^{q_1}))^{\frac{1}{q_1}}
	\|_{ \M_\K^q}
	\le C_{7}
	\|g\|_{ \M_\K^q}.
\end{align*}
In particular, $\mathsf M(|g|^{q_1})<\infty$ almost everywhere.
Since $q_0<q_1$, the Coifman--Rochberg theorem on spaces of
homogeneous type \cite[Lemma~3.4]{LLW} implies
$$
\widetilde g^{q_0}
=
\bigl(\mathsf M(|g|^{q_1})\bigr)^{q_0/q_1}
\in A_1,
\qquad
\mathsf M(\widetilde g^{q_0})
\le C_8 \widetilde g^{q_0}.
$$

The $A_1$ characteristic in the last formula is bounded in terms of $d,q_0,q_1$, independently of $g$.
Thus $\widetilde g$ satisfies the hypothesis of Step~1 with a uniform
$A_1$ constant. We may now apply the estimate \eqref{weighted_eqn} with
$\widetilde g$ in place of $g$ and obtain
\begin{align*}
	\int_{\mathbb R^{1+2d}}
	|g| \,|I_{\frac{1}{p}}f|^p
	\le
	\int_{\mathbb R^{1+2d}}
	|\widetilde{g}| \,|I_{\frac{1}{p}}f|^p
	\le C_6
	\|\widetilde{g}\|_{\M_\K^{q_0}}\|f\|_{p}^p
	\le C_6
	\|\widetilde{g}\|_{\M_\K^{q}}\|f\|_{p}^p
	\le C_{9}
	\|g\|_{\M_\K^{q}}\|f\|_{p}^p.
\end{align*}
This completes the proof.
\end{proof}

Combining the fractional kernel estimates with the Adams  inequality above yields the following Rellich-type estimates, which will be used in the subsequent PDE and SDE analysis.

\begin{proposition}[Rellich estimates from kinetic Adams inequality]
\label{prop:rellich}
Let $q>1$, $1<p<\infty$,  and $\lambda>0$.  If $g\in\M^q_{\K}$, then
\begin{align}
 \norm{\abs g^{\frac{1}{p}}\A^{-\frac{1}{2p}}}_{p\rightarrow p}
 +\norm{\abs g^{\frac{1}{p}}(\A^*)^{-\frac{1}{2p}}}_{p\rightarrow p}
 &\le C_{d,p,q}\norm g_{\M^q_{\K}}^{\frac{1}{p}},
 \label{eq:rellich-zero}\\
 \norm{\abs g^{\frac{1}{p}}\nabla_v\A^{-\frac{1}{2}-\frac{1}{2p}}}_{p\rightarrow p}
 +\norm{\abs g^{\frac{1}{p}}\nabla_v(\A^*)^{-\frac{1}{2}-\frac{1}{2p}}}_{p\rightarrow p}
 &\le C_{d,p,q}\norm g_{\M^q_{\K}}^{\frac{1}{p}}.
 \label{eq:rellich-gradient}
\end{align}
If $g\in L^\infty$, the right-hand side of each estimate may be replaced by
\begin{equation*}
 C_{d,p}\lambda^{-\frac{1}{2p}}\norm g_\infty^{\frac{1}{p}}.
\end{equation*}
\end{proposition}

\begin{proof}
Taking $\gamma=\frac{1}{p}$ in \eqref{eq:power-by-I} and \eqref{eq:gradient-by-I}, and then applying the Adams inequality \eqref{eq:adams-operator}, yields \eqref{eq:rellich-zero} and \eqref{eq:rellich-gradient} for $f\in C_c^\infty(\mathbb R^{1+2d})$. The estimates extend to all $f\in L^p$ by density. If $g\in L^\infty$, the corresponding bounds follow directly from \eqref{eq:power-Lp} and \eqref{eq:gradient-power-Lp}, again with $\gamma=\frac{1}{p}$.
\end{proof}

\bigskip

\section{Proof of Theorem \ref{thm:main}}
\label{sec:pde}

In addition to the representation for $\eta$ in Theorem \ref{thm:main},
we will use the equivalent identity
\begin{equation}\label{eq:second-representation}
 \eta={}\A^{-1}f-\A^{-1/(2p)}G_p(b)(I+H_p(b)G_p(b))^{-1}H_p(b)
       \A^{-(1-1/(2p))}f.
\end{equation}
For an admissible regularization  $\{b_n\}_{n}$ of $b$ we will prove a more general result: 
if $f_n \rightarrow f$ in 
$\W_{\K}^{-1-1/p',p}$, then
\begin{equation}\label{eq:approx-convergence}
 \eta_n\rightarrow\eta
 \qquad\text{in }\W_{\K}^{1/p,p}.
\end{equation}
In particular, the point-source equation is given, for $\zeta_0\in\R^{1+2d}$, by
\begin{equation*}
 \A\eta-\divv(b\eta)=\delta_{\{\zeta_0\}}.
\end{equation*}
\begin{proof}[Proof of Theorem \ref{thm:main}]
By \eqref{eq:smallness}, both $I+G_p(b) H_p(b)$ and $I+H_p(b)G_p(b)$ are invertible on
$L^p$.  Formula \eqref{eq:main-representation} is therefore well defined.
Since \(\A^{-(1-1/(2p))}\) maps
$\W_{\K}^{1/p-2,p}$ isometrically into $L^p$ and
$\A^{-1/(2p)}$ maps $L^p$ isometrically onto
$\W_{\K}^{1/p,p}$ (see Appendix \ref{app:fractional-powers}), estimate \eqref{eq:apriori} follows.  The identity
\[
 (I+G_p(b)H_p(b))^{-1}
 =I-G_p(b)(I+H_p(b)G_p(b))^{-1}H_p(b)
\]
gives \eqref{eq:second-representation}. 

By Lemma \ref{lem:factorization}, \eqref{eq:main-representation} is a weak
solution.  Conversely, every weak solution produces an $h\in L^p$ satisfying
\eqref{eq:factorized-equation}.  Invertibility gives a unique $h$, and hence
a unique $\eta=\A^{-1/(2p)}h$.  This proves existence, the estimate, and uniqueness in
part~\textup{(i)}.

For part~\textup{(ii)}, admissibility gives $H_p(b_n)\rightarrow H_p(b)$ and
$G_p(b_n)\rightarrow G_p(b)$ strongly, while their norms have a common product bound
strictly below one.  The resolvent identity for inverses, first on a dense
set and then by uniform boundedness, yields
\[
 (I+G_p(b_n)H_p(b_n))^{-1}g
 \rightarrow (I+G_p(b)H_p(b))^{-1}g
 \quad\text{in }L^p
\]
for every $g\in L^p$.  Together with $f_n\rightarrow f$, this proves
\eqref{eq:approx-convergence}.
For smooth $b_n$ and $f$, the distributional solution $\eta_n$ is
classical.  Indeed, the diffusion fields $\partial_{v_j}$ and the
transport field $\partial_t+v\cdot\nabla_x-b_n\cdot\nabla_v$, together
with their commutators, span all space--time directions: the first
commutators give $\partial_{x_j}$ modulo velocity fields, and the
transport field then gives $\partial_t$.  The smooth zeroth-order
coefficient $\lambda-\divv b_n$ does not change this condition.
H\"ormander's hypoellipticity theorem \cite{Hormander} therefore applies. Specifically, for a smooth
regularized drift \(b_n\), put
\[
 X_j:=\partial_{v_j},\qquad
 Y_n:=\partial_t+v\cdot\nabla_x-b_n\cdot\nabla_v .
\]
Then
\[
 \mathscr A\eta-\divv(b_n\eta)
 =
 \lambda\eta+Y_n\eta-\sum_{j=1}^dX_j^2\eta
 -(\divv b_n)\eta .
\]
The zeroth-order term does not enter the bracket condition.  Moreover,
\[
 [X_j,Y_n]
 =
 \partial_{x_j}
 -\sum_{k=1}^d(\partial_{v_j}b_n^k)X_k.
\]
It follows that every \(\partial_{x_j}\) belongs pointwise to the span of
the fields \(X_k\) and the commutators \([X_j,Y_n]\).  Since
\[
 \partial_t
 =
 Y_n-\sum_{j=1}^dv_j\partial_{x_j}
 +\sum_{k=1}^db_n^kX_k,
\]
these fields and their first commutators span the entire tangent space
\(T_\zeta\mathbb R^{1+2d}\) at every \(\zeta\).  The regularized operator
therefore satisfies Hörmander's bracket condition.  Hörmander's
hypoellipticity theorem \cite{Hormander} consequently implies that if
\[
 \mathscr A u-\divv(b_nu)=F
\]
in the sense of distributions on an open set \(U\) and \(F\in C^\infty(U)\), then
\(u\in C^\infty(U)\).

Finally, for part~\textup{(iii)}, set $s=1+1/p'$.  Condition
\eqref{eq:point-p-range} is equivalent to
$p'>Q-1$, hence
\[
 s=1+\frac1{p'}>\frac Q{p'}.
\]
By Corollary \ref{cor:point-evaluation},
$\delta_{\{\zeta_0\}}\in\W_{\K}^{-s,p}
=\W_{\K}^{-1-1/p',p}$, and so part \textup{(i)} applies. Note also that every negative power of
$\A$ is causal, multiplication by $b$ preserves support in time, and the Duhamel
series in \eqref{eq:main-representation} is a sum of causal terms.
The solution therefore vanishes for $t<t_0$.
\end{proof}

\begin{remark}[Operator-norm formulation]\label{rem:operator-norm}
The Morrey hypothesis is a concrete sufficient condition for the two
Rellich estimates.  Theorem~\ref{thm:main} remains true verbatim for every
Borel drift for which the operators in \eqref{eq:Hp}--\eqref{eq:Gp} are
bounded and $\norm{G_p}_{p\rightarrow p}\norm{H_p}_{p\rightarrow p}<1$.
\end{remark}

\bigskip

\section{Proof of Theorem \ref{thm:backward}}
\label{sec:backward}

We consider here the backward PDE
\begin{equation}\label{eq:backward-equation}
 \A u+b\cdot\nabla_vu=f.
\end{equation}
Recall that
\begin{align*}
 R_p{(b)}:=\abs b^{1/p}\nabla_v\A^{-1/2-1/(2p)},\qquad Q_p{(b)}:=\A^{-1/(2p')}\bigl(b^{\frac{1}{p'}}\cdot\bigr).
\end{align*}
The operator $R_p{(b)}$ maps scalar functions to vector fields and $Q_p{(b)}$ contracts vector
fields with the polar factor $b^{\frac{1}{p'}}$.  By Proposition \ref{prop:rellich}, both
operators are bounded on $L^p$, and their product/composition
has norm less than one under
the same smallness regime \eqref{eq:smallness}, with $p$ and $p'$ placed as in
\eqref{eq:hbound}-\eqref{eq:gbound}.

\begin{proof}[Proof of Theorem \ref{thm:backward}]
Let us show that the weak solution to \eqref{eq:backward-equation} implies necessarily the representation \eqref{eq:backward-representation}. Existence and uniqueness in $\W^{1+1/p,p}_\K$ then follow as in the proof of Theorem \ref{thm:main}. Let $u\in \W^{1+1/p,p}_\K$ be a weak solution to \eqref{eq:backward-equation}. Write \(u=\A^{-1/2-1/(2p)}h\)  with $h\in L^p$ (this simply follows from the very definition of $\W^{1+1/p,p}_\K$).  The gradient estimate (again, in Proposition \ref{prop:rellich}) ensures 
$\nabla_vu\in L^p$ and $|b|^{1/p}\nabla_vu=R_p(b)h\in L^p(\mathbb R^{1+2d},\mathbb R^d)$.
H\"older's inequality on any compact set shows that $b\cdot\nabla_vu=b|b|^{-1/p}\cdot |b|^{1/p}\nabla_vu$ is locally
integrable.  By duality,
\[
 \A^{-1/(2p')}(b\cdot\nabla_vu)=Q_p(b)R_p(b)h\in L^p,
\]
so the first-order 
term $b\cdot\nabla_v u$ 
belongs to $\W_{\K}^{-1/p',p}=\A^{1/(2p')}L^p$.
Since
\(1-(1/2+1/(2p))=1/(2p')\), applying \(\A^{-1/(2p')}\) to
\eqref{eq:backward-equation} gives
\[
	\A^{1-1/(2p')}u+\A^{-1/(2p')}(b\cdot\nabla_vu)= (I+Q_p(b)R_p(b))h=\A^{-1/(2p')}f.
\]
Inverting $I+Q_p(b)R_p(b)$, we obtain $h=(I+Q_p(b)R_p(b))^{-1}\A^{-1/(2p')}f$. By inverting $\A^{1-1/(2p')}$, we end the proof.
\end{proof}

\bigskip

\section{Proof of Theorem \ref{thm:sde}}
\label{sec:sde}

For most of this section we assume that the drift satisfies the global hypothesis
\eqref{eq:HK1}. In particular, there is no separate bounded component. We recall that the
role of global $L^{1+\varepsilon_0}$ integrability is only to shorten the
proof of the higher-power estimate used in the tightness argument in the construction of a solution to \eqref{eq:kinetic-sde} -- we will remove it in Section
\ref{sec:local-removal}.

To start, we will need the time-reversed variant of Theorem \ref{thm:backward}. We only record the corresponding notations: the same proof remains valid.
For $1<r<\infty$, set
\begin{equation}\label{eq:adjoint-backward-factors}
 \begin{split}
  \widehat R_r(b)
   &:=\abs b^{1/r}\nabla_v(\A^*)^{-1/2-1/(2r)},\\
  \widehat Q_r(b)
   &:=(\A^*)^{-1/(2r')}
       \bigl(b^{1/r'}\cdot\bigr).
 \end{split}
\end{equation}
The estimates of Proposition \ref{prop:rellich} apply to these adjoint operators.  We call
$r$ \textit{backward admissible} if
\begin{equation}\label{eq:backward-smallness}
 \norm{\widehat Q_r(b)}_{r\rightarrow r}
 \norm{\widehat R_r(b)}_{r\rightarrow r}<1.
\end{equation}
The notion of an admissible regularization of $b$, and Proposition \ref{prop:regularization}, extend in a straightforward manner to the operators $\widehat R_r(b)$, $\widehat Q_r(b)$:

\begin{proposition}[A variant of Proposition \ref{prop:regularization}]
\label{lem:simultaneous-regularization}
Let $b=b_{\mathrm s}+b_{\mathrm b}$ satisfy \eqref{eq:HK}, and let
$E_{\mathrm f},E_{\mathrm b}\subset(1,\infty)$ be finite sets of forward
and backward exponents r, respectively given by
\[
 \mathfrak g_{r,\lambda}(b)\mathfrak h_{r,\lambda}(b)<1
 \quad(r\in E_{\mathrm f}),
 \qquad
 \|\widehat Q_r(b)\|_{r\rightarrow r}\,\|\widehat R_r(b)\|_{r\rightarrow r}<1
 \quad(r\in E_{\mathrm b}).
\]
Then there exists a
sequence
\[
 b_n=b_{n,\mathrm s}+b_{n,\mathrm b}
 \in C_c^\infty(\R^{1+2d};\R^d)
\]
such that
\begin{equation}\label{eq:local-regularization}
 \begin{aligned}
 b_{n,\mathrm s}&\rightarrow b_{\mathrm s}
     \quad\text{in }L^q_{\mathrm{loc}},
 &\sup_n\norm{b_{n,\mathrm s}}_{\M^q_{\K}}&<\infty,\\
 b_{n,\mathrm b}&\rightarrow b_{\mathrm b}
     \quad\text{in }L^a_{\mathrm{loc}},
 &\sup_n\norm{b_{n,\mathrm b}}_\infty&\le\norm{b_{\mathrm b}}_\infty\quad (a\in [1,\infty)).
 \end{aligned}
\end{equation}
In particular, $b_n\rightarrow b$ in $L^q_{\mathrm{loc}}$.  Moreover, for every
$r\in E:=E_{\mathrm f}\cup E_{\mathrm b}$,
\begin{equation*}
 \begin{aligned}
  H_r(b_n)&\rightarrow H_r(b),&\qquad G_r(b_n)&\rightarrow G_r(b),\\
 \widehat R_r(b_n)&\rightarrow\widehat R_r(b),&
  \qquad\widehat Q_r(b_n)&\rightarrow\widehat Q_r(b),
 \end{aligned}
\end{equation*}
strongly on the corresponding $L^r$ spaces.  The required product
bounds remain uniformly strict:
\begin{align*}
 \sup_n\mathfrak g_{r,\lambda}(b_n)\mathfrak h_{r,\lambda}(b_n)<1 &
       \quad \text{ for all }r\in E_{\mathrm f},\\
 \sup_n\|\widehat Q_r(b_n)\|_{r\rightarrow r}\,\|\widehat R_r(b_n)\|_{r\rightarrow r}<1 &
       \quad \text{ for all } r\in E_{\mathrm b}.
\end{align*}
In fact, each individual backward operator norm does not increase under
this regularization.
If $b_{\mathrm b}=0$ and, in addition to \eqref{eq:HK}, $b\in L^a$ for some
$1<a<\infty$, the sequence may be chosen so that
\begin{equation*}
 \sup_n\norm{b_n}_a\le\norm{b}_a.
\end{equation*}
\end{proposition}

The proof of the above follows closely the proof of Proposition \ref{prop:regularization}, so, for the sake of brevity, we skip it.

\begin{corollary}
\label{cor:backward-stability}
Let $\{b_n\}_n$ be bounded smooth drifts such that
$\widehat R_p(b_n)$ and $\widehat Q_p(b_n)$ converge strongly to the
corresponding operators of $b$, and suppose
\[
 \sup_n\norm{\widehat Q_p(b_n)}_{p\rightarrow p}
       \norm{\widehat R_p(b_n)}_{p\rightarrow p}<1.
\]
If $f_n\rightarrow f$ in $\W_{\K^*}^{-1/p',p}$, then the corresponding solutions
of
\[
 \A^*u_n+b_n\cdot\nabla_vu_n=f_n
\]
satisfy
\begin{equation}\label{eq:backward-stability}
 u_n\rightarrow u
 \qquad\text{in }\W_{\K^*}^{1+1/p,p}.
\end{equation}
If $p>Q-1$, the convergence is also uniform on $\R^{1+2d}$.
\end{corollary}

\begin{proof}
Put
\[
 S_n:=\widehat Q_p(b_n)\widehat R_p(b_n),
 \qquad S:=\widehat Q_p(b)\widehat R_p(b).
\]
Strong convergence of the individual operators and their uniform boundedness
give $S_nh\rightarrow Sh$ for every $h\in L^p$.  If their common product bound is
$\sup_n\|\widehat Q_p(b_n)\|_{p\rightarrow p}\|\widehat R_p(b_n)\|_{p\rightarrow p}=:\vartheta<1$, then
\[
 \sup_n\norm{(I+S_n)^{-1}}_{p\rightarrow p}
 \le\frac1{1-\vartheta}.
\]
The identity
\[
 (I+S_n)^{-1}-(I+S)^{-1}
 =(I+S_n)^{-1}(S-S_n)(I+S)^{-1}
\]
therefore shows that the inverses converge strongly on $L^p$.
Using the adjoint version of \eqref{eq:backward-representation}, together
with
\[
 \norm{f_n-f}_{\W_{\K^*}^{-1/p',p}}
 =\norm{(\A^*)^{-1/(2p')}(f_n-f)}_p,
\]
proves \eqref{eq:backward-stability}.  If $p>Q-1$, then
$1+1/p>Q/p$, and the uniform convergence follows from Proposition
\ref{prop:sobolev}(ii).
\end{proof}

\subsection{Kinetic Krylov estimates}

The next lemma is the bridge from the PDE theory to the SDE.  The vector field $k$ below is an arbitrary test drift.

\begin{lemma}[Krylov estimate]
\label{lem:kinetic-krylov}
Let $p>Q-1$ be backward admissible, and let $\{b_n\}_n$ be the approximating sequence from Proposition
\ref{lem:simultaneous-regularization}, with a uniform backward product bound
strictly smaller than one.  Let
\begin{equation}\label{eq:auxiliary-drift}
 k\in\M^q_{\K}(\R^{1+2d}, \R^d)
       \cap L^1(\R^{1+2d}, \R^d).
\end{equation}
For every bounded Borel function $h$ such that
$\abs k^{1/p}h\in L^p$, the solution of
\begin{equation}\label{eq:backward-krylov-pde}
 \A^*u+b_n\cdot\nabla_vu=\abs{kh}
\end{equation}
satisfies, with a constant independent of $n$,
\begin{equation}\label{eq:krylov-pde-bound-singular}
 \norm u_\infty
 \le C\norm{k}_{\M^q_{\K}}^{1/p'}\times
       \norm{\abs{k}^{1/p}h}_p.
\end{equation}
If $Z^n$ ($=(X^n,V^n)$) denotes the strong 
solution of \eqref{eq:kinetic-sde} with
$b_n$ in place of $b$, then
\begin{equation}\label{eq:krylov-smooth}
 \mathbb E^n_{s,z}\int_s^T\abs{k(r,Z^n_r)h(r,Z^n_r)}\,dr
 \le C_T\norm{k}_{\M^q_{\K}}^{1/p'}\times
       \norm{\mathbf1_{[s,T]}\abs{k}^{1/p}h}_p.
\end{equation}
One also has the unweighted estimate
\begin{equation}\label{eq:krylov-one}
 \mathbb E^n_{s,z}\int_s^T\abs{h(r,Z^n_r)}\,dr
 \le C_T\norm{\mathbf1_{[s,T]}h}_p.
\end{equation}
All constants are uniform in the starting point $(s,z)$.
\end{lemma}

\begin{proof}
The adjoint version of Theorem \ref{thm:backward} gives the (unique) representation
\begin{equation}\label{eq:adjoint-backward-representation}
 u=(\A^*)^{-1/2-1/(2p)}
   (I+\widehat Q_p(b_n)\widehat R_p(b_n))^{-1}
   (\A^*)^{-1/(2p')}F,\qquad F:=|kh|,
\end{equation}
for any weak solution of \eqref{eq:backward-krylov-pde}. 
By
duality and the Rellich estimates \eqref{eq:rellich-zero},
\begin{align*}
 \norm{(\A^*)^{-1/(2p')}(\abs{k}h)}_p
 &=\norm{(\A^*)^{-1/(2p')}
   \bigl(\abs{k}^{1/p'}\abs{k}^{1/p}h\bigr)}_p\notag\\
 &=\sup_{\|\varphi\|_{p'}= 1}\langle \abs{k}^{1/p'}(\A)^{-1/(2p')}\varphi, \abs{k}^{1/p}h\rangle\notag\\
 &\le C\norm{k}_{\M^q_{\K}}^{1/p'}
       \norm{\abs{k}^{1/p}h}_p.
\end{align*}
Since \(1+1/p>Q/p\), Proposition \ref{prop:sobolev}(ii) maps the range of
\((\A^*)^{-1/2-1/(2p)}\) into \(C_\infty\).  This proves
\eqref{eq:krylov-pde-bound-singular}.

For $\{F_m\}_m$ a smooth compactly supported approximation of $F$ such that $\|(\A^*)^{-1/(2p')}F_m\|_p \le C\norm{k}_{\M^q_{\K}}^{1/p'}
	\norm{\abs{k}^{1/p}h}_p$ (e.g. Proposition \ref{prop:regularization}), It\^o's formula applied to  $e^{-\lambda (t-s)}u_{m}(t,\cdot)$
	of \eqref{eq:backward-krylov-pde} for $F_m$ in place of $F=|kh|$ along $Z^n_{t}$  gives 
\[
		\mathbb E_{s,z}\int_s^T
		e^{-\lambda(r-s)}F_m(r,Z_r^n)\,dr
		=
		u_m(s,z)
		-\mathbb E_{s,z}\!\left[
		e^{-\lambda(T-s)}u_m(T,Z_T^n)
		\right] \le  2\|u_m\|_{\infty}.
	\]
The resolvent parameter introduces
$e^{-\lambda(r-s)}$ and removing it on $[s,T]$ costs at most $e^{\lambda T}$.
Applying truncation and using Fatou's lemma we extend the estimate to the stated Borel right-hand
sides.

For \eqref{eq:krylov-one}, we use
\eqref{eq:adjoint-backward-representation} with $F=\abs h$ and estimate
$(\A^*)^{-1/(2p')}$ by \eqref{eq:power-Lp}. The same embedding and It\^o argument then yield \eqref{eq:krylov-one}.
\end{proof}

\begin{lemma}[Higher-power Krylov estimate]
\label{lem:global-power-krylov}
Let $p>Q-1$ be backward admissible and assume \eqref{eq:HK1}.  Select a
regularization $\{b_n\}$ that additionally satisfies
\begin{equation}\label{eq:global-higher-integrability-approximation}
 \sup_n\norm{b_n}_{1+\varepsilon_0}
 \le \norm{b}_{1+\varepsilon_0}.
\end{equation}
Then
\begin{equation}\label{eq:global-power-krylov}
 \sup_n\sup_{0\le s<T}\sup_{z\in\R^{2d}}
 \mathbb E^n_{s,z}\int_s^T
       \abs{b_n(r,Z^n_r)}^{1+\varepsilon_0/p}\,dr
 \le C_T<\infty.
\end{equation}
\end{lemma}

\begin{proof}
Apply \eqref{eq:krylov-smooth} with $k=b_n$ and
$h_R=(\abs{b_n}\wedge R)^{\varepsilon_0/p}$. Then
\[
 \mathbb E^n_{s,z}\int_s^T
   \abs{b_n(r,Z_r^n)}
       (\abs{b_n(r,Z_r^n)}\wedge R)^{\varepsilon_0/p}\,dr
 \le C_T\norm{b_n}_{\M^q_{\K}}^{1/p'}
       \left(\int_s^T\!\int_{\R^{2d}}
          \abs{b_n}^{1+\varepsilon_0}\,dx\,dv\,dr\right)^{1/p}.
\]
The right-hand side is bounded uniformly in $n$, $s$, and $z$ by
\eqref{eq:global-higher-integrability-approximation}.  Letting
$R\rightarrow\infty$ and using the monotone convergence theorem proves
\eqref{eq:global-power-krylov}.  The regularization in
Proposition \ref{lem:simultaneous-regularization} can be chosen to satisfy
\eqref{eq:global-higher-integrability-approximation} since truncation,
cutoff and convolution do not increase the $L^{1+\varepsilon_0}$ norm.
\end{proof}

\subsection{Proof of Theorem \ref{thm:sde} completed}

We divide the proof into five steps. Steps 1 and  2 are dedicated to existence and distribution estimates parts of Theorem \ref{thm:sde}: In Step 1, we show tightness of $\{\mathbb P^n_{s,z}:=\mathbb P\circ (Z^n)^{-1}\}_n$.
In Step 2, we identify subsequential limits as martingale  (and so  weak) solutions to \eqref{eq:kinetic-sde}, and establish the stability (as $n\rightarrow \infty$) of the Krylov bounds obtained in Lemmas \ref{lem:global-power-krylov} and \ref{lem:kinetic-krylov} and the estimate \eqref{eq:drift-integrability} on $b$. Steps 3 and 4 address respectively the conditional uniqueness and Markov-Feller properties of this solution, and eventually Step 5 identifies the estimates of the one-time marginal densities.

\smallskip
\noindent
\emph{Step 1: Tightness.}
Choose the single sequence $b_n$ from
Proposition \ref{lem:simultaneous-regularization} at $p$, $p_1$, and, when needed, $r$.
Choose it also to satisfy
\eqref{eq:global-higher-integrability-approximation}.
For every $n$, the kinetic SDE with drift $b_n$ has a unique nonexplosive strong
solution $Z^n=(X^n,V^n)$.
Put
\[
 \alpha:=\frac{\varepsilon_*}{1+\varepsilon_*}\in(0,1).
\]
Apply Lemma \ref{lem:global-power-krylov}.  Then
\begin{equation}\label{eq:uniform-power-drift}
 D_n:=\left(\int_s^T
       \abs{b_n(r,Z^n_r)}^{1+\varepsilon_*}\,dr
       \right)^{1/(1+\varepsilon_*)}
	   \end{equation}
	   satisfies
	   \begin{equation*}
 \sup_n\mathbb E
 D_n^{1+\varepsilon_*}
 =\sup_n\mathbb E
 \int_s^T
       \abs{b_n(r,Z^n_r)}^{1+\varepsilon_*}\,dr\le C_T.
\end{equation*}
Put $B^n_t:=\int_s^t b_n(r,Z^n_r)\,dr$.  H\"older's inequality gives, for $s\le u<t\le T$,
\begin{equation}\label{eq:drift-holder-modulus}
 \abs{B^n_t-B^n_u}
 \le \abs{t-u}^{\varepsilon_* /(1+\varepsilon_*)}D_n.
\end{equation}
In particular,
\begin{equation}\label{eq:Bn-tightness-conditions}
\mathbb P\bigl(\max_{s\le r\le T}|B^n_r|>R\bigr)\le \frac{C_T(T-s)^{\varepsilon_*}}{R^{1+\varepsilon_*}},\quad
\mathbb P\bigl(\mathfrak w(B^n,\delta)>\gamma\bigr)
		\le \frac{C_T\delta^{\varepsilon_*}}
		{\gamma^{1+\varepsilon_*}},
\end{equation}
for $\delta\in(0,1)$ and $\mathfrak w(B^n,\delta)$ the modulus of continuity 
\[
\mathfrak w(B^n,\delta):=\sup_{t,r\in[s,T],|t-r|\le \delta}|B^n_t-B^n_r|.
\]
The first bound in \eqref{eq:Bn-tightness-conditions} together with standard Brownian (moment) estimates immediately gives 
\begin{equation}
	\lim_{R\rightarrow\infty}\sup_n
	\mathbb P
	\left(
	\sup_{s\le r\le T}(\abs{X_r^n}+\abs{V^n_r})>R\right)=0,
\end{equation}
while the second bound, together with standard Brownian (increment) estimates yields, for any $\epsilon>0$, and compact $K\subset \R^{2d}$,
\begin{equation}\label{eq:compact-containment}
	\lim_{\delta\rightarrow0}\limsup_n\sup_{\substack{s\in[0,T]\\z\in K}}
	\mathbb P
	\left(
	\mathfrak w(Z^n,\delta)>\epsilon\right)=0.
\end{equation}
Prokhorov's theorem (see e.g. \cite[Theorem 7.3, Chapter 2]{Billingsley-99}) so applies, ensuring tightness of the measures $\{\mathbb P^n_{s,z}=\mathbb P\circ(Z^n)^{-1}\}_n$ on the path space $\Omega_{s,T}$.

\medskip
\noindent
\emph{Step 2: Occupation measures, limit to a weak solution} and Krylov bounds.
1.~Following Step 1, and without loss of generality, let $\mathbb P^n_{s,z}$ denote a converging (sub-)sequence of the laws of $Z^n$ with $\mathbb P^n_{s,z}\Rightarrow \mathbb P_{s,z}$ on $\Omega_{s,T}$.  Define the approximating occupation measures
\[
 \mu_n(D):=\mathbb E^n_{s,z}\int_s^T
       \mathbf1_D(r,Z_r)\,dr,\qquad D,\text{ a Borel set of }\R^{1+2d},
\]
(we recall that $Z$ is the canonical process of $\Omega_{s,T}$).
The unweighted PDE estimate \eqref{eq:krylov-one} and $L^p$ duality give
\begin{equation}\label{eq:occupation-Lp}
 \mu_n(dr,dz)=\rho_n(r,z)\,dr\,dz
 \qquad \sup_n\norm{\rho_n}_{p'}<\infty.
\end{equation}
The weak convergence $\mathbb P^n_{s,z}\Rightarrow \mathbb P_{s,z}$ on $\Omega_{s,T}$ implies vague convergence of $\mu_n$ to the occupation measure $\mu(D)=\mathbb E_{s,z}\int_s^T\mathbf1_D(r,Z_r)\,dr$, while the compactness in $L^{p'}$ and weak lower semicontinuity of the $L^{p'}$-norm  therefore give,  up to extracting a further converging subsequence of $\mathbb P^n_{s,z}$, 
	\begin{equation}\label{occup-measure-limit}
	\rho_n\rightharpoonup\rho\quad\text{in }L^{p'},\qquad \norm{\rho}_{p'}
	\le\liminf_{n\to\infty}\norm{\rho_n}_{p'}
	\le C_T<\infty.
	\end{equation}
As such   $d\mu=\rho\,dr\,d\zeta$ and 
\begin{align}\label{first-Krylov-bound}
	\mathbb E_{s,z}\int_s^T\abs{h(r,Z_r)}\,dr=\int_{\mathbb R^{1+2d}}\mathbf1_{[s,T]}(r)\abs{h(r,\zeta)}\mu(dr,d\zeta)
	&=\int_{\R^{1+2d}}\mathbf1_{[s,T]}(r)
	\abs{h(r,\zeta)}\rho(r,\zeta)\,dr\,d\zeta\notag\\
	&\le\norm{\mathbf1_{[s,T]}h}_p\norm{\rho}_{p'}
	\le C_T\norm{\mathbf1_{[s,T]}h}_p.
\end{align}
This proves the first claimed Krylov bound.

2.~Let $k$ satisfy \eqref{eq:auxiliary-drift} and set
$d\nu_k=\abs k\,dr\,d\zeta$.  The weighted PDE estimate
\eqref{eq:krylov-smooth} gives
\begin{equation}\label{eq:weighted-density-bound}
 \sup_n\norm{\rho_n}_{L^{p'}(\nu_k)}
 \le C_T\norm{k}_{\M^q_{\K}}^{1/p'}.
\end{equation}
From this, we can check that, for all bounded and compactly supported Borel function $\varphi$, 
\begin{equation}\label{weighted-occupation-limit}
\lim_{n\to\infty}\int\varphi\rho_n\,d\nu_k
=\int\varphi\rho\,d\nu_k.
\end{equation}
For every $R>0$,
the function $\varphi(\abs k\wedge R)$ belongs to $L^p$ with respect
to Lebesgue measure.  Hence \eqref{occup-measure-limit} gives
\[
\int\varphi\rho_n(\abs k\wedge R)\,dr\,d\zeta
\rightarrow 
\int\varphi\rho(\abs k\wedge R)\,dr\,d\zeta.
\]
Moreover, by H\"older's inequality with respect to
$d\nu_k=\abs k\,dr\,d\zeta$,
\[
\begin{aligned}
	&\left|\int\varphi\rho_n
	\bigl(\abs k-\abs k\wedge R\bigr)\,dr\,d\zeta\right|\\
	&\qquad\le
	\norm{\rho_n}_{L^{p'}(\nu_k)}
	\left(
	\int_{\{\abs k>R\}}\abs{\varphi}^p\abs k\,dr\,d\zeta
	\right)^{1/p}.
\end{aligned}
\]
The right-hand side tends to zero as $R\to\infty$, uniformly in $n$,
by \eqref{eq:weighted-density-bound} and the assumption $k\in L^1$. Thus \eqref{weighted-occupation-limit} follows.

Now, arguing as in the previous step, i.e.\,using weak lower semicontinuity of the weighted $L^{p'}$ norm and the a priori weighted Krylov bound, we conclude
\begin{equation}\label{second-Krylov-bound}
 \mathbb E_{s,z}\int_s^T\abs{k(r,Z_r)h(r,Z_r)}\,dr
 \le C_T\norm{k}_{\M^q_{\K}}^{1/p'}
       \norm{\mathbf1_{[s,T]}\abs k^{1/p}h}_p.
\end{equation}

3.~The higher-power estimate on $b_n(r,Z^n_r)$ in Lemma \ref{lem:global-power-krylov} also survives the passage to the limit.  Put
$r_*=1+\varepsilon_*$.  Fix a cutoff
$\chi\in C_c^\infty(\R^{1+2d})$ with $0\le\chi\le1$.
We first observe that
\[
 \mathbb E^n_{s,z}\int_s^T\big(\chi|b_n-b|\big)(r,Z_r)\,dr=\int\chi\abs{b_n-b}\,d\mu_n\rightarrow 0.
\]
Indeed, set $k_n=\chi(b_n-b)$.  These vector fields belong to
$\M^q_{\K}\cap L^1$, their Morrey norms are uniformly bounded,
and $\norm{k_n}_1\to0$ by \eqref{eq:local-regularization}.
Applying \eqref{eq:krylov-smooth} with $k=k_n$ and $h=1$ gives
\[
 \int\chi\abs{b_n-b}\,d\mu_n
 \le C_T\norm{k_n}_{\M^q_{\K}}^{1/p'}
           \norm{k_n}_1^{1/p}
 \rightarrow 0.
\]

For $M>0$, the truncated power satisfies the elementary inequality
\[
 \left|
   (\abs a\wedge M)^{r_*}-(\abs c\wedge M)^{r_*}
 \right|
 \le r_*M^{r_*-1}\abs{a-c},
 \qquad a,c\in\R^d.
\]
This follows from the mean value theorem on $[0,M]$ and
$\bigl|\abs a-\abs c\bigr|\le\abs{a-c}$. Hence, for fixed $M$,
\[
 \left|
 \int\chi(\abs{b_n}\wedge M)^{r_*}\,d\mu_n
 -
 \int\chi(\abs b\wedge M)^{r_*}\,d\mu_n
 \right|
 \le r_*M^{r_*-1}
       \int\chi\abs{b_n-b}\,d\mu_n
 \rightarrow 0.
\]
On the other hand, $\chi(\abs b\wedge M)^{r_*}$ is a bounded,
compactly supported Borel function and hence belongs to $L^p$
with respect to Lebesgue measure.  The weak convergence
$\rho_n\rightharpoonup\rho$ in $L^{p'}$ therefore gives
\[
 \int\chi(\abs b\wedge M)^{r_*}\,d\mu_n
 \rightarrow 
 \int\chi(\abs b\wedge M)^{r_*}\,d\mu.
\]
Combining these two convergences, we obtain
\begin{equation}\label{eq:local-drift-duality-convergence}
 \int\chi(\abs{b_n}\wedge M)^{r_*}\,d\mu_n
 \rightarrow 
 \int\chi(\abs b\wedge M)^{r_*}\,d\mu.
\end{equation}

Since $0\le\chi\le1$, it follows that
\[
 \int\chi(\abs b\wedge M)^{r_*}\,d\mu
 \le\liminf_{n\to\infty}
       \int\abs{b_n}^{r_*}\,d\mu_n.
\]
First let $M\to\infty$.  Then choose cutoffs
$0\le\chi_j\le1$ with $\chi_j\uparrow1$ on $\R^{1+2d}$ and let
$j\to\infty$.  Two applications of the monotone convergence
theorem, followed by \eqref{eq:global-power-krylov}, yield
\begin{equation}\label{eq:power-lower-semicontinuity}
 \begin{split}
 \mathbb E_{s,z}\int_s^T
       \abs{b(r,Z_r)}^{r_*}\,dr
 &=\int\abs b^{r_*}\,d\mu\\
 &\le\liminf_{n\to\infty}
       \int\abs{b_n}^{r_*}\,d\mu_n
 \le C_T.
 \end{split}
\end{equation}
This proves the first estimate in \eqref{eq:drift-integrability}.
The second follows from H\"older's inequality:
\begin{equation}\label{eq:krylov-bound-b-local}
 \mathbb E_{s,z}\int_u^t\abs{b(r,Z_r)}\,dr
 \le (t-u)^{1-1/r_*}
      \left(
       \mathbb E_{s,z}\int_u^t\abs{b(r,Z_r)}^{r_*}\,dr
      \right)^{1/r_*}
 \le C_T^{1/r_*}(t-u)^{1-1/r_*}.
\end{equation}

4.~We next show that the probability measure $\mathbb P_{s,z}$, on $(\Omega_{s,T},\{\mathcal F_t\}_{s\le t\le T})$, constructed above solves the martingale problem for the drift $b$ and initial condition $z$ at time $s$ in the sense of Definition \ref{def:martingale-problem}.  For $\varphi\in C^1([s,T];C_c^2(\R^{2d}))$, and 
\begin{equation*}
	L^n_r\varphi(r,x,v)
	:=\partial_r\varphi(r,x,v)+v\cdot\nabla_x\varphi(r,x,v)
	-b_n(r,x,v)\cdot\nabla_v\varphi(r,x,v)
	+\Delta_v\varphi(r,x,v),
\end{equation*}
under $\mathbb P^n_{s,z}$, the process
$\{ M^{n,\varphi}_t
 :=\varphi(t,Z_t)-\varphi(s,z)-\int_s^tL^n_r\varphi(r,Z_r)\,dr\}_{t\in[s,T]}$ is a martingale and so, for all $s\le t_1\le t_2\le \cdots \le t_\ell\le u<t\le T$ and $G^\ell_u:=G(Z_{t_1},\ldots,Z_{t_\ell})$ for $G:\mathbb R^{2d\ell}\rightarrow \mathbb R$ continuous bounded
$$
\mathbb E^n_{s,z}\left[G^\ell_u\big(M^n_t-M^n_u\big)\right]=0.
$$
(Owing to a monotone-class argument, the "test function" $G^\ell_u$ extends the
identity from cylinder functions $G(Z_{t_1},\ldots,Z_{t_\ell})$ to every bounded
$\mathcal F_u$-measurable random variable.)
As $n\rightarrow\infty$ and since $\mathbb P^{n}_{s,z}\Rightarrow \mathbb P_{s,z}$,
 \begin{align*}
 &\mathbb E^n_{s,z}\left[G^\ell_u\Big(\varphi(t,Z_t)-\varphi(u,Z_u)-\int_u^t \partial_r\varphi(r,Z_r)+V_r\cdot\nabla_x\varphi(r,Z_r)+\triangle_v\varphi(r,Z_r)\,dr\Big)\right]\\
 &\rightarrow \mathbb E_{s,z}\left[G^\ell_u\Big(\varphi(t,Z_t)-\varphi(u,Z_u)-\int_u^t \partial_r\varphi(r,Z_r)+V_r\cdot\nabla_x\varphi(r,Z_r)+\triangle_v\varphi(r,Z_r)\,dr\Big)\right],
 \end{align*}
 and it remains to prove that
 \begin{equation}\label{eq:proof-step-weakexistence}
 \mathbb E^n_{s,z}\left[G^\ell_u\int_u^t b_n(r,Z_r)\cdot\nabla_v\varphi(r,Z_r)\,dr\right]\\
 \rightarrow \mathbb E_{s,z}\left[G^\ell_u\int_u^t b(r,Z_r)\cdot\nabla_v\varphi(r,Z_r)\,dr\right].
 \end{equation}
 To this end, let us introduce a two-parameter approximation $\{b_{n,m}\}_{n,m}$ of $b$ such that $b_m:=\lim_{n\rightarrow \infty}b_{m,n}$ and $b_n:=\lim_{m\rightarrow \infty}b_{m,n}$ still approximate $b$ in the sense of Proposition \ref{prop:regularization} (Such approximation can be constructed by a slight modification of the truncation -- mollification construction in the proof of Proposition \ref{prop:regularization}, e.g. duplicating the cut-off $\chi_n$ into $\chi_n\chi_m$ and modifying the calibrated mollifier $\rho_{\epsilon_n}$ into $\rho_{\epsilon_{n+m}}$ with the bound $2^{-n}+2^{-m}$ in \eqref{eq:diagonal-local-errors}.) With this, it follows that
 \begin{equation*}
 	\begin{split}
 		&\Big|\mathbb E^n_{s,z}\left[G^\ell_u
 		\int_u^t
 		b_n(r,Z_r)\cdot\nabla_v\varphi(r,Z_r
 		)\,dr\right]-
 		\mathbb E_{s,z}\left[G^\ell_u
 		\int_u^t
 		b(r,Z_r)\cdot\nabla_v\varphi(r,Z_r)\,dr\right]\Big|\\
 		&\le \Big|\mathbb E^n_{s,z}\left[G^\ell_u
 		\int_u^t
 		b_{n,m}(r,Z_r)\cdot\nabla_v\varphi(r,Z_r
 		)\,dr\right]-
 		\mathbb E_{s,z}\left[G^\ell_u
 		\int_u^t
 		b_{m}(r,Z_r)\cdot\nabla_v\varphi(r,Z_r)\,dr\right]\Big|\\
 		&\quad +\Big|\mathbb E^n_{s,z}\left[G^\ell_u
 		\int_u^t
 		\Big(b_{n,m}(r,Z_r)-b_n(r,Z_r)\Big)\cdot\nabla_v\varphi(r,Z_r
 		)\,dr\right]\Big|\\
 		&\quad+\Big|\mathbb E_{s,z}\left[G^\ell_u
 		\int_u^t
 		\Big(b_{m}(r,Z_r)-b(r,Z_r)\Big)\cdot\nabla_v\varphi(r,Z_r
 		)\,dr\right]\Big|.
 	\end{split}
 \end{equation*}
 Invoking again $\mathbb P^n_{s,z}\Rightarrow \mathbb P_{s,z}$, the first upper-bound vanishes as $n\rightarrow \infty$, for all fixed $m$. For the rest, choose $\chi\in C_c^\infty(\R^{1+2d})$ equal to one on $\supp\nabla_v\varphi$ 
and put $k_{n,m}:=\chi(b_{n,m}-b_n)$.  Then
$k_{n,m}\in\M^q_{\K}\cap L^1$, its Morrey norms are uniformly bounded in $n,m$, and $\lim_{m}\sup_n\norm{k_{n,m}}_1=0$
by \eqref{eq:local-regularization}.  Hence the estimate \eqref{eq:krylov-smooth} for $Z^n$ gives
\begin{equation}\label{eq:direct-coefficient-error}
 \mathbb E^n_{s,z}\int_s^T
  \abs{b_{n,m}-b_{n}}(r,Z_r)\abs{\nabla_v\varphi(r,Z_r)}\,dr
 \le C_{T,\|\nabla_v\varphi\|_\infty}\norm{k_{n,m}}_{\M^q_{\K}}^{1/p'}
       \norm{k_{n,m}}_1^{1/p}\rightarrow0.
\end{equation}
In the same way,  taking $k_m=\chi(b_{m}-b)$, the estimate \eqref{second-Krylov-bound} for $Z$ yields
\begin{equation*}
	\mathbb E_{s,z}\int_s^T
	\abs{b_m-b}(r,Z_r)\abs{\nabla_v\varphi(r, Z_r)}\,dr
	\le C_{T,\|\nabla_v\varphi\|_\infty}\norm{k_{m}}_{\M^q_{\K}}^{1/p'}
	\norm{k_{m}}_1^{1/p}\rightarrow0.
\end{equation*}	
Subsequently, \eqref{eq:proof-step-weakexistence} follows. The estimates in $2.$ and $3.$ naturally still hold  and this proves $(i)$ in Theorem \ref{thm:sde}.

\smallskip
\noindent
\emph{Step 3: Conditional uniqueness.} For any $(s_0,z_0)\in[0,T]\times\mathbb R^{2d}$, let us introduce the class $\mathcal Q_{s_0,z_0}$ consisting of all probability measures $\mathbb Q_{s_0,z_0}$ on $\Omega_{s_0,T}$ that solve the martingale problem (in the sense of Definition \ref{def:martingale-problem}) for $b$ with initial time-state $(s_0,z_0)$, and lying in the uniqueness class set in $(iii)$ i.e. satisfying, for every Borel bounded function $h$ on $[s_0,T]\times\mathbb R^{2d}$,
\begin{align}
\mathbb E^{\mathbb Q}_{s_0,z_0}\int_{s_0}^T\abs{h(r,Z_r)}\,dr
&\le K_{\mathbb Q}\norm{\mathbf1_{[s_0,T]}h}_p,
		\label{eq:uniqueness-krylov-one}\\
		\mathbb E^{\mathbb Q}_{s_0,z_0}\int_{s_0}^T
		\abs{b(r,Z_r)h(r,Z_r)}\,dr
		&\le K_{\mathbb Q}
		\norm{\mathbf1_{[s_0,T]}\abs b^{1/p}h}_p,
		\label{eq:uniqueness-krylov-b}
	\end{align}
	with $K_{\mathbb Q}$ independent of $h$. Then weak uniqueness of \eqref{eq:kinetic-sde} holds provided that, for all $(s_0,z_0)\in[0,T]\times\mathbb R^{2d}$, $\mathbb Q_{s_0,z_0}\in \mathcal Q_{s_0,z_0}$, $t\in[s_0,T]$ and for all Borel function $F=F(r,z)$ with $z\mapsto F(r,z)$ belonging to a determining class (or dense in a determining class) for $\mathcal P(\R^{2d})$, 
	\begin{equation}\label{eq:conditional-occupation-duality}
 u(s_0,z_0):=\mathbb E^{\mathbb Q_{s_0,z_0}}\Big[\int_{s_0}^t F(r,Z_r)\,dr\Big]
	\end{equation}
	is independent of the choice of ${\mathbb Q_{s_0,z_0}}$. This characterization is classically enough to ensure uniqueness of the solution to the martingale problem of $b$ at $(s_0,z_0)$ (we refer to e.g. \cite[Theorem 6.2.3 and Corollaries 6.2.4 and 6.2.6]{StrookVaradhan-06} and \cite[Chapter 4.4]{EthierKurtz-05}), and to the uniqueness in law of \eqref{eq:kinetic-sde}.

	Fix $F\in C_c^\infty((s_0,T)\times\R^{2d})$, and, applying  Theorem \ref{thm:backward}, let $\omega_n$ be the solution to the backward PDE $\A^*(\omega_n)+b_n\cdot\nabla_v\omega_n=e^{\lambda(t-T)}F$. Backward causality gives $\omega_n(r,z)=0$ for all $r\ge T$, and so $u_n(r,z)=e^{\lambda(T-r)}\omega_n(r,z)$ solves the backward PDE:
	\begin{equation}\label{eq:terminal-approximation}
		\begin{cases}
			(-\partial_t-v\cdot\nabla_x-\Delta_v
			+b_n\cdot\nabla_v)u_n=F,&s_0<t<T,\\
			u_n(T,\cdot)=0.
		\end{cases}
	\end{equation}
	Let us next apply Corollary
	\ref{cor:backward-stability} to $\omega_n=e^{\lambda(t-T)}u_n$ at both the exponents $p$ and $p_1$ given in \eqref{eq:sde-exponents}.
	The estimate at $p$ and the kinetic Sobolev embedding give uniform
	convergence on $[s_0,T]\times\R^{2d}$. The estimate at $p_1$, followed by
	the velocity-gradient bound for the fractional resolvent, gives
	\begin{equation}\label{eq:un-convergence}
		u_n\rightarrow  u_F\quad\hbox{uniformly},
		\qquad
		\sup_n\norm{\nabla_vu_n}_{L^{p_1}([s_0,T]\times\R^{2d})}
		\le C_T\norm F_{p_1}.
	\end{equation}
	In particular, $u_F$ is bounded and continuous, vanishes at time $T$,
	and is only determined by \eqref{eq:terminal-approximation}.
	
	For each fixed $n$, the function $u_n$ is smooth.  Indeed, the
	whole-space resolvent construction above has smooth coefficients
	and a smooth source, so the same hypoellipticity argument used in
	the proof of \cref{thm:main} applies to the adjoint equation.
	The extension of $F$ by zero is smooth because its support lies
	strictly between $s_0$ and $T$.  Thus $u_n$ is smooth on a neighborhood
	of $[s_0,T]\times\R^{2d}$ and is bounded by the PDE estimates.
	In particular, its derivatives are bounded on every compact
	space--time cylinder, with bounds that may depend on $n$.
	
	Consider now $\mathbb Q_{s_0,z_0}\in\mathcal Q_{s_0,z_0}$ for which, since $u_n$ is smooth, 
	\[
	u_n(t,Z_t)-	u_n(s_0,z_0)-\int_{s_0}^t(\partial_r+V_r\cdot\nabla_x+\triangle_v-b\cdot\nabla_v)u_n(r,Z_r)\,dr,\qquad s_0\le t\le T,
	\]
	is a martingale. As such, define the bounded stopping time 
	\[
	\tau_R:=T\wedge\inf\{r\ge s_0:\abs{Z_r}\ge R\}.
	\]
	Owing to \eqref{eq:terminal-approximation}, it follows that 
	\begin{equation}\label{eq:uniqueness-error}
	u_n(s_0,z_0)=\mathbb E^{\mathbb Q_{s_0,z_0}}\Big[u_n(\tau_R,Z_{\tau_R})+\int_{s_0}^{\tau_R}F(r,Z_r)\,dr\Big]-\mathbb E^{\mathbb Q_{s_0,z_0}}\Big[\int_{s_0}^{\tau_R}((b_n-b)\cdot\nabla_vu_n)(r,Z_r)\,dr\Big].
	\end{equation}
	Let us control the last ``error'' term. Recall from \eqref{eq:regularization-scalar-cutoff} and
	\cref{lem:simultaneous-regularization} that the approximating sequence $\{b_n\}_n$ used in
	Step~1 has the form
	$
	b_n=\rho_{\eps_n}*(a_nb),
	$
	with $a_n$ the $b$-dependent truncation given by \eqref{eq:regularization-scalar-cutoff}. Thus
	\[
	b_n-b=(b_n-a_nb)+(a_n-1)b,
	\]
	and we can estimate these two contributions separately and show both tend to $0$ as $n\rightarrow \infty$.
	
	As a preliminary, for $p$ and $p_1$ as above, put
	\begin{equation}\label{eq:holder-a}
	a:=\big(p^{-1}-p_1^{-1}\big)^{-1}\Leftrightarrow	\frac1a=\frac1p-\frac1{p_1}.
	\end{equation}
	Since $p_1>p$, $a$ is finite.  Moreover,
	$\frac ap=\frac{p_1}{p_1-p}<q$
	by \eqref{eq:sde-exponents}.
	In the original diagonal choice of the mollification scales
	$\eps_n$ (see \eqref{eq:diagonal-local-errors} in the proof of Proposition \ref{prop:regularization}), we have imposed the additional requirement
	\begin{equation}\label{eq:mollification-error}
		\norm{\rho_{\eps_n}*(a_nb)-a_nb}_a\le2^{-n}.
	\end{equation}
	This is possible because, for each fixed $n$, the vector field
	$a_nb$ is bounded and compactly supported, and hence belongs to
	$L^a$. In turn, by the standard properties of the mollifiers,
	\[
	\norm{\rho_\eps*(a_nb)-a_nb}_a\rightarrow 0
	\qquad\text{as }\eps\downarrow0.
	\]
	Thus, \eqref{eq:mollification-error} can be imposed together
	with all the earlier requirements in the same diagonal construction.
	The functions $u_n$ are defined using this fixed sequence.
	
	We first estimate the truncation contribution $(a_n-1)b$.  Fix $R$ and write
	\[
	K_R=[s_0,T]\times\{\abs z\le R\},
	\qquad E_n=K_R\cap\{\abs b>n\}.
	\]
	For all sufficiently large $n$, the spatial--temporal cutoff
	$\chi_n$ in \eqref{eq:regularization-scalar-cutoff} equals one
	on $K_R$.  Therefore, on this cylinder,
	\[
	a_n=\min\left\{1,\frac{n}{\abs b}\right\},
	\qquad
	\abs{(a_n-1)b}
	=(\abs b-n)_+
	\le\abs b\,\mathbf1_{E_n},
	\]
	with the usual convention that $a_n=1$ on $\{b=0\}$.
	Using this inequality, the weighted
	Krylov estimate \eqref{eq:uniqueness-krylov-b} ensures that 
	\[
	\begin{aligned}
		&\left|
		\mathbb E^{\mathbb Q_{s_0,z_0}}\left[
		\int_{s_0}^{\tau_R}
		(a_n-1)b(r,Z_r)\cdot\nabla_vu_n(r,Z_r)\,dr
		\right]\right|\\
		&\quad\le
		\mathbb E^{\mathbb Q_{s_0,z_0}}\int_{s_0}^T
		\mathbf1_{E_n}(r,Z_r)\abs{b(r,Z_r)}
		\abs{\nabla_vu_n(r,Z_r)}\,dr\\
		&\quad\le
		 K_{\mathbb Q}
		\norm{\mathbf1_{E_n}\abs b^{1/p}\nabla_vu_n}_p.
	\end{aligned}
	\]
	H\"older's inequality with the expression \eqref{eq:holder-a} of $a$ gives
	\begin{equation*}
		\norm{\mathbf1_{E_n}\abs b^{1/p}\nabla_vu_n}_p
		\le
		\left(\int_{E_n}\abs b^{a/p}\,dr\,d\zeta\right)^{1/a}
		\norm{\nabla_vu_n}_{L^{p_1}(K_R)}
		\rightarrow 0.
	\end{equation*}
	Indeed, the second factor is uniformly bounded by
	\eqref{eq:un-convergence}, while
	\[
	\int_{E_n}\abs b^{a/p}\,dr\,d\zeta
	\le n^{a/p-q}\int_{K_R}\abs b^q\,dr\,d\zeta
	\rightarrow 0,
	\]
	because $b\in L^q_{\mathrm{loc}}$ and $a/p<q$.
	This proves that the truncation contribution tends to zero
	for every fixed $R$.
	
	For the mollification contribution $b_n-a_nb$, apply the unweighted Krylov
	estimate \eqref{eq:uniqueness-krylov-one}, followed by H\"older's
	inequality:
	\begin{align*}
		\left|
		\mathbb E^{\mathbb Q_{s_0,z_0}}\left[
		\int_{s_0}^{\tau_R}
		(b_n-a_nb)(r,Z_r)\cdot\nabla_vu_n(r,Z_r)\,dr
		\right]\right|
		&\le
		 K_{\mathbb Q}
		\norm{\mathbf1_{K_R}\abs{b_n-a_nb}
			\abs{\nabla_vu_n}}_p\\
		&\le
		 K_{\mathbb Q}
		\norm{b_n-a_nb}_a
		\norm{\nabla_vu_n}_{L^{p_1}(K_R)}\\
		&\le
		2^{-n} K_{\mathbb Q}
		C_T\norm F_{p_1}
		\rightarrow 0,
	\end{align*}
	using, for the last inequality, \eqref{eq:mollification-error} and
	\eqref{eq:un-convergence}.

	Combining the two bounds proves, for every fixed $R$,
	\begin{equation}\label{error-bound}
	\lim_n\left|
	\mathbb E^{\mathbb Q_{s_0,z_0}}\left[
	\int_{s_0}^{\tau_R}
	(b_n-b)(r,Z_r)\cdot\nabla_vu_n(r,Z_r)\,dr
	\right]\right|= 0.
	\end{equation}
	Therefore, letting $n\to\infty$ in the stopped identity \eqref{eq:uniqueness-error}, using the uniform
	convergence in \eqref{eq:un-convergence}, gives
	\[
	u_F(s_0,z_0)-\mathbb E^{\mathbb Q_{s_0,z_0}}\left[u_F(\tau_R,Z_{\tau_R})\right]=\mathbb E^{\mathbb Q_{s_0,z_0}}\left[
	\int_{s_0}^{\tau_R}F(r,Z_r)\,dr\right].
	\]
	For every continuous path on $[s_0,T]$, $\tau_R\rightarrow T$ as $R\rightarrow \infty$, and the boundedness of $u_F$ and $F$, together with $u_F(T,\cdot)=0$ yields the claim  \eqref{eq:conditional-occupation-duality}.

\smallskip
\noindent
\emph{Step 4: The Feller and Markov properties.}  We write, for all bounded Borel function $g$,
\begin{equation*}
 P^{s,t}g(z):=\mathbb E_{s,z}g(Z_t),
 \qquad 0\le s\le t\le T.
\end{equation*}
The Markov property of the weak solution $Z$ of \eqref{eq:kinetic-sde} is a direct consequence of Steps 1. and 2. yielding the wellposedness of the martingale problem for $b$ at initial time-state $(s,z)$ (see  \cite[Theorem 10.1.1]{StrookVaradhan-06}).

We can highlight that, if $(s_n,z_n)\rightarrow(s,z)$, the estimates in Step~1 are uniform for initial
points in compact sets and for starting times in $[0,T]$.  Every limit point
is therefore identified by Step~2 as a weak solution started from $(s,z)$
and satisfies the same Krylov estimates.  Conditional uniqueness gives
$\mathbb P_{s_n,z_n}\Rightarrow\mathbb P_{s,z}$, so
$P^{s,t}g\in C_b$ for $g\in C_\infty$.

Since each $\mathbb P_{s,z}$ is a probability measure on
$\Omega_{s,T}$, its transition kernel is conservative:
\begin{equation*}
 P^{s,t}\mathbf1(z)
 =\mathbb P_{s,z}(Z_t\in\R^{2d})=1,
 \qquad 0\le s\le t\le T.
\end{equation*}

We now prove strong continuity of $P^{s,t}$ in $C_\infty(\mathbb R^{2d})$. Fix
$g\in C_\infty(\R^{2d})$ and write $h=t-s$. Set
$\Phi_h(x,v)=(x+hv,v)$. The integral form of the SDE \eqref{eq:kinetic-sde} gives
\[
 Z_t=\Phi_h z+R_{s,t},
\]
where the velocity and position components of $R_{s,t}$ are
\[
 \begin{aligned}
 R^v_{s,t}
 &=-\int_s^t b(r,Z_r)\,dr+\sqrt2(W_t-W_s),\\
 R^x_{s,t}
 &=-\int_s^t(t-r)b(r,Z_r)\,dr
   +\sqrt2\int_s^t(W_r-W_s)\,dr.
 \end{aligned}
\]
From our PDE theory, uniformly in $(s,z)$,
\begin{equation*}
 \mathbb E_{s,z}\int_s^t\abs{b(r,Z_r)}\,dr
 \le C_T^{1/(1+\varepsilon_*)}
       \abs{t-s}^{\varepsilon_*/(1+\varepsilon_*)}\rightarrow0
 \quad\text{as }t-s\downarrow0.
\end{equation*}
This yields
\begin{equation*}
 \sup_z\mathbb E_{s,z}|R_{s,t}|
 \le C_T\left(
 h^{\varepsilon_*/(1+\varepsilon_*)}+h^{1/2}
 \right),
 \qquad 0\le s<t\le T.
\end{equation*}

Every function in $C_\infty(\R^{2d})$ is uniformly continuous.
Let
\[
 \omega_g(\delta)
 :=\sup\{|g(y)-g(y')|:\ |y-y'|\le\delta\}
\]
be the modulus of continuity of $g$. Splitting according to whether
$|R_{s,t}|\le\delta$ or not, and applying Markov's inequality, we obtain
\[
 \begin{aligned}
 |P^{s,t}g(z)-g(\Phi_hz)| & =\big| \mathbb E_{s,z} g(\Phi_h z + R_{s,t}) - g(\Phi_h z)\big| \\[2mm]
 &\le\omega_g(\delta)
   +2\|g\|_\infty
       \mathbb P_{s,z}(|R_{s,t}|>\delta)\\
 &\le\omega_g(\delta)
   +\frac{2\|g\|_\infty}{\delta}
       \mathbb E_{s,z}|R_{s,t}|.
 \end{aligned}
\]
Taking the supremum over $s$ and $z$, then letting $h\downarrow0$
and afterwards $\delta\downarrow0$, proves
\[
 \sup_{0\le s\le T-h}
 \|P^{s,s+h}g-g\circ\Phi_h\|_\infty\rightarrow 0.
\]
It remains to show that
$\|g\circ\Phi_h-g\|_\infty\to0$.
Here uniform continuity alone is insufficient, since
$|\Phi_hz-z|=h|v|$ is unbounded as $v$ varies.
Fix $\epsilon>0$ and choose $R$ so that
$|g(y)|<\epsilon$ whenever $|y|>R$.
For $0\le h\le1$, the inverse of $\Phi_h$ has operator norm
at most two, so $|\Phi_hz|\ge|z|/2$.
Thus, when $|z|>2R$,
\[
 |g(\Phi_hz)-g(z)|\le2\epsilon.
\]
When $|z|\le2R$, we instead use
$|\Phi_hz-z|\le2Rh$ to obtain
\[
 |g(\Phi_hz)-g(z)|\le\omega_g(2Rh).
\]
These two estimates prove the required convergence. Combining
it with the preceding estimate gives
\begin{equation*}
 \lim_{h\downarrow0}
 \sup_{0\le s\le T-h}
 \|P^{s,s+h}g-g\|_\infty=0.
\end{equation*}

Let us verify joint strong continuity in
the two time variables. Suppose
$(s_n,t_n)\to(s,t)$ with $s_n\le t_n$ and $s\le t$.
The weak continuity of the solution laws established above
implies that, whenever $z_n\rightarrow z$,
\[
 P^{s_n,t_n}g(z_n)\rightarrow  P^{s,t}g(z).
\]
To interpret this statement on a common path space, extend
each path by a constant before its starting time so that it is still in
$C([0,T];\R^{2d})$. The assertion then follows because
$(\omega,r)\mapsto\omega(r)$ is continuous for the uniform
topology on paths. This sequential statement implies
\[
 \sup_{|z|\le R}
 |P^{s_n,t_n}g(z)-P^{s,t}g(z)|\rightarrow 0
 \qquad\text{for every }R<\infty.
\]
Indeed, otherwise one could choose a sequence of points in
the closed ball witnessing failure of convergence and then
pass to a convergent subsequence, contradicting the preceding
statement and continuity of $P^{s,t}g$.

To pass from compact sets to the whole space, we use the
uniform version of the vanishing-at-infinity estimate.
Since
\[
 |\Phi_{t-s}z|\ge\frac{|z|}{1+T},
 \qquad
 \sup_{0\le s\le t\le T}\sup_z
       \mathbb E_{s,z}|R_{s,t}|\le C_T,
\]
we have, for $|z|>(1+T)R$,
\[
 |P^{s,t}g(z)|
 \le
 \sup_{|y|>R}|g(y)|
 +\frac{C_T\|g\|_\infty}{|z|/(1+T)-R}.
\]
Here the second term bounds the probability that $Z_t$ lies
in the ball of radius $R$: on this event,
$|R_{s,t}|\ge|\Phi_{t-s}z|-R$.
First choosing $R$ large and then $|z|$ large shows that
$P^{s,t}g(z)\to0$ uniformly over $0\le s\le t\le T$.
Together with convergence on compact sets, this proves
\[
 \|P^{s_n,t_n}g-P^{s,t}g\|_\infty\rightarrow 0,
\]
which is joint strong continuity.

Thus the transition operators form a strongly continuous
Feller propagator on $C_\infty(\R^{2d})$.

\begin{remark}
Another argument that could be applied above consists of adjoining time to the state space and using the standard result that weak continuity of semigroups implies strong continuity.
\end{remark}

\smallskip
\noindent
\emph{Step 5: The probability density of $Z_t$.} Fix
\begin{equation}\label{eq:density-exponent}
 1<r<\frac{Q-1}{Q-2}
\end{equation}
from Theorem \ref{thm:main}(iii) and let $\eta_{s,z}^{\lambda}$ denote the 
 solution of the point-source PDE:
\begin{equation*}
 \A\eta-\divv(b\eta)=\delta_{\{(s,z)\}}.
\end{equation*}
Our goal is to prove
\begin{equation}\label{eq:density-identification}
 e^{-\lambda(t-s)}\mathbb P_{s,z}(Z_t\in d\zeta)\,dt
 =\eta_{s,z}^{\lambda}(t,\zeta)\,dt\,d\zeta.
\end{equation}
For smooth $b_n$, the discounted space--time transition density from $(s,z)$
is the point-source solution of
\[
 \A\eta_n-\divv(b_n\eta_n)=\delta_{\{(s,z)\}}.
\]
At the exponent $r$ in \eqref{eq:density-exponent},
Theorem \ref{thm:main}(ii)--(iii) gives
$\eta_n\rightarrow\eta_{s,z}^{\lambda}$ in $\W_{\K}^{1/r,r}$.  On the other hand,
weak convergence $\mathbb P^n_{s,z}\Rightarrow \mathbb P_{s,z}$ of the path laws following Step 1, identifies the distributional limit of
$\eta_n\,dt\,d\zeta$ with
$e^{-\lambda(t-s)}\mathbb P_{s,z}(Z_t\in d\zeta)\,dt$.  The limits agree,
which proves \eqref{eq:density-identification} and yields (iv). \hfill \qed

\begin{remark}
We could give an alternative proof of Theorem \ref{thm:sde} that reverses the order of
the preceding proof: first one constructs a conservative Feller propagator
directly from the backward PDE, and only then identifies the corresponding probability measures, forming a strong Markov family parameterized by the initial time and the initial point, as weak solutions of
\eqref{eq:kinetic-sde}. Our proof in \cite{KiY} has this structure. This construction does not use
tightness of the laws of the smooth SDEs and does not require conditional uniqueness in the Krylov class. The same applies to the proof of Theorem \ref{thm:sde-local} that we give next.
\end{remark}

\bigskip

\section{Removal of global integrability condition using weights (Proof of Theorem \ref{thm:sde-local})}
\label{sec:local-removal}

We now remove the global \(L^{1+\varepsilon_0}\) assumption from Theorem \ref{thm:sde} and restore the bounded component of the drift permitted by the PDE
theory. Specifically, throughout this section we replace the global hypothesis
\eqref{eq:HK1} by the original condition \eqref{eq:HK}:
\begin{equation}\label{eq:HK1-local}
 b=b_{\mathrm s}+b_{\mathrm b},
 \qquad b_{\mathrm s}\in
 \M^q_{\K}(\R^{1+2d}, \R^d),
 \qquad b_{\mathrm b}\in L^\infty(\R^{1+2d}, \R^d),
 \qquad q>1,
\end{equation}
where no global Lebesgue $L^{1+\epsilon}$-integrability is imposed on $b_{\mathrm s}$.  The
Morrey norm of $b_{\mathrm s}$ is small, whereas the size of
$b_{\mathrm b}$ is arbitrary and is absorbed by taking $\lambda$ large.
All backward and forward smallness conditions on the denominator of the Duhamel series used below are assumed strict.

The localization used to handle \eqref{eq:HK} will be carried out using a vanishing at infinity weight, centered at
the point $\zeta_0$ from which the diffusion $Z_t=(X_t,V_t)$ in \eqref{eq:kinetic-sde} is started. Namely, define
\begin{equation*}
 \kappa>Q,
 \qquad
 \rho_{\zeta_0}(\zeta):=
       \bigl(1+d_*(\zeta,\zeta_0)\bigr)^{-\kappa},
 \qquad \zeta_0,\zeta\in\R^{1+2d},
\end{equation*}
and write
\[
 \norm f_{L^p_{\rho_{\zeta_0}}}
 :=\left(\int_{\R^{1+2d}}\abs{f(\zeta)}^p
                  \rho_{\zeta_0}(\zeta)\,d\zeta\right)^{1/p}.
\]
The centering at $\zeta_0$ is essential: $\rho_{\zeta_0}(\zeta_0)=1$, so a weighted
estimate followed by evaluation at $\zeta_0$ is uniform in the starting
point. Note also that the triangle inequality for $d_*$ gives
\begin{equation}\label{eq:weight-moderate}
 \frac{\rho_{\zeta_0}(\zeta)}{\rho_{\zeta_0}(\xi)}
 \le C_{d,\kappa}\bigl(1+d_*(\zeta,\xi)\bigr)^\kappa,
\end{equation}
uniformly in $\zeta_0$.  Moreover, by \eqref{eq:ball-volume} and
$\kappa>Q$,
\begin{equation}\label{eq:weight-integrable}
 \sup_{\zeta_0}\int\rho_{\zeta_0}(\zeta)\,d\zeta<\infty.
\end{equation}

\begin{lemma}[Centered weighted backward estimate]
\label{lem:weighted-backward}
Let $p>Q-1$, $q>1$, and let
$b=b_{\mathrm s}+b_{\mathrm b}$ satisfy \eqref{eq:HK1-local}.  For
$\lambda\ge1$, define
\begin{align}
 \mathfrak r_{p,\lambda}^{\rho}(b)
 &:=C^{R,\mathrm s}_{d,p,q,\kappa}
       \norm{b_{\mathrm s}}_{\M^q_{\K}}^{1/p}
   +C^{R,\mathrm b}_{d,p,\kappa}\lambda^{-1/(2p)}
       \norm{b_{\mathrm b}}_\infty^{1/p},
 \label{eq:weighted-r-bound}\\
 \mathfrak q_{p,\lambda}^{\rho}(b)
 &:=C^{Q,\mathrm s}_{d,p,q,\kappa}
       \norm{b_{\mathrm s}}_{\M^q_{\K}}^{1/p'}
   +C^{Q,\mathrm b}_{d,p,\kappa}\lambda^{-1/(2p')}
       \norm{b_{\mathrm b}}_\infty^{1/p'}.
 \label{eq:weighted-q-bound}
\end{align}
The operators in \eqref{eq:adjoint-backward-factors} extend to
$L^p_{\rho_{\zeta_0}}$, uniformly in the center $\zeta_0$, and satisfy
\begin{align}
 \norm{\widehat R_p(b)}_{L^p_{\rho_{\zeta_0}}\rightarrow
       L^p_{\rho_{\zeta_0}}}
 &\le \mathfrak r_{p,\lambda}^{\rho}(b),
 \label{eq:weighted-R}\\
 \norm{\widehat Q_p(b)}_{L^p_{\rho_{\zeta_0}}\rightarrow
       L^p_{\rho_{\zeta_0}}}
 &\le \mathfrak q_{p,\lambda}^{\rho}(b).
 \label{eq:weighted-Q}
\end{align}
Moreover,
\begin{equation}\label{eq:weighted-point-evaluation}
 \abs{(\A^*)^{-1/2-1/(2p)}f(\zeta_0)}
 \le C_{d,p,\lambda,\kappa}\norm f_{L^p_{\rho_{\zeta_0}}}.
\end{equation}
Therefore, if
\begin{equation}\label{eq:weighted-smallness}
 \Theta_{p,\lambda}^{\rho}(b)
 :=\mathfrak q_{p,\lambda}^{\rho}(b)
   \mathfrak r_{p,\lambda}^{\rho}(b)<1,
\end{equation}
then $I+\widehat Q_p(b)\widehat R_p(b)$ is invertible on
$L^p_{\rho_{\zeta_0}}$, with inverse norm bounded uniformly in
$\zeta_0$.
\end{lemma}

\begin{proof}
Put
\[
 M:=\norm{b_{\mathrm s}}_{\M^q_{\K}},
 \qquad B:=\norm{b_{\mathrm b}}_\infty.
\]
All constants  in the estimates below are independent of the center $\zeta_0$.

\smallskip
\noindent\emph{Step 1: Off-diagonal kernel bounds.}
Let $K_\gamma^*$ be the kernel of $(\A^*)^{-\gamma/2}$ (see \eqref{eq:fractional-kernel-formula} in Appendix \ref{prop:kernels_proof}) and let
$D_\gamma^*$ be the kernel of
$\nabla_v(\A^*)^{-(1+\gamma)/2}$.  The explicit formula used in the proof of
Proposition \ref{prop:kernels} (see again Appendix \ref{prop:kernels_proof}) shows, on the time half-space on which either kernel is
nonzero, that
\begin{equation}\label{eq:weighted-gaussian-kernels}
 \abs{K_\gamma^*(\zeta,\xi)}
 +\abs{D_\gamma^*(\zeta,\xi)}
 \le C_{d,\gamma}\tau^{(\gamma-Q)/2}
       e^{-\lambda\tau}e^{-c_d\Phi_\tau(\zeta,\xi)}.
\end{equation}
Here $\tau$ is the absolute time difference and $\Phi_\tau$ is the sum of
the two normalized quadratic expressions appearing in
\eqref{eq:gaussian-upper}.  Notice that the powers of $\tau$ are the same
for the two kernels: differentiation contributes $\tau^{-1/2}$, while the
fractional power in $D_\gamma^*$ is higher by one.

Set $r=d_*(\zeta,\xi)$.  By the equivalence of $d$ and $d_*$, either
$\tau\ge c r^2$, or the Gaussian exponent satisfies
$\Phi_\tau(\zeta,\xi)\ge c r^2/\tau$.  In the first case the factor
$e^{-\lambda\tau}$ is bounded by $e^{-c r^2}$ because $\lambda\ge1$.  In
the second case,
\[
 e^{-\lambda\tau-c r^2/\tau}\le e^{-c' r}
\]
by the arithmetic--geometric mean inequality.  It follows from
\eqref{eq:weighted-gaussian-kernels} that, for every $L>0$,
\begin{equation}\label{eq:weighted-off-diagonal}
 \abs{K_\gamma^*(\zeta,\xi)}
 +\abs{D_\gamma^*(\zeta,\xi)}
 \le C_{d,\gamma,L}
 r^{\gamma-Q}(1+r)^{-L},
 \qquad r\ge1.
\end{equation}
Together with the near-diagonal estimates of Proposition \ref{prop:kernels}, this
implies that, for every $\theta\ge0$ and $0<\gamma<Q$,
\begin{equation}\label{eq:weighted-kernel-majorization}
 \bigl(1+d_*(\zeta,\xi)\bigr)^\theta
 \bigl(\abs{K_\gamma^*(\zeta,\xi)}
       +\abs{D_\gamma^*(\zeta,\xi)}\bigr)
 \le C_{d,\gamma,\theta}
       d_*(\zeta,\xi)^{\gamma-Q}.
\end{equation}
Indeed, the polynomial factor is bounded near the diagonal and is absorbed
by the arbitrary power $(1+r)^{-L}$ in
\eqref{eq:weighted-off-diagonal} away from it.  The same statements hold
for the non-adjoint kernels.  We will use
\eqref{eq:weighted-kernel-majorization} with
$\gamma=1/p$ and $1/p'$; the off-diagonal form
\eqref{eq:weighted-off-diagonal} with $\gamma=1+1/p$ will be used for point
evaluation.

The constants in \eqref{eq:weighted-off-diagonal} and
\eqref{eq:weighted-kernel-majorization} can be chosen uniformly for
$\lambda\ge1$: for these bounds one uses $e^{-\lambda\tau}\le e^{-\tau}$
throughout.  This fixes the singular-part constants independently of
$\lambda$.  We will also retain the $\lambda$-dependence needed for the
bounded component.  Upon writing
\[
\mathcal E_\gamma^*(\zeta,\xi)
:=\abs{K_\gamma^*(\zeta,\xi)}
+\abs{D_\gamma^*(\zeta,\xi)},
\]
Gaussian integration and the kinetic scaling give, for every
$\theta\ge0$ and $0<\gamma<Q$, 
\begin{equation}\label{eq:weighted-kernel-L1}
 \begin{split}
 &\sup_\zeta\int
  (1+d_*(\zeta,\xi))^\theta
  \mathcal E_\gamma^*(\zeta,\xi)\,d\xi\\
 &\quad+\sup_\xi\int
  (1+d_*(\zeta,\xi))^\theta
  \mathcal E_\gamma^*(\zeta,\xi)\,d\zeta
 \le C_{d,\gamma,\theta}
       \bigl(\lambda^{-\gamma/2}
             +\lambda^{-(\gamma+\theta)/2}\bigr)
 \le C_{d,\gamma,\theta}\lambda^{-\gamma/2}.
 \end{split}
\end{equation}
Indeed, after integration of the normalized Gaussian in the space
variables, the terms without and with the distance moment are bounded by
constant multiples of
\[
 \int_0^\infty \tau^{\gamma/2-1}e^{-\lambda\tau}\,d\tau
 \quad\hbox{and}\quad
 \int_0^\infty \tau^{(\gamma+\theta)/2-1}
          e^{-\lambda\tau}\,d\tau,
\]
respectively.  The same computation applies with the variables reversed.

\smallskip
\noindent\emph{Step 2: Conjugation by the weight.}
Put
\[
 \omega_{\zeta_0}:=\rho_{\zeta_0}^{1/p},
 \qquad
 (U_{\zeta_0}f)(\zeta):=\omega_{\zeta_0}(\zeta)f(\zeta).
\]
Then
\[
 U_{\zeta_0}:L^p_{\rho_{\zeta_0}}\rightarrow L^p
\]
is an isometric isomorphism.  Thus the weighted operator norm of an
operator \(T\) is exactly the ordinary \(L^p\)-operator norm of
\(U_{\zeta_0}TU_{\zeta_0}^{-1}\).  From
\eqref{eq:weight-moderate},
\begin{equation}\label{eq:omega-ratio}
 \frac{\omega_{\zeta_0}(\zeta)}
      {\omega_{\zeta_0}(\xi)}
 \le C_{d,\kappa,p}
       \bigl(1+d_*(\zeta,\xi)\bigr)^{\kappa/p},
\end{equation}
uniformly in \(\zeta_0\).

We first treat \(\widehat R_p(b)\).  Since
\[
 \abs b^{1/p}\le\abs{b_{\mathrm s}}^{1/p}+B^{1/p},
\]
its conjugated kernel is bounded by the sum of a singular part and a bounded
part.  For the singular part, \eqref{eq:omega-ratio} and
\eqref{eq:weighted-kernel-majorization}, with
\(\theta=\kappa/p\) and \(\gamma=1/p\), give
\begin{equation*}
 \abs{U_{\zeta_0}\widehat R_p(b)U_{\zeta_0}^{-1}h}
 \le C\abs{b_{\mathrm s}}^{1/p}I_{1/p}(\abs h)
     +B^{1/p}\mathcal D_{p,\lambda}^{\rho}(\abs h),
\end{equation*}
where \(\mathcal D_{p,\lambda}^{\rho}\) is the positive integral operator
with kernel
\[
 C(1+d_*(\zeta,\xi))^{\kappa/p}
       \abs{D_{1/p}^*(\zeta,\xi)}.
\]
The Adams estimate \eqref{eq:adams-operator} controls the first term, while
\eqref{eq:weighted-kernel-L1} and the Schur test give
\[
 \norm{\mathcal D_{p,\lambda}^{\rho}}_{p\rightarrow p}
 \le C_{d,p,\kappa}\lambda^{-1/(2p)}.
\]
Hence
\[
 \norm{U_{\zeta_0}\widehat R_p(b)U_{\zeta_0}^{-1}}_{p\rightarrow p}
 \le C M^{1/p}
      +C\lambda^{-1/(2p)}B^{1/p}
 =\mathfrak r_{p,\lambda}^{\rho}(b),
\]
after fixing the constants in \eqref{eq:weighted-r-bound}.  This is
\eqref{eq:weighted-R}.

For \(\widehat Q_p(b)\), use
\[
 \abs b^{1/p'}
 \le\abs{b_{\mathrm s}}^{1/p'}+B^{1/p'}.
\]
The singular part of the conjugated kernel satisfies
\begin{equation*}
 \abs{U_{\zeta_0}\widehat Q_p(b)U_{\zeta_0}^{-1}h}
 \le CI_{1/p'}
       \bigl(\abs{b_{\mathrm s}}^{1/p'}\abs h\bigr)
     +B^{1/p'}\mathcal K_{p,\lambda}^{\rho}(\abs h),
\end{equation*}
where \(\mathcal K_{p,\lambda}^{\rho}\) has kernel
\[
 C(1+d_*(\zeta,\xi))^{\kappa/p}
       \abs{K_{1/p'}^*(\zeta,\xi)}.
\]
Since \(I_{1/p'}\) has a symmetric kernel, duality and
\eqref{eq:adams-operator}, used with exponent \(p'\), yield
\[
 \begin{split}
 \norm{I_{1/p'}
       (\abs{b_{\mathrm s}}^{1/p'}\abs h)}_p
 &=\sup_{\norm\varphi_{p'}\le1}
   \int\abs{b_{\mathrm s}}^{1/p'}\abs h\,
        I_{1/p'}(\abs\varphi)\\
 &\le\norm h_p\sup_{\norm\varphi_{p'}\le1}
       \norm{\abs{b_{\mathrm s}}^{1/p'}
              I_{1/p'}(\abs\varphi)}_{p'}\\
 &\le C M^{1/p'}\norm h_p.
 \end{split}
\]
The Schur test and \eqref{eq:weighted-kernel-L1}, now with
\(\gamma=1/p'\), give
\[
 \norm{\mathcal K_{p,\lambda}^{\rho}}_{p\rightarrow p}
 \le C_{d,p,\kappa}\lambda^{-1/(2p')}.
\]
It follows that
\[
 \norm{U_{\zeta_0}\widehat Q_p(b)U_{\zeta_0}^{-1}}_{p\rightarrow p}
 \le C M^{1/p'}
      +C\lambda^{-1/(2p')}B^{1/p'}
 =\mathfrak q_{p,\lambda}^{\rho}(b),
\]
which proves \eqref{eq:weighted-Q}.  Initially the calculation is made on
compactly supported functions.  Since \(C_c^\infty\) is dense in
\(L^p_{\rho_{\zeta_0}}\), the two conjugated estimates give the asserted
extensions.

\smallskip
\noindent\emph{Step 3: Evaluation at the center and inversion.}
By \eqref{eq:fractional-kernel-bound}, the kernel of
\((\A^*)^{-1/2-1/(2p)}\) is bounded near the diagonal by
\(Cd_*(\zeta,\xi)^{1+1/p-Q}\).  Its \(p'\)-th power is locally integrable if
and only if
\begin{equation}\label{eq:local-kernel-integrability}
 \left(1+\frac1p-Q\right)p'+Q>0,
 \qquad\text{equivalently}\qquad
 1+\frac1p>\frac Qp.
\end{equation}
Condition \eqref{eq:local-kernel-integrability} is precisely \(p>Q-1\).

Let \(G^*(\zeta_0,\xi)\) be the kernel of
\((\A^*)^{-1/2-1/(2p)}\).  H\"older's inequality with the weight gives
\begin{equation*}
 \abs{(\A^*)^{-1/2-1/(2p)}f(\zeta_0)}
 \le\norm f_{L^p_{\rho_{\zeta_0}}}
 \left(\int\abs{G^*(\zeta_0,\xi)}^{p'}
       \rho_{\zeta_0}(\xi)^{-p'/p}\,d\xi\right)^{1/p'}.
\end{equation*}
On $B_1(\zeta_0)$ the last integral is finite by
\eqref{eq:local-kernel-integrability}.  On its complement,
\eqref{eq:weighted-off-diagonal} and
$\rho_{\zeta_0}^{-p'/p}\le(1+d_*(\zeta_0,\cdot))^{\kappa p'/p}$
make it finite after taking $L$ sufficiently large.  This proves
\eqref{eq:weighted-point-evaluation}.

Finally, $L^p_{\rho_{\zeta_0}}$ is a Banach space, and
\eqref{eq:weighted-R}--\eqref{eq:weighted-Q} imply
\[
 \norm{\widehat Q_p(b)\widehat R_p(b)}_{
       L^p_{\rho_{\zeta_0}}\rightarrow L^p_{\rho_{\zeta_0}}}
 \le\mathfrak q_{p,\lambda}^{\rho}(b)
       \mathfrak r_{p,\lambda}^{\rho}(b)
 =\Theta_{p,\lambda}^{\rho}(b)<1.
\]
Therefore, the series
\[
 \bigl(I+\widehat Q_p(b)\widehat R_p(b)\bigr)^{-1}
 =\sum_{n=0}^\infty
   \bigl(-\widehat Q_p(b)\widehat R_p(b)\bigr)^n
\]
converges in operator norm on $L^p_{\rho_{\zeta_0}}$, uniformly in the
center.  This completes the proof.
\end{proof}

The weighted space does not impose an additional assumption on the drift.

In the next lemma the choice $\kappa>Q$ in the definition of the weight is used in an essential manner.

\begin{lemma}[Centered weighted moments of the full drift]
\label{lem:weighted-morrey-integrability}
Let $b=b_{\mathrm s}+b_{\mathrm b}$ satisfy
\eqref{eq:HK1-local}, 
$0<a\le q$, and let $\kappa>Q$.  Then
\begin{equation*}
 \sup_{\zeta_0\in\R^{1+2d}}
 \int_{\R^{1+2d}}\abs{b(\zeta)}^a
       \rho_{\zeta_0}(\zeta)\,d\zeta
 \le C_{d,q,a,\kappa}
       \left(\norm{b_{\mathrm s}}_{\M^q_{\K}}^a
             +\norm{b_{\mathrm b}}_\infty^a\right).
\end{equation*}
\end{lemma}

\begin{proof}
Put $M=\norm{b_{\mathrm s}}_{\M^q_{\K}}$ and
$B=\norm{b_{\mathrm b}}_\infty$.  From the definition of the Morrey norm and
\eqref{eq:ball-volume},
\[
 \int_{B_R(\zeta_0)}\abs{b_{\mathrm s}}^q
 \le C_dM^qR^{Q-q}.
\]
Hence H\"older's inequality gives, for every $R>0$,
\begin{equation}\label{eq:morrey-ball-lower-power}
 \int_{B_R(\zeta_0)}\abs{b_{\mathrm s}}^a
 \le |B_R(\zeta_0)|^{1-a/q}
      \left(\int_{B_R(\zeta_0)}\abs{b_{\mathrm s}}^q\right)^{a/q}
 \le C_{d,q,a}M^aR^{Q-a}.
\end{equation}
Now set $A_0=B_1(\zeta_0)$ and
\[
 A_{k+1}=B_{2^{k+1}}(\zeta_0)\setminus B_{2^k}(\zeta_0),
 \qquad k\ge0.
\]
On $A_{k+1}$ one has $\rho_{\zeta_0}\le2^{-k\kappa}$, and therefore
\eqref{eq:morrey-ball-lower-power} with $R=2^{k+1}$ yields
\[
 \int_{A_{k+1}}\abs{b_{\mathrm s}}^a\rho_{\zeta_0}
 \le C M^a2^{-k\kappa}2^{k(Q-a)}
 =CM^a2^{-k(\kappa-Q+a)}.
\]
The contribution of $A_0$ is bounded by
\eqref{eq:morrey-ball-lower-power} with $R=1$, and the displayed geometric
series converges because $\kappa>Q$.  Thus
\[
 \sup_{\zeta_0}\int\abs{b_{\mathrm s}}^a\rho_{\zeta_0}
 \le CM^a.
\]
Moreover, \eqref{eq:weight-integrable} gives
\[
 \sup_{\zeta_0}\int\abs{b_{\mathrm b}}^a\rho_{\zeta_0}
 \le C B^a.
\]
The result follows from
$\abs{b_{\mathrm s}+b_{\mathrm b}}^a
 \le C_a(\abs{b_{\mathrm s}}^a+\abs{b_{\mathrm b}}^a)$.
\end{proof}

We call $p$ \textit{SDE admissible} if it is backward
admissible in the sense of \eqref{eq:backward-smallness} and also satisfies
$\Theta_{p,\lambda}^{\rho}(b)<1$ in
\eqref{eq:weighted-smallness}.  By
\eqref{eq:weighted-r-bound} and \eqref{eq:weighted-q-bound}, this condition
holds after increasing $\lambda$, provided that the Morrey norm of
$b_{\mathrm s}$ satisfies the corresponding strict smallness condition.
No smallness is imposed on $\norm{b_{\mathrm b}}_\infty$.

\begin{lemma}[Higher-power Krylov estimate]
\label{lem:higher-power-krylov}
Let $p>Q-1$ be SDE admissible and choose
\begin{equation*}
 0<\varepsilon<\min\left\{1,\frac{q-1}{p}\right\}.
\end{equation*}
Let $b_n=b_{n,\mathrm s}+b_{n,\mathrm b}$ be the 
regularization from
Proposition \ref{lem:simultaneous-regularization}, chosen with a uniform strict bound in
\eqref{eq:weighted-smallness}.  Then
\begin{equation*}
 \sup_n\sup_{0\le s<T}\sup_{z\in\R^{2d}}
 \mathbb E^n_{s,z}\int_s^T
       \abs{b_n(r,Z^n_r)}^{1+\varepsilon}\,dr
 \le C_T<\infty.
\end{equation*}
\end{lemma}

\begin{proof}
Fix $n,s,T,z$ and put $\zeta_0=(s,z)$.  Set
$F_n=\mathbf1_{[s,T]}\abs{b_n}^{1+\varepsilon}$ and define the vector fields
\[
 \ell_n:=\mathbf1_{[s,T]}\sgn(b_n)
              \abs{b_n}^{1/p+\varepsilon},
 \qquad
 \sgn(b_n):=\frac{b_n}{\abs{b_n}}\mathbf1_{\{b_n\ne0\}},
\]
so that $F_n=b^{\frac{1}{p'}}_n\cdot\ell_n$.  Since
$1+p\varepsilon<q$, Lemma \ref{lem:weighted-morrey-integrability}, with
$a=1+p\varepsilon$, gives
\begin{equation}\label{eq:ell-weighted}
 \norm{\ell_n}_{L^p_{\rho_{\zeta_0}}}^p
 =\int_{\{s\le t\le T\}}\abs{b_n}^{1+p\varepsilon}
       \rho_{\zeta_0}
 \le C_{d,p,q,\varepsilon,\kappa}
       \left(\norm{b_{n,\mathrm s}}_{\M^q_{\K}}^{1+p\varepsilon}
             +\norm{b_{n,\mathrm b}}_\infty^{1+p\varepsilon}\right)
 \le C.
\end{equation}
The centered weighted backward factorization is therefore well-defined and
gives
\begin{equation*}
 u_n=(\A^*)^{-1/2-1/(2p)}
 \bigl(I+\widehat Q_p(b_n)\widehat R_p(b_n)\bigr)^{-1}
 \widehat Q_p(b_n)\ell_n.
\end{equation*}
Because $b_n$ and $\ell_n$ are compactly supported, the same Duhamel
partial sums occur in the global $L^p$ factorization from Theorem
\ref{thm:backward}.  They converge both in global $L^p$ and in
$L^p_{\rho_{\zeta_0}}$. Since both convergences imply local $L^p$
convergence, the two limits coincide.  Thus $u_n$ is the usual smooth
resolvent solution.  By
\eqref{eq:weighted-Q}, \eqref{eq:weighted-point-evaluation},
\eqref{eq:weighted-smallness}, and \eqref{eq:ell-weighted},
\begin{equation}\label{eq:power-backward-point}
 \abs{u_n(s,z)}
 \le C\mathfrak q_{p,\lambda}^{\rho}(b_n)
       \norm{\ell_n}_{L^p_{\rho_{\zeta_0}}}
 \le C.
\end{equation}
The adjoint resolvent is backward causal, so $u_n$ vanishes for $t>T$.
Apply It\^o's formula to $u_n(t,Z^n_t)$, first with a smooth approximation
of $F_n$.  The resolvent parameter gives the ``discount'' factor
$e^{-\lambda(r-s)}$.  Positivity of $F_n$, \eqref{eq:power-backward-point},
and removal of the discount on $[s,T]$ yield
\[
 \mathbb E^n_{s,z}\int_s^T
       \abs{b_n(r,Z^n_r)}^{1+\varepsilon}\,dr
 \le e^{\lambda T}\abs{u_n(s,z)}\le C_T.
\]
The constant is uniform in $n,s,z$. Applying Fatou's lemma
ends the proof.
\end{proof}

We are now in a position to 
demonstrate Theorem \ref{thm:sde-local}.

\begin{proof}[Proof of Theorem \ref{thm:sde-local}]
Choose, simultaneously at \(p,p_1\)
satisfying $p>Q-1$, $p_1>pq/(q-1)$, and, when necessary, \(r\), the
single sequence
\[
 b_n=b_{n,\mathrm s}+b_{n,\mathrm b}
\]
from Proposition \ref{lem:simultaneous-regularization}. Both  approximating  components use the same
cutoff index and the same convolution scale.  In view of
\eqref{eq:local-regularization} and
\eqref{eq:weighted-r-bound}--\eqref{eq:weighted-q-bound}, the sequence may
be chosen so that
\begin{equation*}
 \sup_n\Theta_{p,\lambda}^{\rho}(b_n)<1.
\end{equation*}
Here one first fixes the strict smallness margin for
\(\norm{b_{\mathrm s}}_{\M^q_{\K}}\) and then takes \(\lambda\) large
enough to absorb the common \(L^\infty\) bound of
\(b_{n,\mathrm b}\).

Starting from the diffusion $Z^n=(X^n,V^n)$ with drift $b_n=b_{n,\mathrm s}+b_{n,\mathrm b}$, we can extend each main step in the preceding proof of Theorem \ref{thm:sde}.

\noindent
\emph{Tightness}. Apply Lemma \ref{lem:higher-power-krylov} with
\(\varepsilon=\varepsilon_*\) and its conclusion is exactly the control
\eqref{eq:uniform-power-drift}. Subsequently, the pathwise Hölder estimate
\eqref{eq:drift-holder-modulus}, the tail bounds
\eqref{eq:Bn-tightness-conditions}, and the control \eqref{eq:compact-containment} 
in Step~1 of Theorem \ref{thm:sde} apply verbatim.  Thus, again up to a subsequence $\mathbb P^n_{s,z}:=\mathbb P\circ(Z^n)^{-1}$ 
are tight and converge to a limit $\mathbb P_{s,z}$. 

\smallskip
\noindent\emph{Occupation measures, passage to a weak solution and Krylov bounds.}
The proof of the bounds \eqref{first-Krylov-bound} and \eqref{second-Krylov-bound} proceeds as in parts $1$ and $2$ in Step 2 of the proof of Theorem \ref{thm:sde}. Regarding the control $\mathbb E_{s,z}\int_s^T|b(r,Z_r)|^{1+\varepsilon_*}\,dr$, we verify explicitly how the two components $b_{\mathrm s}$, $b_{\mathrm b}$ enter the arguments of part 3 of Step~2.
Let \(\chi\in C_c^\infty(\R^{1+2d})\), \(0\le\chi\le1\),
be one on the space--time support of the derivative of a martingale test
function, and set
\[
 k_{n,\mathrm s}:=\chi(b_{n,\mathrm s}-b_{\mathrm s}),
 \qquad
 h_{n,\mathrm b}:=\chi(b_{n,\mathrm b}-b_{\mathrm b}).
\]
Then
\[
 \sup_n\norm{k_{n,\mathrm s}}_{\M^q_{\K}}<\infty,
 \qquad \norm{k_{n,\mathrm s}}_1\rightarrow0,
 \qquad \norm{h_{n,\mathrm b}}_p\rightarrow0.
\]
The first two assertions follow from the Morrey bound and local
\(L^q\)-convergence; the last follows from the local \(L^p\)-convergence of
the bounded components.  Applying respectively the weighted and unweighted
Krylov estimates \eqref{eq:krylov-smooth} and \eqref{eq:krylov-one} gives
\begin{equation*}
 \begin{split}
 &\mathbb E^n_{s,z}\int_s^T
   \abs{b_n-b}(r,Z_r)
   \abs{\nabla_v\varphi(Z_r
   	)}\,dr\\
 &\quad\le C_{T,\varphi}\left(
   \norm{k_{n,\mathrm s}}_{\M^q_{\K}}^{1/p'}
       \norm{k_{n,\mathrm s}}_1^{1/p}
   +\norm{h_{n,\mathrm b}}_p\right)
 \rightarrow0.
 \end{split}
\end{equation*}
Thus the varying coefficient causes no new difficulty.

For the fixed coefficient, repeat the weighted occupation-density argument
of Step~2 with \(k=\chi b_{\mathrm s}\).  The bounded component is simpler:
the signed occupation densities are bounded in \(L^{p'}\) by
\eqref{eq:krylov-one}, while
\(\chi b_{\mathrm b}\cdot\nabla_v\varphi\in L^p\); hence ordinary
\(L^p\)-\(L^{p'}\) duality passes this term to the limit.  The same
decomposition proves, for every
\(\psi\in C_c(\R^{1+2d};\R^d)\),
\[
 \int b_n\cdot\psi\,d\mu_n\rightarrow
 \int b\cdot\psi\,d\mu.
\]
Consequently the argument
\eqref{eq:local-drift-duality-convergence},
\eqref{eq:power-lower-semicontinuity} applies to the full drift and gives
\[
 \mathbb E_{s,z}\int_s^T\abs{b(r,Z_r)}^{1+\varepsilon_*}\,dr
 \le C_T.
\]
Owing to this estimate, \eqref{eq:direct-coefficient-error} and all the arguments of  Step~2.4 hold and ensure that $\mathbb P_{s,z}$ is a solution to the martingale problem for $b$.

\smallskip
\noindent\emph{Conditional uniqueness.}
Only the error term \eqref{error-bound} in Step~3 of Theorem \ref{thm:sde} needs comments.
Let
\[
 \frac1a=\frac1p-\frac1{p_1};
\]
then \(a/p<q\).  On the compact cylinder determined by the stopping time,
write
\[
 \widetilde b_{n,\mathrm s}:=a_n b_{\mathrm s},
 \qquad
 b_{n,\mathrm s}=\rho_{\eps_n}*\widetilde b_{n,\mathrm s}.
\]
Here $a_n$ is the common scalar cutoff of
\eqref{eq:regularization-scalar-cutoff}.
On the fixed cylinder, $\chi_n=1$ for all sufficiently large $n$, and
\[
 |b_{\mathrm s}-\widetilde b_{n,\mathrm s}|
 \le |b_{\mathrm s}|\mathbf1_{\{|b|>n\}}
 \le (|b|+\|b_{\mathrm b}\|_\infty)\mathbf1_{\{|b|>n\}}.
\]
Set $E_n=\{|b|>n\}$ on this cylinder.  These sets decrease to a null set.
The two conditional Krylov bounds in Theorem \ref{thm:sde}(iii) give
\begin{equation}\label{eq:singular-tail-with-bounded-part}
 \begin{split}
 &\mathbb E^{\mathbb Q}_{s,z}\int_s^T
   \abs{b_{\mathrm s}-\widetilde b_{n,\mathrm s}}(r,Z_r)
       \abs{\nabla_vu_n(r,Z_r)}\,dr\\
 &\quad\le C_{\mathbb Q}
  \left(\norm{\abs b^{1/p}\mathbf1_{E_n}\nabla_vu_n}_p
  +\norm{b_{\mathrm b}}_\infty
       \norm{\mathbf1_{E_n}\nabla_vu_n}_p\right)\\
 &\quad\le C_{\mathbb Q}\norm{\nabla_vu_n}_{p_1}
  \left(\norm{\abs b^{1/p}\mathbf1_{E_n}}_a
  +\norm{b_{\mathrm b}}_\infty
       \norm{\mathbf1_{E_n}}_a\right)\rightarrow0.
 \end{split}
\end{equation}
All norms in this display are restricted to the fixed cylinder.  The last
limit follows from \(a/p<q\),
\(b=b_{\mathrm s}+b_{\mathrm b}\in L^{a/p}_{\mathrm{loc}}\), and absolute
continuity of the integral.

The single scale \(\eps_n\) in
Proposition \ref{lem:simultaneous-regularization} can simultaneously be chosen so that
\[
 \norm{\widetilde b_{n,\mathrm s}
       -\rho_{\eps_n}*\widetilde b_{n,\mathrm s}}_a\le2^{-n}
\]
on the \(n\)th compact cylinder, without changing any factor convergence.
The remaining singular error is therefore controlled by the unweighted
conditional Krylov estimate and the uniform \(L^{p_1}\)-bound on
\(\nabla_vu_n\).  For the bounded component the same estimate gives,
on every fixed cylinder,
\begin{equation}\label{eq:bounded-component-uniqueness-error}
 \begin{split}
 &\mathbb E^{\mathbb Q}_{s,z}\int_s^T
   \abs{b_{n,\mathrm b}-b_{\mathrm b}}(r,Z_r)
       \abs{\nabla_vu_n(r,Z_r)}\,dr\\
 &\qquad\le C_{\mathbb Q}
   \norm{b_{n,\mathrm b}-b_{\mathrm b}}_a
   \norm{\nabla_vu_n}_{p_1}\rightarrow0
 \end{split}
\end{equation}
by \eqref{eq:local-regularization}. 
It follows from \eqref{eq:singular-tail-with-bounded-part} and
\eqref{eq:bounded-component-uniqueness-error} that the error term 
 tends to zero.  The same estimates hold after
multiplication by any bounded $\mathcal F_{t_0}$-measurable factor $G$,
with the right-hand sides multiplied by $\|G\|_\infty$.  Consequently
they prove the representation \eqref{eq:conditional-occupation-duality}, and by extension uniqueness of the martingale problem for the full drift $b=b_{\mathrm s}+b_{\mathrm b}$.

\noindent
\emph{The Feller and Markov properties}. Finally, the Feller and strong Markov arguments use only the uniform
higher-power estimate for the full drift, and the density identification is
Theorem \ref{thm:main}(ii)--(iii) for the decomposition
\(b=b_{\mathrm s}+b_{\mathrm b}\).  Steps~4 and 5 of
Theorem \ref{thm:sde} therefore apply without further changes and complete the
proof.
\end{proof}

\appendix

\bigskip

\section{Fractional powers of linear operators}\label{app:fractional-powers}

We recall the construction of the fractional powers of operators that we use throughout the paper. 

We follow Komatsu \cite[Sections~7 and~11]{Komatsu1966}.
Specifically, see Propositions~11.1 and~11.4 there, and
Theorem~11.8 for the definition of positive fractional powers via Bochner-Phillips subordination. (Regarding the definition via subordination, see also
the exposition in Yosida~\cite[Chapter~IX, \S 11]{Yosida}.)

\textbf{1.~}For $\lambda>0$ and $s\in\R$, recall that the Fourier transform definition of the fractional power of the (shifted) Laplacian on $L^2=L^2(\R^d)$ is
\begin{equation}\label{eq:app-fp-fourier}
	\begin{split}
		\mathscr F\bigl[(\lambda-\Delta)^sf\bigr](\xi)
		&=(\lambda+|\xi|^2)^s\widehat f(\xi),\\
		D((\lambda-\Delta)^s)
		&=\{f\in L^2:(\lambda+|\xi|^2)^s\widehat f\in L^2\},
	\end{split}
\end{equation}
where $\mathscr F$ is the unitary Fourier transform normalized by
$\mathscr F(-\Delta f)(\xi)=|\xi|^2\widehat f(\xi)$.

\textbf{2.~}The semigroup formulation below also works in $L^p$, and agrees with the Fourier transform definition.

Let $1<p<\infty$, and let $-B$ generate a strongly continuous semigroup
$T(t)$ on $L^p=L^p(E
,\mu)$, where $(E
,\mu)$ is a fixed
$\sigma$-finite measure space, such that
$\|T(t)\|_{p\to p}\le M e^{-\omega t}$ for some $M\ge1$ and $\omega>0$.
For every $\beta>0$, the negative (fractional) power $B^{-\beta}$ of $B$ is the bounded operator
\begin{equation}\label{eq:app-fp-negative}
	B^{-\beta}f:=\frac1{\Gamma(\beta)}
	\int_0^\infty t^{\beta-1}T(t)f\,dt,
	\qquad \|B^{-\beta}\|_{p\to p}\le M\omega^{-\beta},
\end{equation}
where the integral converges in $L^p$. Using the identity
\begin{equation}\label{eq:app-fp-beta-integral}
	\int_0^1 u^{\beta-1}(1-u)^{\delta-1}\,du
	=\frac{\Gamma(\beta)\Gamma(\delta)}{\Gamma(\beta+\delta)},
	\qquad \beta,\delta>0,
\end{equation}
the semigroup property $T(t)T(s)=T(t+s)$, Fubini's theorem and the change of variables
$r=t+s$, $u=t/(t+s)$, we obtain
$B^{-\beta}B^{-\delta}=B^{-(\beta+\delta)}$.

\begin{remark}[Injectivity and dense range]\label{rem:app-fp-injectivity}
	The Laplace-resolvent formula identifies $B^{-1}$ with the inverse of $B$.
	To see that $B^{-\beta}$ is injective, choose an integer $m>\beta$.
	Then $B^{-\beta}f=0$ implies
	$B^{-m}f=B^{-(m-\beta)}B^{-\beta}f=0$, hence $f=0$, since
	$B^{-m}=(B^{-1})^m$ is injective. Moreover,
	$\operatorname{Ran}(B^{-\beta})$ contains
	$\operatorname{Ran}(B^{-m})=D(B^m)$, which is dense by the standard
	resolvent approximation.
\end{remark}

We therefore define the positive power as the closed operator
$B^\beta=(B^{-\beta})^{-1}$, with
$D(B^\beta)=\operatorname{Ran}(B^{-\beta})$.
Thus $B^\beta B^{-\beta}=I$ on $L^p$ and
$B^{-\beta}B^\beta=I$ on $D(B^\beta)$.
For $0<\alpha<1$, the scalar identity
\begin{equation}\label{eq:app-fp-scalar-positive}
	r^\alpha=\frac{\alpha}{\Gamma(1-\alpha)}
	\int_0^\infty(1-e^{-tr})\,\frac{dt}{t^{1+\alpha}},
	\qquad r\ge0,
\end{equation}
motivates the corresponding semigroup representation: $B^\alpha$ is
equivalently the closure of
\begin{equation}\label{eq:app-fp-positive}
	B^\alpha f=\frac{\alpha}{\Gamma(1-\alpha)}
	\int_0^\infty\bigl(f-T(t)f\bigr)\,\frac{dt}{t^{1+\alpha}},
	\qquad f\in D(B).
\end{equation}
In particular, $D(B)$ is a core for $B^\alpha$. The positive powers satisfy
$B^\beta B^\delta=B^{\beta+\delta}$ on the natural composition domain.

B\"ochner subordination further shows that $-B^\alpha$, $0<\alpha<1$,
generates a strongly continuous semigroup and represents it as an average
of $T$. Namely, let $\eta_{\alpha,t}$ be the probability measures on
$[0,\infty)$ characterized by
$\int_{[0,\infty)}e^{-\rho r}\eta_{\alpha,t}(dr)=e^{-t\rho^\alpha}$,
$\rho\ge0$. Bochner subordination gives
\begin{equation}\label{eq:app-fp-subordination}
	e^{-tB^\alpha}f
	=\int_{[0,\infty)}T(r)f\,\eta_{\alpha,t}(dr),
	\qquad
	\|e^{-tB^\alpha}\|_{p\to p}\le M e^{-t\omega^\alpha}.
\end{equation}
The representation also preserves positivity whenever $T$ is positive.

If $B^*$ is the Banach
adjoint on $L^{p'}$, $p'=p/(p-1)$, its powers satisfy
$(B^{-\beta})^*=(B^*)^{-\beta}$ and
$(B^\beta)^*=(B^*)^\beta$.
Semigroups $T_p$ and $T_q$ on $L^p(\Omega,\mu)$ and $L^q(\Omega,\mu)$
are called \emph{consistent} if
$T_p(t)f=T_q(t)f$ almost everywhere for every $t\ge0$ and
$f\in L^p\cap L^q$. For such semigroups satisfying the preceding bounds,
the negative powers agree on $L^p\cap L^q$, and the positive powers agree
on their common domains.

\medskip

\textbf{3.~}We now specify the previous construction to the kinetic operator. Recall the operators introduced in Section~\ref{sec:geometry}:
\begin{equation}\label{eq:app-fp-kinetic-operator}
	\K=\partial_t+v\cdot\nabla_x-\Delta_v,
	\qquad \A=\lambda+\K,\quad \lambda>0,
	\qquad \A^*=\lambda-\partial_t-v\cdot\nabla_x-\Delta_v.
\end{equation}
The realization of $-\K$ on $L^p(\R^{1+2d})$ is the generator of the
strongly continuous contraction semigroup
\begin{equation}\label{eq:app-fp-kinetic-semigroup}
	(\mathsf P_\tau f)(t,x,v)
	=\int_{\R^{2d}}p_\tau(x,v;y,w)f(t-\tau,y,w)\,dy\,dw,
\end{equation}
where $p_\tau$ is the kernel in~\eqref{eq:heat-kernel}; thus
\eqref{eq:app-fp-kinetic-semigroup} is~\eqref{eq:spacetime-semigroup}.
Consequently
$e^{-\tau\A}=e^{-\lambda\tau}\mathsf P_\tau$, and the preceding statements
apply with $M=1$ and $\omega=\lambda$. We suppress the $L^p$ subscript on
the realizations of $\A$ and $\A^*$; in adjoint identities they act on the
dual exponents $p$ and $p'$. For $\gamma>0$,
\begin{equation}\label{eq:app-fp-kinetic-negative}
	\A^{-\gamma/2}f
	=\frac1{\Gamma(\gamma/2)}\int_0^\infty
	\tau^{\gamma/2-1}e^{-\lambda\tau}\mathsf P_\tau f\,d\tau,
	\qquad \|\A^{-\gamma/2}\|_{p\to p}\le\lambda^{-\gamma/2}.
\end{equation}
This recovers~\eqref{eq:fractional-power} and the corresponding bound
in~\eqref{eq:power-Lp}. The same formula and bound hold for $\A^*$ with $\mathsf P_\tau^*$.
In particular,
$\langle\A^{-\gamma/2}f,g\rangle
=\langle f,(\A^*)^{-\gamma/2}g\rangle$
for $f\in L^p$ and $g\in L^{p'}$. For positive power, as previously discussed, $\A^{\gamma/2}=(\A^{-\gamma/2})^{-1}$, $\gamma>0$.

The following lemma is needed to establish the compatibility with test functions used
in the weak formulation~\eqref{eq:weak-distribution} and in
Lemma~\ref{lem:factorization}.

\begin{lemma}[Compatibility with test functions]\label{lem:app-fp-test-functions}
	Let $\lambda>0$. For every $s\in\R$, the operators $\A^s$ and $(\A^*)^s$
	restrict to continuous automorphisms of $\cS(\R^{1+2d})$, with inverses
	$\A^{-s}$ and $(\A^*)^{-s}$, respectively. Their extensions to $\cS'$ by
	transposition satisfy
	\begin{equation}\label{eq:app-fp-transposition}
		\langle\A^su,\varphi\rangle=\langle u,(\A^*)^s\varphi\rangle,
		\qquad \A^s\A^r=\A^{s+r}\quad\text{on }\cS',\qquad r,s\in\R,
	\end{equation}
	and the analogous identities hold with $\A$ replaced by $\A^*$.
\end{lemma}

\begin{proof}
	Both $\mathsf P_\tau$ and $\mathsf P_\tau^*$ act continuously
	on $\cS(\R^{1+2d})$, and the Gaussian formula gives, for every Schwartz
	seminorm $q$, a seminorm $\widetilde q$ and constants $C,N$ such that
	\begin{equation}\label{eq:app-fp-schwartz}
		q(\mathsf P_\tau\varphi)+q(\mathsf P_\tau^*\varphi)
		\le C(1+\tau)^N\widetilde q(\varphi),\qquad \tau\ge0.
	\end{equation}
	Indeed, differentiation produces polynomial factors in $\tau$, while the
	polynomial weights are controlled by the corresponding Gaussian moments.
	The factor $e^{-\lambda\tau}$ therefore makes
	\eqref{eq:app-fp-kinetic-negative} convergent in $\cS$.
	For $s>0$, choosing an integer $m>s$ and using
	$\A^s=\A^m\A^{-(m-s)}$ shows that positive powers also preserve $\cS$.
	Together with the inverse identities, this proves the automorphism
	assertions on $\cS$. Transposition gives the extensions to $\cS'$.
	Then the power identities follow from those on $\cS$ for the adjoint powers.
\end{proof}

For the Bessel-potential spaces of Section~\ref{sec:spaces}, this gives,
for $\gamma>0$,
$\W_{\K}^{\gamma,p}=D(\A^{\gamma/2})=\A^{-\gamma/2}L^p$
with norm $\|\A^{\gamma/2}u\|_p$. With $\W_{\K}^{0,p}=L^p$ and the
negative-order spaces defined by completion in the norms
$\|\A^{\gamma/2}\cdot\|_p$, one obtains, for every $\gamma\in\R$,
\begin{equation}\label{eq:app-fp-scale}
	\A:\W_{\K}^{\gamma,p}\longrightarrow\W_{\K}^{\gamma-2,p}
	\quad\text{isometrically onto}
	\qquad
	(\W_{\K}^{\gamma,p})^*=\W_{\K^*}^{-\gamma,p'}.
\end{equation}
The duality is the distributional pairing in~\eqref{eq:Bessel-duality}.
These identities justify the fractional factorizations in
Lemma~\ref{lem:factorization} and Section~\ref{sec:main-results}.

\section{Proof of Proposition \ref{prop:kernels}}
\label{prop:kernels_proof}

\noindent
\emph{Step 1: The kernel of the fractional resolvent.}
Let $\zeta=(t,x,v)$ and $\xi=(s,y,w)$.  Formula
\eqref{eq:spacetime-semigroup} shows that the space--time kernel of
$\A^{-\gamma/2}$ vanishes for $t\le s$.  If $t>s$ and
$\tau:=t-s$, the integration variable in \eqref{eq:fractional-power} is
forced to equal $t-s$ by the time translation in
\eqref{eq:spacetime-semigroup}.  Consequently 
\[
\mathscr{A}^{-\gamma/2} f(\zeta)
=
\int_{\mathbb{R}^{1+2d}}
K_{\gamma,\lambda}(\zeta,\xi) f(\xi)\,d\xi,
\]
where the integral kernel is
\begin{equation}\label{eq:fractional-kernel-formula}
 K_{\gamma,\lambda}(\zeta,\xi)
 =\frac{\mathbf1_{\{t>s\}}}{\Gamma(\gamma/2)}
   \tau^{\gamma/2-1}e^{-\lambda\tau}
   p_\tau(x,v;y,w).
\end{equation}
For $t>s$, the definition \eqref{eq:d} reads
\begin{equation}\label{eq:distance-positive-time}
 R:=d(\zeta,\xi)
 =\max\left\{\sqrt\tau,\ \abs{w-v},\
       \abs{y-x+\tau v}^{1/3}\right\}.
\end{equation}
In particular, $R/\sqrt\tau\ge1$.  From
\eqref{eq:gaussian-upper} and $Q=4d+2$ we obtain
\begin{equation}\label{eq:kernel-before-distance}
 \begin{split}
 \abs{K_{\gamma,\lambda}(\zeta,\xi)}
 &\le C\tau^{\gamma/2-1-2d}
 \exp\left[-c\left(
  \frac{\abs{w-v}^2}{\tau}
 +\frac{\abs{y-x+\tau v}^2}{\tau^3}\right)\right]\\
 &=C\tau^{(\gamma-Q)/2}
 \exp\left[-c\left(
  \frac{\abs{w-v}^2}{\tau}
 +\frac{\abs{y-x+\tau v}^2}{\tau^3}\right)\right].
 \end{split}
\end{equation}
Here and below we have used $e^{-\lambda\tau}\le1$, which is why the
pointwise constants do not depend on $\lambda$.

We now compare the Gaussian exponent with the intrinsic distance.  If the
maximum in \eqref{eq:distance-positive-time} is attained by the velocity
term, then
$\abs{w-v}^2/\tau=R^2/\tau$.  If it is attained by the position term, then
\[
 \frac{\abs{y-x+\tau v}^2}{\tau^3}
 =\left(\frac{R}{\sqrt\tau}\right)^6
 \ge \frac{R^2}{\tau}.
\]
If it is attained by $\sqrt\tau$, then $R^2/\tau=1$.  After changing the
constants, all three cases are therefore summarized by
\begin{equation}\label{eq:gaussian-distance-reduction}
 \exp\left[-c\left(
  \frac{\abs{w-v}^2}{\tau}
 +\frac{\abs{y-x+\tau v}^2}{\tau^3}\right)\right]
 \le C\exp\left(-c'\frac{R^2}{\tau}\right).
\end{equation}
Set $u=R/\sqrt\tau\ge1$.  Combining
\eqref{eq:kernel-before-distance} and
\eqref{eq:gaussian-distance-reduction} gives
\begin{equation}\label{eq:kernel-scaling-reduction}
 \abs{K_{\gamma,\lambda}(\zeta,\xi)}
 \le C R^{\gamma-Q}u^{Q-\gamma}e^{-c'u^2}
 \le C_{d,\gamma}R^{\gamma-Q},
\end{equation}
because $Q-\gamma>0$ and
$\sup_{u\ge1}u^{Q-\gamma}e^{-c'u^2}<\infty$.  Finally,
$d\le d_*\le(1+2^{1/3})d$ by
\eqref{eq:reverse-comparable}--\eqref{eq:dstar}.  Since
$\gamma-Q<0$, replacing $R=d(\zeta,\xi)$ by the equivalent metric $d_*$
changes only the constant.  This proves the estimate for
$\A^{-\gamma/2}$ in \eqref{eq:fractional-kernel-bound}.

The kernel of $(\A^*)^{-\gamma/2}$ is
$K_{\gamma,\lambda}^*(\zeta,\xi)
=K_{\gamma,\lambda}(\xi,\zeta)$.  Since $d_*$ is symmetric, the same bound
proves the adjoint part of \eqref{eq:fractional-kernel-bound}.

\smallskip
\noindent
\emph{Step 2: Differentiated kernels.}
Differentiating the explicit quadratic form in
\eqref{eq:heat-kernel} shows that
\begin{equation}\label{eq:two-velocity-derivatives}
 \abs{\nabla_v p_\tau(x,v;y,w)}
 +\abs{\nabla_w p_\tau(x,v;y,w)}
 \le C\tau^{-1/2}\overline p_\tau(x,v;y,w),
\end{equation}
where $\overline p_\tau$ satisfies the same estimate as
\eqref{eq:gaussian-upper}, with a smaller positive constant in the
exponential.  To see the power $\tau^{-1/2}$ directly, differentiation of
the exponent produces terms bounded by
\[
 C\left(
  \frac{\abs{w-v}}{\tau}
 +\frac{\abs{y-x+\tau v}}{\tau^2}
 \right)p_\tau.
\]
Each normalized linear factor can be absorbed into a slightly wider
Gaussian by the elementary inequality
$a e^{-ca^2}\le C_c e^{-ca^2/2}$.  The derivative with respect to $w$ is
handled in exactly the same way.  This proves
\eqref{eq:two-velocity-derivatives}.

For test functions, differentiation under the integral is justified by
\eqref{eq:two-velocity-derivatives}.  The kernel of
$\nabla_v\A^{-(1+\gamma)/2}$ is therefore bounded by
\[
 C\tau^{(1+\gamma)/2-1}\tau^{-1/2}\tau^{-2d}
 \exp\left[-c\left(
  \frac{\abs{w-v}^2}{\tau}
 +\frac{\abs{y-x+\tau v}^2}{\tau^3}\right)\right].
\]
The power of $\tau$ is again
$(1+\gamma)/2-1-1/2-2d=(\gamma-Q)/2$.  Thus the argument in
\eqref{eq:gaussian-distance-reduction}--
\eqref{eq:kernel-scaling-reduction} applies without change and proves the
first term in \eqref{eq:gradient-fractional-kernel-bound}.  The adjoint
kernel uses the $w$-derivative in
\eqref{eq:two-velocity-derivatives}; symmetry of $d_*$ then gives the second
term.

\smallskip
\noindent
\emph{Step 3: $L^r$ bounds.}
Both $\mathsf P_\tau$ and its adjoint are contractions on $L^r$.  Hence
\[
 \begin{split}
 \norm{\A^{-\gamma/2}f}_r
 &\le\frac{\norm f_r}{\Gamma(\gamma/2)}
   \int_0^\infty\tau^{\gamma/2-1}e^{-\lambda\tau}\,d\tau
 =\lambda^{-\gamma/2}\norm f_r.
 \end{split}
\]
The identical calculation for $\A^*$ proves
\eqref{eq:power-Lp}.

The broadened Gaussian in \eqref{eq:two-velocity-derivatives} has uniformly
bounded mass. Therefore,
\begin{equation}\label{eq:gradient-semigroup-Lr}
 \norm{\nabla_v\mathsf P_\tau f}_r
 +\norm{\nabla_v\mathsf P_\tau^*f}_r
 \le C_{d,r}\tau^{-1/2}\norm f_r.
\end{equation}
Initially this is proved for smooth compactly supported functions and extended to $L^r$ using a standard density argument.  Using
\eqref{eq:gradient-semigroup-Lr} in the fractional power formula gives
\[
 \begin{split}
 \norm{\nabla_v\A^{-(1+\gamma)/2}f}_r
 &\le
 \frac{C_{d,r}}{\Gamma((1+\gamma)/2)}
 \int_0^\infty
 \tau^{(1+\gamma)/2-1}\tau^{-1/2}e^{-\lambda\tau}\,d\tau\,\norm f_r\\
 &=C_{d,r,\gamma}\lambda^{-\gamma/2}\norm f_r.
 \end{split}
\]
The same calculation for $\A^*$ completes the proof of
\eqref{eq:gradient-power-Lp}. \hfill \qed

\bigskip

\section{Proof of Proposition \ref{prop:sobolev}}
\label{prop:sobolev_proof}

We give the proof for $\A$ and then indicate why it also applies to
$\A^*$.  We will only need the doubling property, the volume estimate
\eqref{eq:ball-volume} and the kernel estimates of
Proposition \ref{prop:kernels}.

\smallskip
\noindent
\emph{Step 1: Hardy--Littlewood--Sobolev estimate.}
Let $\mathsf M$ be the centered Hardy--Littlewood maximal operator associated
with the metric $d_*$ and Lebesgue measure.  Since the space is doubling,
$\mathsf M$ is bounded on $L^r$ for every $1<r<\infty$, see, for example,
\cite[Chapter~2]{Heinonen}.  We recall the
standard argument that converts this maximal estimate into the fractional
integration estimate needed here.

Fix $\zeta\in\R^{1+2d}$ and $R>0$.  Decompose the integral defining
$I_\gamma f(\zeta)$ into the regions
$d_*(\zeta,\xi)<R$ and $d_*(\zeta,\xi)\ge R$.  A decomposition of the first
region into the annuli
\[
 2^{-k-1}R\le d_*(\zeta,\xi)<2^{-k}R,\qquad k\ge0,
\]
together with \eqref{eq:ball-volume}, gives
\begin{equation}\label{eq:HLS-near}
 \begin{split}
 \int_{\{d_*(\zeta,\xi)<R\}}
 \frac{\abs{f(\xi)}}{d_*(\zeta,\xi)^{Q-\gamma}}\,d\xi
 &\le C\sum_{k=0}^\infty
 (2^{-k}R)^{\gamma-Q}
 \int_{B_{2^{-k}R}(\zeta)}\abs f\,d\xi\\
 &\le C R^\gamma\mathsf Mf(\zeta).
 \end{split}
\end{equation}
For the outer region, H\"older's inequality gives
\[
 \int_{\{d_*(\zeta,\xi)\ge R\}}
 \frac{\abs{f(\xi)}}{d_*(\zeta,\xi)^{Q-\gamma}}\,d\xi
 \le\norm f_r
 \left(\int_{\{d_*(\zeta,\xi)\ge R\}}
 d_*(\zeta,\xi)^{(\gamma-Q)r'}\,d\xi\right)^{1/r'}.
\]
Decomposing into the annuli
$2^kR\le d_*(\zeta,\xi)<2^{k+1}R$ and again using
\eqref{eq:ball-volume}, we find
\begin{equation}\label{eq:HLS-far}
 \int_{\{d_*(\zeta,\xi)\ge R\}}
 \frac{\abs{f(\xi)}}{d_*(\zeta,\xi)^{Q-\gamma}}\,d\xi
 \le C\norm f_r R^{\gamma-Q/r}.
\end{equation}
The geometric series used here converges precisely when
$\gamma<Q/r$.

At points where $\mathsf Mf(\zeta)>0$, choose
\[
 R=\left(\frac{\norm f_r}{\mathsf Mf(\zeta)}\right)^{r/Q}.
\]
The cases $\mathsf Mf(\zeta)=0$ and $f=0$ are immediate.  Equations
\eqref{eq:HLS-near}--\eqref{eq:HLS-far} yield the pointwise Hedberg
inequality
\begin{equation}\label{eq:HLS-Hedberg}
 I_\gamma f(\zeta)
 \le C\norm f_r^{\gamma r/Q}
       \bigl(\mathsf Mf(\zeta)\bigr)^{1-\gamma r/Q}.
\end{equation}
The relation $1/s=1/r-\gamma/Q$ is equivalent to
\[
 s\left(1-\frac{\gamma r}{Q}\right)=r.
\]
Taking the $L^s$ norm in \eqref{eq:HLS-Hedberg} and using the
$L^r$ boundedness of $\mathsf M$ therefore gives
\begin{equation*}
 \norm{I_\gamma f}_s\le C_{d,\gamma,r}\norm f_r.
\end{equation*}
The pointwise estimates \eqref{eq:power-by-I} now imply
\eqref{eq:HLS} for both $\A$ and $\A^*$. 

\smallskip
\noindent
\emph{Step 2: Uniform $L^{r'}$ bound for the kernel.}
Assume now that $\gamma>Q/r$.  We first verify explicitly the kernel
integrability used for point evaluation.  From
\eqref{eq:fractional-kernel-formula}, for fixed
$\zeta=(t,x,v)$,
\[
 \int_{\R^{1+2d}}
 \abs{K_{\gamma,\lambda}(\zeta,\xi)}^{r'}\,d\xi
 =C_{\gamma,r}\int_0^\infty
 \tau^{(\gamma/2-1)r'}e^{-\lambda r'\tau}
 \left(\int_{\R^{2d}}p_\tau(x,v;y,w)^{r'}\,dy\,dw\right)d\tau.
\]
The change of variables
\[
 a=\frac{w-v}{\sqrt\tau},\qquad
 c=\frac{y-x+\tau v}{\tau^{3/2}}
\]
has Jacobian $\tau^{2d}$.  The explicit Gaussian formula
\eqref{eq:heat-kernel} consequently gives
\begin{equation*}
 \int_{\R^{2d}}p_\tau(x,v;y,w)^{r'}\,dy\,dw
 =C_{d,r}\tau^{-2d(r'-1)}.
\end{equation*}
It follows that
\begin{equation*}
 \int\abs{K_{\gamma,\lambda}(\zeta,\xi)}^{r'}\,d\xi
 =C\int_0^\infty
 \tau^{((\gamma-Q)r'+Q)/2-1}e^{-\lambda r'\tau}\,d\tau.
\end{equation*}
The integral is finite near zero exactly when
\[
 (\gamma-Q)r'+Q>0,
\]
which is equivalent to $\gamma>Q/r$; it is finite at infinity because
$\lambda>0$.  The Gamma-integral also shows that
\begin{equation*}
 \sup_{\zeta\in\R^{1+2d}}
 \norm{K_{\gamma,\lambda}(\zeta,\cdot)}_{r'}
 \le C_{d,\gamma,r}
       \lambda^{-(\gamma-Q/r)/2}.
\end{equation*}
The adjoint kernel is obtained by interchanging $\zeta$ and $\xi$, and
Lebesgue measure is both left and right invariant.  Thus the same estimate
holds for its kernel sections.  H\"older's inequality now gives the
$L^r$-to-$L^\infty$ bound
\begin{equation}\label{eq:Bessel-Linfty}
 \norm{\A^{-\gamma/2}f}_\infty
 +\norm{(\A^*)^{-\gamma/2}f}_\infty
 \le C_{d,\gamma,r}\lambda^{-(\gamma-Q/r)/2}\norm f_r.
\end{equation}

\smallskip
\noindent
\emph{Step 3: Continuity.}
The operator $\A^{-\gamma/2}$ is left invariant, so its kernel sections are
left translates of one fixed function in $L^{r'}$.  Left translations are
strongly continuous on $L^{r'}$ for $1\le r'<\infty$: this follows first
for continuous compactly supported functions and then for all $L^{r'}$
functions by density and invariance of Haar measure.  Hence, if
$\zeta_n\rightarrow\zeta$,
\[
 \norm{K_{\gamma,\lambda}(\zeta_n,\cdot)
       -K_{\gamma,\lambda}(\zeta,\cdot)}_{r'}\rightarrow0.
\]
By H\"older's inequality,
\[
 \abs{\A^{-\gamma/2}f(\zeta_n)
       -\A^{-\gamma/2}f(\zeta)}
 \le\norm{K_{\gamma,\lambda}(\zeta_n,\cdot)
       -K_{\gamma,\lambda}(\zeta,\cdot)}_{r'}\norm f_r,
\]
which proves continuity.  The same reasoning applies to $\A^*$.  More explicitly, if
$K(\zeta,\xi)=k(\zeta^{-1}\circ\xi)$, then
$K^*(\zeta,\xi)=k^\vee(\zeta^{-1}\circ\xi)$ with
$k^\vee(\eta)=k(\eta^{-1})$.  Inversion preserves Haar measure on this
unimodular group, so $k^\vee\in L^{r'}$ and its left translates are
strongly continuous as well.

\smallskip
\noindent
\emph{Step 4: Vanishing at infinity.}
Let first $f\in C_c^\infty$ and let $E=\supp f$.  Since $d_*$ is a proper
metric compatible with the Euclidean topology, the triangle inequality
implies
\[
 \inf_{\xi\in E}d_*(\zeta,\xi)\rightarrow\infty
 \qquad\text{as }\zeta\rightarrow\infty.
\]
Because $\gamma<Q$, \eqref{eq:fractional-kernel-bound} gives
\[
 \abs{\A^{-\gamma/2}f(\zeta)}
 \le C\norm f_1
 \left(\inf_{\xi\in E}d_*(\zeta,\xi)\right)^{\gamma-Q}
 \rightarrow0.
\]
The same argument applies to $\A^*$.  For a general $f\in L^r$, choose
$f_n\in C_c^\infty$ with $f_n\rightarrow f$ in $L^r$.  Estimate
\eqref{eq:Bessel-Linfty} shows that the corresponding potentials converge
uniformly.  As a uniform limit of functions in $C_\infty$, each limiting
potential belongs to $C_\infty$.  This proves part~\textup{(ii)}. \hfill \qed

\bigskip

\section{Many-particle and Lebesgue class drifts}

\begin{proposition}[Morrey norm for pairwise interaction kernels]
\label{prop:many-particle}
Let $N\ge2$, $d_0\ge2$, and $d=Nd_0$. Write
$x=(x^1,\ldots,x^N)$ and $v=(v^1,\ldots,v^N)$, with
$x^i,v^i\in\mathbb R^{d_0}$ and suppose that $b(t,x,v)=b(x,v)$ with
\[
 b_i(x,v)
 =\sum_{j\ne i}K_{ij}(x^i-x^j,v^i-v^j),
 \qquad b=(b_1,\ldots,b_N),
\]
where the measurable kernels satisfy, almost everywhere,
\[
 |K_{ij}(\xi,\nu)|
 \le \kappa_{ij}|\xi|^{-1/3}+\gamma_{ij}|\nu|^{-1},
 \qquad \kappa_{ij}, \gamma_{ij}\ge0.
\]
Then, for every $1<q<d_0$,
\[
 \|b\|_{\mathcal M_{\mathcal K}^q}
 \le C_{N,d_0,q}
       \sum_{i\ne j}(\kappa_{ij}+\gamma_{ij}).
\]
If all $\gamma_{ij}$ vanish, the same conclusion holds for
$1<q<3d_0$.
\end{proposition}

\begin{proof}
We first recall the elementary estimate
\[
 \sup_{a\in\mathbb R^{d_0}}
 \int_{|y-a|<R}|y|^{-\beta}\,dy
 \le C_{d_0,\beta}R^{d_0-\beta},
 \qquad 0<\beta<d_0.
\]
Indeed, on the set $\{|y|<R\}$ one integrates in  spherical coordinates.
On its complement the integrand is at most $R^{-\beta}$,
while the integration domain has volume at most $C R^{d_0}$.

Consider a kinetic cylinder $\mathcal C_r:=\mathcal C_r(t_0,x_0,v_0)$  as in \eqref{eq:backward-cylinder}.
Its time interval has length $r^2$, and each time slice is
contained in the product of the sets
\[
 |x^k-x_0^k-(t-t_0)v_0^k|<r^3,
 \qquad |v^k-v_0^k|<r,
 \qquad k=1,\ldots,N.
\]
Fix $i\ne j$. Integrating first in $x^i$, with all other
variables fixed, and applying the preceding estimate with
$\beta=q/3$ gives
\[
 \begin{aligned}
 \int_{\mathcal C_r}|x^i-x^j|^{-q/3}\,dt\,dx\,dv
 &\le C r^2 r^d r^{3(d-d_0)}r^{3d_0-q}\\
 &=C r^{4d+2-q},
 \qquad q<3d_0.
 \end{aligned}
\]
Similarly, integrating first in $v^i$ gives
\[
 \begin{aligned}
 \int_{\mathcal C_r}|v^i-v^j|^{-q}\,dt\,dx\,dv
 &\le C r^2 r^{3d}r^{d-d_0}r^{d_0-q}\\
 &=C r^{4d+2-q},
 \qquad q<d_0.
 \end{aligned}
\]
Both estimates are uniform in the cylinder center $(t_0,x_0,v_0)$:
the elementary integral estimate is uniform under translations,
including the time-dependent spatial shifts $x_0+(t-t_0)v_0$.

Since $|b|\le\sum_{i\ne j}|K_{ij}|$, we have
\[
 \|b\|_{L^q(\mathcal C_r)}
 \le C\sum_{i\ne j}(\kappa_{ij}+\gamma_{ij})
          r^{(4d+2)/q-1}.
\]
Dividing by $|\mathcal C_r|^{1/q}$, multiplying by $r$,
and taking the supremum proves the assertion since
$|\mathcal C_r| =c r^{4d+2}$ and the definitions of the Morrey norm using kinetic cylinders and metric balls are equivalent.
\end{proof}

\begin{proposition}
\label{prop:lebesgue-class}
Let $p_t,p_v,p_x\in(1,\infty]$ satisfy
\[
 \frac{2}{p_t}+\frac{d}{p_v}+\frac{3d}{p_x}=1.
\]
If
\[
 |b|\in L^{p_t}_tL^{p_v}_vL^{p_x}_x:=L^{p_t}\bigl(\mathbb R_t,
       L^{p_v}(\mathbb R_v^d, L^{p_x}(\mathbb R_x^d))\bigr),
\]
then, for every $1<q\le\min\{p_t,p_v,p_x\}$,
\[
 \|b\|_{\mathcal M_{\mathcal K}^q}
 \le C\|b\|_{L^{p_t}_tL^{p_v}_vL^{p_x}_x}.
\]
The same assertion holds with the integration order of the $x$ and $v$
norms reversed.
\end{proposition}

\begin{proof}
	Let $\mathcal C_r$ be as in \eqref{eq:backward-cylinder}.
Applying H\"older's inequality
successively in $x$, $v$, and $t$, we obtain, for $q\in(1,\min\{p_t,p_v,p_x\}]$,
\[
 \|b\|_{L^q(\mathcal C_r)}
 \le
 C r^{3d(1/q-1/p_x)}
   r^{d(1/q-1/p_v)}
   r^{2(1/q-1/p_t)}
   \|b\|_{L^{p_t}_tL^{p_v}_vL^{p_x}_x}.
\]
By the assumed relation between
the exponents, the total power of $r$ is given by
\[
 \frac{4d+2}{q}
 -\frac{3d}{p_x}-\frac{d}{p_v}-\frac{2}{p_t}
 =\frac{4d+2}{q}-1.
\]
Since $|\mathcal C_r| = cr^{4d+2}$, it follows that
\[
 r\left(\frac{1}{|\mathcal C_r|}
              \int_{\mathcal C_r}|b|^q\right)^{1/q}
 \le C\|b\|_{L^{p_t}_tL^{p_v}_vL^{p_x}_x}.
\]
Taking the supremum over all centers and radii proves
the claim. 

One can finally note that, H\"older's inequality (applied to $v$, $x$ and $t$) also yields
\[
r\left(\frac{1}{|\mathcal C_r|}
\int_{\mathcal C_r}|b|^q\right)^{1/q}
\le C\|b\|_{L^{p_t}_tL^{p_x}_xL^{p_v}_v},
\]
and conclude the estimate $\|b\|_{\mathcal M_{\mathcal K}^q}
\le C
\|b\|_{L^{p_t}_tL^{p_x}_xL^{p_v}_v}$.
\end{proof}

\bigskip

\makeatletter
\let\savedtocwrite\@tocwrite
\let\@tocwrite\@gobbletwo
\section*{AI usage}
\let\@tocwrite\savedtocwrite
\makeatother

The authors used OpenAI ChatGPT 5.6 Sol and 6 Astra for light editorial assistance (namely, grammar and misprints checks, fixing typos, provide style suggestions), and also for  light assistance in carrying out additional literature surveys on pages 5 and 6. 
The motivations, key ideas and all results and proof arguments are the sole outcomes of the authors' works.

\end{document}